\documentclass[11pt]{amsart}
\usepackage{latexsym,amscd,amssymb, graphicx, color, amsthm, bm, amsmath, cancel, enumitem}  
\usepackage{tikz}
\usepackage[margin=1in]{geometry}
\usepackage{hyperref}
\usepackage[all,cmtip]{xy}
\usepackage{ytableau}
\usetikzlibrary{calc}
\usetikzlibrary{arrows}

\usepackage[colorinlistoftodos]{todonotes}

\numberwithin{equation}{section}

\newtheorem{theorem}{Theorem}[section]
\newtheorem{proposition}[theorem]{Proposition}
\newtheorem{corollary}[theorem]{Corollary}
\newtheorem{lemma}[theorem]{Lemma}
\newtheorem{conjecture}[theorem]{Conjecture}
\theoremstyle{definition}
\newtheorem{observation}[theorem]{Observation}

\newtheorem{problem}[theorem]{Problem}
\newtheorem{example}[theorem]{Example}
\newtheorem{remark}[theorem]{Remark}

\newtheorem{notation}[theorem]{Notation}

\makeatletter
\newtheorem*{rep@theorem}{\rep@title}
\newcommand{\newreptheorem}[2]{%
\newenvironment{rep#1}[1]{%
 \def\rep@title{#2 \ref{##1}}%
 \begin{rep@theorem}}%
 {\end{rep@theorem}}}
\makeatother
\theoremstyle{definition}
\newtheorem{defn}[theorem]{Definition}

\newreptheorem{lemma}{Lemma}
\newreptheorem{theorem}{Theorem}

\def\cocoa{{\hbox{\rm C\kern-.13em o\kern-.07em C\kern-.13em o\kern-.15em A}}}

\newcommand{\sign}{{\mathrm {sign}}}
\newcommand{\Loc}{{\mathrm {Loc}}}
\newcommand{\Fun}{{\mathrm {Fun}}}

\newcommand{\ch}{{\mathrm {ch}}}
\newcommand{\height}{{\mathrm {ht}}}
\newcommand{\content}{{\mathrm {cont}}}

\newcommand{\type}{{\mathrm {type}}}
\newcommand{\sort}{{\mathrm {sort}}}

\newcommand{\III}{{\mathcal {I}}}
\newcommand{\WWW}{{\mathcal {W}}}
\newcommand{\UUU}{{\mathcal {U}}}
\newcommand{\TTT}{{\mathcal {T}}}
\newcommand{\OOO}{{\mathcal {O}}}

\newcommand{\maj}{{\mathrm {maj}}}
\newcommand{\wt}{{\mathrm {wt}}}
\newcommand{\depth}{{\mathrm {depth}}}

\newcommand{\shape}{{\mathrm {shape}}}

\newcommand{\End}{{\mathrm {End}}}
\newcommand{\cyc}{{\mathrm {cyc}}}

\newcommand{\symm}{{\mathfrak{S}}}

\newcommand{\CC}{{\mathbb {C}}}
\newcommand{\RR}{{\mathbb {R}}}

\newcommand{\AAA}{{\mathcal{A}}}
\newcommand{\BBB}{{\mathcal{B}}}

\newcommand{\GGG}{{\mathcal{G}}}

\newcommand{\xx}{{\mathbf {x}}}

\newcommand{\one}{{\mathbf{1}}}

\newcommand{\exc}{{\mathrm {exc}}}
\newcommand{\inv}{{\mathrm {inv}}}

\newcommand{\Class}{{\mathrm {Class}}}

\begin{document}

\title[Partial permutations and character evaluations]
{Partial permutations and character evaluations}

\author[Zachary Hamaker]{Zachary Hamaker}
\author[Brendon Rhoades]{Brendon Rhoades}

\address
{Department of Mathematics \newline \indent
University of Florida \newline \indent
Gainesville, FL, 32611, USA}
\email{zhamaker@ufl.edu}

\address
{Department of Mathematics \newline \indent
University of California, San Diego \newline \indent
La Jolla, CA, 92093, USA}
\email{bprhoades@ucsd.edu}

\begin{abstract}
Let $I = (i_1, \dots, i_k)$ and $J = (j_1, \dots, j_k)$ be two length $k$ sequences drawn from $\{1, \dots, n \}$. We have  the group algebra element $[I,J] := \sum_{w(I) = J} w \in \CC[\symm_n]$ where the sum is over permutations $w \in \symm_n$ which satisfy $w(i_p) = j_p$ for $p = 1, \dots, k$. We give an algorithm for evaluating irreducible characters $\chi^\lambda: \CC[\symm_n] \to \CC$ of the symmetric group on the elements $[I,J]$. This algorithm is a hybrid of the classical Murnaghan--Nakayama rule and a new {\em path Murnaghan--Nakayama rule} which reflects the decomposition of a partial permutation into paths and cycles.

These results first appeared in~\href{https://arxiv.org/abs/2206.06567}{arXiv:2206.06567}, which is no longer intended for publication.
We originally used the character theoretic results in this paper to prove asymptotic results on moments of certain permutation statistics restricted to conjugacy classes. A referee generously shared a combinatorial argument which is strong enough to prove these results without recourse to character theory.
These results now appear in our companion paper~\cite{HRMoment}.
However, the approach in this paper are more explicit, as we demonstrate with several examples.
\end{abstract}

\maketitle

\section{Introduction}
\label{sec:Introduction}

Let $\symm_n$ denote the symmetric group of permutations of $[n] := \{1, \dots, n\}$. 
The set $\Fun(\symm_n,\CC)$ of all functions $f: \symm_n \to \CC$ is in bijection with the group algebra $\CC[\symm_n]$ via
\[ f: \symm_n \to \CC \quad \Leftrightarrow \quad \sum_{w \, \in \, \symm_n} f(w) \cdot w \in \CC[\symm_n].
\]
For $k \leq n$, a {\em partial permutation of $[n]$ of size $k$} is an ordered pair $(I,J)$ where $I = (i_1, \dots, i_k)$ and $J = (j_1, \dots, j_k)$ are length $k$ lists of distinct elements of $[n]$. The pair $(I,J)$ represents the function $I \to J$ given by $i_1 \mapsto j_1, i_2 \mapsto j_2, \dots, i_k \mapsto j_k$. We define
\begin{equation}
    \symm_{n,k} := \{ \text{all partial permutations $(I,J)$ in $\symm_n$ of size $k$} \}.
\end{equation}
For any $(I,J) \in \symm_{n,k}$ we have an indicator statistic $\one_{I,J}: \symm_n \to \CC$ given by
\begin{equation}
    \label{eq:one-definition}
    \one_{I,J}(w) := \begin{cases}
        1 & \text{if }w(I) = J, \\
        0 & \text{otherwise,}
    \end{cases}
\end{equation}
where $w(I) = w(i_1,\dots,i_k) := (w(i_1), \dots, w(i_k))$.
We let $[I,J] \in \CC[\symm_n]$ for the associated group algebra element
\begin{equation}
    \label{eq:group-algebra-element-definition}
    [I,J] := \sum_{w \, \in \, \symm_n} \one_{I,J}(w) \cdot w = \sum_{w(I) \, = \, J} w \in \CC[\symm_n].
\end{equation}

The maps $\one_{I,J}: \symm_n \to \CC$ are not linearly independent. For example,  both sides of
\[ \one_{1,1} + \cdots + \one_{1,n} = \one_{1,1} + \cdots + \one_{n,1}\]
give the constant function $1$ on $\symm_n$. The $\one_{I,J}$ measure the complexity of functions $\symm_n \to \CC$ as follows.

\begin{defn}
    \label{def:local-statistics}
    A function $f: \symm_n \to \CC$ is {\em $k$-local} if there exist constants $c_{I,J} \in \CC$ such that 
    \[
    f = \sum_{(I,J) \, \in \, \symm_{n,k}} c_{I,J} \cdot \one_{I,J}
    \]
    as functions $\symm_n \to \CC$. We write $\Loc_k(\symm_n,\CC)$ for the $\CC$-vector space of $k$-local maps $f: \symm_n \to \CC$.
\end{defn}

The concept encoded in Definition~\ref{def:local-statistics} has appeared in previous literature \cite{EFP, EZ, HGG} as the {\em degree} of a permutation statistic. In that language, the $k$-local statistics are those of degree at most $k$.
We use our terminology in this paper to avoid confusion since many other notions degree will occur.


For example, the statistic $\inv(w) = \# \{ 1 \leq i < j \leq n \,:\, w(i) > w(j) \}$ is 2-local. Local statistics have appeared in various guises in the probability literature \cite{EFP, EZ, HGG}.
The spaces $\Loc_k(\symm_n,\CC)$ nest as a filtration
\begin{equation}
\label{eq:local-filtration}
    \CC = \Loc_0(\symm_n,\CC) \subseteq \Loc_1(\symm_n,\CC) \subseteq \cdots \subseteq \Loc_{n-1}(\symm_n,\CC) = \CC[\symm_n]
\end{equation}
where $\CC = \Loc_0(\symm_n,\CC)$ is the space of constant functions $\symm_n \to \CC$.
This filtration satisfies 
\begin{equation}
    \Loc_k(\symm_n,\CC) \cdot \Loc_{\ell}(\symm_n,\CC) \subseteq \Loc_{k + \ell}(\symm_n,\CC).
\end{equation}
under the pointwise product $(f \cdot g)(w) := f(w) \cdot g(w)$ of functions $\symm_n \to \CC$. A result of Ellis--Friedgut--Pilpel \cite{EFP} (see also Even-Zohar \cite{EZ} or Huang--Guestrin--Guibas \cite{HGG})
the vector space dimension of $\Loc_k(\symm_n,\CC)$ is 
\begin{equation}
    \dim \Loc_k(\symm_n,\CC) = \# \{ w \in \symm_n \,:\, w \text{ has an increasing subsequence of length $n-k$} \}.
\end{equation}
The expository Section~\ref{sec:Partial} proves these properties of $\Loc_k(\symm_n,\CC)$ using the techniques of algebraic combinatorics.

Our main results are about character evaluations. Irreducible characters $\chi^\lambda: \symm_n \to \CC$ of  $\symm_n$ are in one-to-one correspondence with partitions $\lambda \vdash n$. The {\em Murnaghan--Nakayama rule} is an algorithm which computes $\chi^\lambda(w)$ for any permutation $w \in \symm_n$ as a signed sum over certain ribbon tableaux. We extend $\chi^\lambda$ to a map $\CC[\symm_n] \to \CC$ by linearity:
\begin{equation}
\label{eq:character-linear-extension}
    \chi^\lambda \left(  \sum_{w \, \in \, \symm_n} a_w \cdot w \right) := \sum_{w \, \in \, \symm_n} a_w \cdot \chi^\lambda(w)
\end{equation}
for $a_w \in \CC$. For a given group algebra element $a = \sum_{w \in \symm_n} a_w \cdot w \in \CC[\symm_n]$, one can ask for a combinatorial rule for the character evaluation $\chi^\lambda(a)$. For general $a$, this problem is intractable: besides applying \eqref{eq:character-linear-extension} and the Murnaghan--Nakayama rule $n!$ times, nothing else can be said. Our main result may be paraphrased as follows.

\begin{theorem}
    \label{thm:intro-character-eval} {\em ($\subset$ Corollary~\ref{cor:trace-interpretation}, Theorem~\ref{rmk:degree-bounds})} Let $(I,J) \in \symm_{n,k}$. There is an explicit combinatorial rule given by a signed enumeration of ribbon tilings for calculating the family of irreducible character evaluation $\{ \chi^\lambda([I,J]) \,:\, \lambda \vdash n \}$.
    This rule calculates  $\{ \chi^\lambda([I,J]) \,:\, \lambda \vdash n \}$ using a constant (and finite) set of tilings whenever $n \geq 2k$.
\end{theorem}

 The na\"ive approach to computing $\{ \chi^\lambda([I,J]) \,:\, \lambda \vdash n \}$  requires $(n-k)!$ applications of the Murnaghan--Nakayama rule or each partition $\lambda \vdash n$. Theorem~\ref{thm:intro-character-eval} is therefore a massive improvement.

Our algorithm for computing $\chi^\lambda([I,J])$ uses the decomposition of the partial permutation $(I,J)$ into disjoint cycles and paths.  The classical Murnaghan--Nakayama rule computes the Schur expansion $p_\nu = \sum_{\lambda \vdash n} \chi^\lambda_\nu s_\lambda$ of a power sum symmetric function $p_\nu$. We define a new {\em path power sum} basis $\{ \vec{p}_\mu \,:\, \mu \vdash n \}$ of the vector space of degree $n$ symmetric functions (Corollary~\ref{cor:path-basis}). We prove a {\em path Murnaghan--Nakayama rule} which computes the Schur expansion $\vec{p}_\mu = \sum_{\lambda \vdash n} \vec{\chi}^\lambda_\mu \cdot s_\lambda$ of a path power sum (Theorem~\ref{thm:path-Murnaghan--Nakayama}).   Given a partial permutation $(I,J)$ of $[n]$, our algorithm for computing $\{ \chi^\lambda([I,J]) \,:\, \lambda \vdash n \}$ is a hybrid of the path Murhaghan-Nakayama rule for the paths of $(I,J)$ and the classical Murnaghan--Nakayama rule for the cycles of $(I,J)$.

The coefficients $\vec{\chi}^\lambda_\mu$ appearing in the Schur expansion of $\vec{p}_\mu = \sum_{\lambda \vdash n} \vec{\chi}^\lambda_\mu \cdot s_\lambda$ are signed counts of certain {\em monotonic ribbon tilings}; see Definition~\ref{def:monotonic-definition}.
An example monotonic ribbon tiling is shown below; the tails of the ribbons (the white circles) lie in distinct columns and left-to-right ribbon addition gives a chain in Young's lattice.  Proving the path Murnaghan--Nakayama rule is the most technical part of the paper.

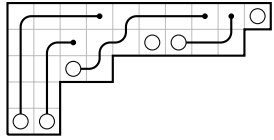
\begin{center}
\begin{tikzpicture}[scale = 0.35]

  \begin{scope}
    \clip (0,0) -| (2,2) -| (4,3) -| (9,4) -| (10,5) -| (0,0);
    \draw [color=black!25] (0,0) grid (10,5);
  \end{scope}

  \draw [thick] (0,0) -| (2,2) -| (4,3) -| (9,4) -| (10,5) -| (0,0);

  \draw [thick, rounded corners] (0.5,0.5) |- (3.5,4.5);
  \draw [color=black,fill=black,thick] (3.5,4.5) circle (.4ex);
  \node [draw, circle, fill = white, inner sep = 2pt] at (0.5,0.5) { };
  
  \draw [thick, rounded corners] (1.5,0.5) |- (2.5,3.5);
  \draw [color=black,fill=black,thick] (2.5,3.5) circle (.4ex);
  \node [draw, circle, fill = white, inner sep = 2pt] at (1.5,0.5) { };
  
  \draw [thick, rounded corners] (2.5,2.5) -| (3.5,3.5) -| (4.5,4.5) -- (7.5,4.5);
  \draw [color=black,fill=black,thick] (7.5,4.5) circle (.4ex);
  \node [draw, circle, fill = white, inner sep = 2pt] at (2.5,2.5) { };

    \node [draw, circle, fill = white, inner sep = 2pt] at (5.5,3.5) { };
    
  \draw [thick, rounded corners] (6.5,3.5) -| (8.5,4.5);
  \draw [color=black,fill=black,thick] (8.5,4.5) circle (.4ex);
  \node [draw, circle, fill = white, inner sep = 2pt] at (6.5,3.5) { };   
  
      \node [draw, circle, fill = white, inner sep = 2pt] at (9.5,4.5) { };

\end{tikzpicture}
\end{center}

To get a flavor of how monotonic tilings can rapidly calculate $\chi^\lambda([I,J])$, consider the special case $I = J = \varnothing$. Up to a scalar, the group algebra element  $[\varnothing,\varnothing] = \sum_{w \in \symm_n} w \in \CC[\symm_n]$ projects onto the trivial representation so that 
\[ [\varnothing,\varnothing] \cdot V^\lambda = 0 \text{ unless $\lambda = (n)$.}\]
Since $V^{(n)}$ is the trivial $\symm_n$-module, we have
\begin{equation}
\label{eq:intro-trivial-evaluation}
\chi^{\lambda}([\varnothing,\varnothing]) = \mathrm{trace}_{V^{\lambda}}\left( \sum_{w \, \in \, \symm_n} w \right) =  \begin{cases}
    n! & \text{if $\lambda = (n)$,} \\
    0 & \text{otherwise.}
\end{cases}
\end{equation}
Equation~\eqref{eq:intro-trivial-evaluation} could be obtained with \eqref{eq:character-linear-extension} and $n!$ applications of the Murnaghan--Nakayama rule for each partition $\lambda \vdash n$. On the other hand, since the partial permutation $(\varnothing,\varnothing) \in \symm_{n,0}$ has $n$ paths of size 1 and no cycles, the family  $\{ \chi^\lambda([\varnothing,\varnothing]) \,:\ \lambda \vdash n \}$ of character evaluations will arise as the coefficients $\vec{p}_{(1^n)} = \sum_{\lambda \vdash n} \chi^\lambda([\varnothing,\varnothing]) \cdot s_\lambda$ in the Schur expansion of the path power sum $\vec{p}_{(1^n)}$. The only monotonic ribbon tiling with $n$ ribbons of size 1 is 
\begin{center}
\begin{tikzpicture}[scale = 0.35]
\begin{scope}
   \clip (0,0) -| (2,1) -|  (0,0);
    \draw [color=black!25] (0,0) grid (10,1);
\end{scope}

\begin{scope}
   \clip (8,0) -| (10,1) -|  (8,0);
    \draw [color=black!25] (0,0) grid (10,1);
\end{scope}

\draw [thick]  (0,0) -| (10,1) -| (0,0);
\node [draw, circle, fill = white, inner sep = 1.2pt] at (0.5,0.5) { };   
\node [draw, circle, fill = white, inner sep = 1.2pt] at (1.5,0.5) { };   
\node [draw, circle, fill = white, inner sep = 1.2pt] at (9.5,0.5) { };   
\node [draw, circle, fill = white, inner sep = 1.2pt] at (8.5,0.5) { };   

\node at (5,0.5) {$\cdots$};
\end{tikzpicture}
\end{center}
because this is the only way to place $n$ size 1 ribbons in different columns and obtain a partition.  The path Murnaghan--Nakayama rule (Theorem~\ref{thm:path-Murnaghan--Nakayama}) gives the Schur expansion $\vec{p}_{(1^n)} = n! \cdot s_{n}$, yielding Equation~\eqref{eq:intro-trivial-evaluation} in a single step.

Let $\Class(\symm_n,\CC)$ be the vector space of class functions  $\symm_n \to \CC$ with its basis of irreducible characters $\{ \chi^\lambda \,:\, \lambda \vdash n\}$.  The {\em Reynolds operator} is the  projection $R: \Fun(\symm_n,\CC) \to \Class(\symm_n,\CC)$ defined by
\[ R \, f(w) := \frac{1}{n!} \sum_{v \, \in \, \symm_n} f(v^{-1} w v)\] for $f : \symm_n \to \CC$. Given an arbitrary function $f:\symm_n \to \CC$, we may regard $R \, f$ as the best class function approximation of $f$. Various authors have considered the irreducible character expansions of $R \, f$ where $f$ is a classical permutation statistic \cite{Hultman} or a statistic related to pattern counting \cite{GP, GR}. Our results calculate character expansions in a systematic way when $f$ is $k$-local for small $k$; this includes classical permutation statistics such as inversion number, descent number, exceedance number, and major index. The theory of character polynomials combines with our approach to give structural results on moments of statistics $f: \symm_n \to \CC$ restricted to conjugacy classes. These probabilistic results are derived in a more straightforward way in the companion paper \cite{HRMoment} by directly considering the symmetrized indicator functions $R \, \one_{I,J}$. An advantage of this paper is that it facilitates the explicit combinatorial expansion of $R \, \one_{I,J}$ into irreducible characters.


The rest of the paper is organized as follows. In {\bf Section~\ref{sec:Background}} we give background information on combinatorics, representation theory, and symmetric functions. {\bf Section~\ref{sec:Partial}} is devoted to the general study of local statistics and local class functions on $\symm_n$. {\bf Section~\ref{sec:Atomic}} introduces the atomic symmetric function $A_{n,I,J}$ of a partial permutation $(I,J)$ and the path power sum $\vec{p}_\mu$ indexed by a partition $\mu$. It is proven that $A_{n,I,J}$ factors into path and classical power sums according to the path and cycle types of $(I,J)$; see Proposition~\ref{prop:path-cycle-factorization}. {\bf Section~\ref{sec:Path}} is the heart of the paper in which the path Murnaghan--Nakayama rule is proven and our combinatorial rule for $\chi^\lambda([I,J])$ is deduced. {\bf Section~\ref{sec:Statistic}} applies the path Murnaghan--Nakayama rule to compute the irreducible character expansions of symmetrized versions for the classical statistics exceedance and major index. We conclude in {\bf Section~\ref{sec:Conclusion}} with some open problems.

\section{Background}
\label{sec:Background}

\subsection{Combinatorics} For $n \geq 0$, a {\em partition} of $n$ is a sequence $\lambda = (\lambda_1 \geq \cdots \geq \lambda_k)$ of positive integers with $\lambda_1 + \cdots + \lambda_k = n$. We write $\ell(\lambda) := k$ for the number of parts of $\lambda$ and $|\lambda| := n$ for the sum of the parts of $\lambda$. We adopt the shorthand $\lambda \vdash n$ to mean that $\lambda$ is a partition of $n$. We sometimes add trailing zeros to $\lambda$ so that e.g. $(3,1,1)$ and $(3,1,1,0,0)$ are the same partition of 5.  

We identify a partition $\lambda$ with its (English) {\em Young diagram} consisting of $\lambda_i$ left-justified boxes in row $i$. The Young diagram of $(4,3,1) \vdash 8$ is shown below.

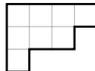
\begin{center}
\begin{tikzpicture}[scale = .3]
  \begin{scope}
    \clip (0,0) -| (1,1) -| (3,2) -| (4,3) -| (0,0);
    \draw [color=black!25] (0,0) grid (4,3);
  \end{scope}

  \draw [thick] (0,0) -| (1,1) -| (3,2) -| (4,3) -| (0,0);


\end{tikzpicture}
\end{center}

Let $\lambda$ be a partition. For $i \geq 1$ we let $m_i(\lambda)$ be the multiplicity of $i$ as a part of $\lambda$. We set \[m(\lambda)! := \prod_{i \, \geq \, 1} m_i(\lambda)!.\] For example, if $\lambda = (4,4,3,1,1,1)$ then $m(\lambda)! = m_4(\lambda)! m_3(\lambda)! m_2(\lambda)! m_1(\lambda)! = 2! \cdot 1! \cdot 0! \cdot 3! = 12$.

Let $\lambda, \mu$ be partitions such that $\mu_i \leq \lambda_i$ for all $i$. The {\em skew shape} $\lambda/\mu$ is the set-theoretic difference $\lambda/\mu := \lambda - \mu$ of Young diagrams. We define $|\lambda/\mu| := |\lambda| - |\mu|$. The skew shape $(4,4,1)/(2,1)$ is shown below.

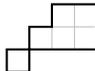
\begin{center}
\begin{tikzpicture}[scale = .3]
  \begin{scope}
    \clip (0,0) -| (1,1) -| (4,3) -| (2,2) -| (1,1) -| (0,0);
    \draw [color=black!25] (0,0) grid (4,3);
  \end{scope}
 
 \draw [thick] (0,0) -| (1,1) -| (4,3) -| (2,2) -| (1,1) -| (0,0);
\end{tikzpicture}
\end{center}

Let $n \geq 0$ A {\em (strong) composition} of $n$ is a sequence $\alpha = (\alpha_1, \dots, \alpha_k)$ of positive integers with $\alpha_1 + \cdots + \alpha_k = n$. As with partitions, we define $|\alpha| := n$ and $\ell(\alpha) := k$. We write $\alpha \models n$ to mean that $\alpha$ is a composition of $n$. 

A {\em set partition} of $[n]$ is a family $\pi = \{ B_1, \dots, B_k\}$ of nonempty subsets of $[n]$ (called {\em blocks}) such that $[n] = B_1 \sqcup \cdots \sqcup B_k$. We write $\Pi_n$ for the collection of all set partitions of $[n]$. The set $\Pi_n$ has a partial order given $\pi \leq \pi'$ if and only if $\pi'$ refines $\pi$.
The set partition $\{\{1\}, \dots, \{n\}\}$ is the minimum elements of this partial order and the set partition $\{\{1,\dots,n\}\}$ is the maximum element of this partial order.

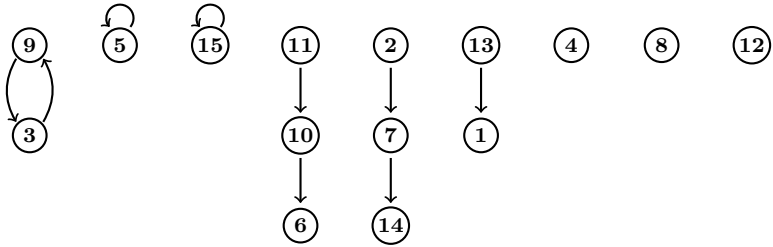
\begin{figure}
 \begin{center}
 \begin{tikzpicture}[scale = 0.6]
 
 \coordinate (v11) at (6,0);
 
 \coordinate (v10) at (6,-2);
 
 \coordinate (v6) at (6,-4);

 \coordinate (v2) at (8,0);
 
 \coordinate (v7) at (8,-2);
 
\coordinate (v14) at (8,-4);
  
 \coordinate (v13) at (10,0);
   
 \coordinate (v1) at (10,-2);
    
 \coordinate (v8) at (14,0);
 
 \coordinate (v4) at (12,0);
 
 \coordinate (v12) at (16,0);
 
 \coordinate (v9) at (0,0);
 
 \coordinate (v3) at (0,-2);
 
 \coordinate (v5) at (2,0);
 
 \coordinate (v15) at (4,0);

   \node [draw, circle, fill = white, inner sep = 2pt, thick] at (v1)
  {\scriptsize {\bf 1} };
     \node [draw, circle, fill = white, inner sep = 2pt, thick] at (v2)
  {\scriptsize {\bf 2}};
     \node [draw, circle, fill = white, inner sep = 2pt, thick] at (v3)
  {\scriptsize {\bf 3}};
     \node [draw, circle, fill = white, inner sep = 2pt, thick] at (v4)
  {\scriptsize {\bf 4}};
     \node [draw, circle, fill = white, inner sep = 2pt, thick] at (v5)
  {\scriptsize {\bf 5}};
     \node [draw, circle, fill = white, inner sep = 2pt, thick] at (v6)
  {\scriptsize {\bf 6}};
     \node [draw, circle, fill = white, inner sep = 2pt, thick] at (v7)
  {\scriptsize {\bf 7}};
     \node [draw, circle, fill = white, inner sep = 2pt, thick] at (v8)
  {\scriptsize {\bf 8}};
     \node [draw, circle, fill = white, inner sep = 2pt, thick] at (v9)
  {\scriptsize {\bf 9}};
     \node [draw, circle, fill = white, inner sep = 1pt, thick] at (v10)
  {\scriptsize {\bf 10}};
     \node [draw, circle, fill = white, inner sep = 1pt, thick] at (v11)
  {\scriptsize {\bf 11}};
     \node [draw, circle, fill = white, inner sep = 1pt, thick] at (v12)
  {\scriptsize {\bf 12}};
     \node [draw, circle, fill = white, inner sep = 1pt, thick] at (v13)
  {\scriptsize {\bf 13}};
     \node [draw, circle, fill = white, inner sep = 1pt, thick] at (v14)
  {\scriptsize {\bf 14}};
     \node [draw, circle, fill = white, inner sep = 1pt, thick] at (v15)
  {\scriptsize {\bf 15}};
  
  \draw [->, thick] (6,-0.5) -- (6,-1.5);
  \draw [->, thick] (6,-2.5) -- (6,-3.5);
 \draw [->, thick] (8,-0.5) -- (8,-1.5);
 \draw [->, thick] (8,-2.5) -- (8,-3.5);
 \draw [->, thick] (10,-0.5) -- (10,-1.5);
 \draw [->, thick]  (-0.3,-0.3) to[bend right]  (-0.3,-1.7);
 \draw [->, thick]  (0.3,-1.7) to[bend right]  (0.3,-0.3);
 
 \draw[->, thick]  ($(2,0.6) + (-40:3mm)$) arc (-40:220:3mm);

 \draw[->, thick]  ($(4,0.6) + (-40:3mm)$) arc (-40:220:3mm);
 
 \end{tikzpicture} 
 \end{center}
 \caption{The graph $G(I,J)$ of $(I,J) \in \symm_{15,9}$  where $I = (11,10,2,7,13,9,3,5,15)$ and $J = (10,6,7,14,1,3,9,5,15)$. 
 The cycle type of is $(2,1,1)$ and the path type is $(3,3,2,1,1,1)$.}
 \label{fig:graph}
 \end{figure}

As in the introduction, a {\em partial permutation of $[n]$} is an ordered pair $(I,J)$ where $I$ and $J$ are sequences of elements of $[n]$ of the same length and the entries of both $I$ and $J$ are distinct. The common length of $I$ and $J$ is the {\em size} of the partial permutation $(I,J)$. We write $\symm_{n,k}$ for the family of partial permutations of $[n]$ of size $k$. 

Let $(I,J)$ be a partial permutation of $[n]$ with $I = (i_1, \dots, i_k)$ and $J = (j_1, \dots, j_k)$. The {\em graph} of $(I,J)$ is the directed graph $G(I,J)$ on the vertex set $[n]$ with edges $i_1 \to j_1, \dots, i_k \to j_k$. See Figure~\ref{fig:graph} for an example. 

For any partial permutation $(I,J)$ of $[n]$, each component of $G(I,J)$ is a directed path or a directed cycle. The {\em size} of a path or a cycle is the number of vertices in that path or cycle. The {\em path type} of $(I,J)$ is the partition $(\mu_1 \geq \mu_2 \geq \cdots )$ obtained by writing the path sizes in $G(I,J)$ in weakly decreasing order.\footnote{This convention differs from that in our companion paper \cite{HRMoment}; our path size is the path size in \cite{HRMoment} plus one.}
The {\em cycle type} of $(I,J)$ is the partition $(\nu_1 \geq \nu_2 \geq \cdots)$ obtained by writing the cycle sizes in $G(I,J)$ in weakly decreasing order; see Figure~\ref{fig:graph}.

\begin{observation}
    \label{obs:path-cycle-type}
    Suppose $(I,J) \in \symm_{n,k}$ has path type $\mu$ and cycle type $\nu$. Then
    \[ |\mu| + |\nu| = n \quad \text{and} \quad |\mu| - \ell(\mu) + |\nu| = k.\]
\end{observation}

\subsection{Representation theory} All representations in this paper will be finite-dimensional and over  $\CC$. Let $G$ be a finite group and let $V$ be a $G$-module. Write $\End_\CC(V)$ for the $\CC$-algebra of $\CC$-linear maps $\varphi: V \to V$. The space $\End_\CC(V)$ carries the structure of a $(G \times G)$-module via
\[ ((g_1,g_2) \cdot \varphi)(v) := g_1 \cdot \varphi( g_2^{-1} \cdot v)\]
for $g_1, g_2 \in G$, $\varphi \in \End_\CC(V)$, and $v \in V$. 
The following result will be crucial in our analysis of local statistics.  

\begin{theorem}
    \label{thm:aw}
    {\em (Artin-Wedderburn)} Let $G$ be a finite group. We have an isomorphism of $\CC$-algebras
    \[ \Phi: \CC[G] \longrightarrow \bigoplus_V \End_\CC(V) \]
    where the direct sum ranges over (isomorphism class representatives of) irreducible $G$-modules and $\Phi$ is defined by
    \[ \Phi(a) \cdot v := a \cdot v \quad \quad \text{for all $v \in V$}\]
    for all $a \in \CC[G]$, and each $G$-irreducible $V$. 
\end{theorem}

Let $\Fun(G,\CC)$ be the $\CC$-vector space of functions $f: G \to \CC$. We identify $\Fun(G,\CC)$ with $\CC[G]$ by associating $f: G \to \CC$ to the group algebra element $\sum_{g \in G} f(g) \cdot g \in \CC[G]$.  A function $f: G \to \CC$ is a {\em class function} if $f(h g h^{-1}) = f(g)$ for all $g,h \in G$. Let $\Class(G,\CC)$ be the vector space of class functions $f: G \to \CC$. Then $\Class(G,\CC)$ is a subspace of $\Fun(G,\CC)$. The center $Z(\CC[G])$ of the group algebra $\CC[G]$ is given by
\begin{equation} Z(\CC[G]) = \left\{ \sum_{g \, \in \, G} f(g) \cdot g \,:\, f \in \Class(G,\CC) \right\}. \end{equation}
The image of $Z(\CC[G])$ under the Artin-Wedderburn isomorphism $\Phi$ of Theorem~\ref{thm:aw} is as follows.

\begin{corollary}
    \label{cor:aw-central}
    Let $G$ be a finite group with group algebra $\CC[G]$. Let $Z(\CC[G])$ be the center of $\CC[G]$. The image of $Z(\CC[G])$ under the isomorphism $\Phi: \CC[G] \xrightarrow{ \, \sim \,} \bigoplus_V \End_\CC(V)$ of Theorem~\ref{thm:aw} is
    \[ \Phi( Z(\CC[G])) = \bigoplus_V \CC \cdot \mathrm{id}_V.\]
    The direct sum is again over (isomorphism class representatives of) irreducible $G$-modules $V$ and $\CC \cdot \mathrm{id}_V$ is the 1-dimensional subspace of $\End_\CC(V)$ consisting of scalar maps $V \to V$.
\end{corollary}

Corollary~\ref{cor:aw-central} follows from Theorem~\ref{thm:aw} as follows. If $V$ is a finite-dimensional vector space, then $\End_\CC(V)$ is isomorphic to the algebra of $(\dim V) \times (\dim V)$ complex matrices. It is not difficult to show that the center of this matrix algebra consists of scalar matrices.

\subsection{Symmetric functions}
Let $\Lambda = \bigoplus_{n \geq 0} \Lambda_n$ be the graded ring of symmetric functions in an infinite variable set $\xx = (x_1, x_2, \dots )$ over the field $\CC$. If $f \in \Lambda$ and $N \geq 0$, we write $f(x_1, \dots, x_N) := f(x_1, \dots, x_N, 0, 0, \dots )$ for the polynomial in $x_1, \dots, x_N$ obtained by evaluating $x_i \to 0$ for all $i > N$.  

Bases of $\Lambda_n$ are indexed by partitions $\lambda \vdash n$. We will use the {\em power sum basis} $p_\lambda$ and {\em Schur basis} $s_\lambda$.
If $\lambda = (\lambda_1, \dots, \lambda_k)$, we have the power sum symmetric function
\[
p_\lambda := p_{\lambda_1} \cdots p_{\lambda_k} \text{ where } p_d := \sum_{i \, \geq \, 1} x_i^d \text{ for $d > 0$}.
\]
The bialternant definition of Schur functions will be the most convenient for our purposes. Fix $N \geq 0$ and let $\varepsilon \in \CC[\symm_N]$ be the group algebra element
\begin{equation}
    \varepsilon := \sum_{w \, \in \, \symm_N} \sign(w) \cdot w
\end{equation}
which antisymmetrizes over $\symm_N$. The element $\varepsilon$ acts on the polynomial ring $\CC[x_1, \dots, x_N]$. An {\em alternant} is a polynomial $f \in \CC[x_1, \dots, x_N]$ such that $\varepsilon \cdot f = f$. For any sequence $(\beta_1, \dots, \beta_N)$ of nonnegative integers we have an alternant 
\begin{equation}
    a_\beta(x_1,\dots,x_N) := \varepsilon \cdot x^{\beta + \delta}
\end{equation}
where $\delta := (N-1,\dots,1,0)$ and $\beta + \delta$ is componentwise addition of exponent vectors.
For a partition $\lambda = (\lambda_1, \lambda_2, \dots )$, the {\em Schur function} $s_\lambda \in \Lambda$ is the limit
\begin{equation}
    s_\lambda = \lim_{N \to \infty} s_\lambda (x_1, \dots, x_N) := \lim_{N \to \infty} \frac{\varepsilon \cdot x^{\lambda + \delta}}{\varepsilon \cdot x^\delta}.
\end{equation}

If $w \in \symm_n$ is a permutation, the {\em cycle type} $\mathrm{cyc}(w) \vdash n$ is the partition of $n$ obtained by writing the cycle lengths of $w$ in weakly decreasing order. The {\em Frobenius characteristic map}
\[ \ch_n : \Class(\symm_n,\CC) \longrightarrow \Lambda_n \]
is defined by
\[ \ch_n: f \mapsto \frac{1}{n!} \sum_{w \, \in \, \symm_n} f(w) \cdot p_{\mathrm{cyc}(w)}. \]
The map $\ch_n$ is an isomorphism of $\CC$-vector spaces $\Class(\symm_n, \CC) \xrightarrow{ \, \sim \, } \Lambda_n$.

Irreducible representations of $\symm_n$ are in one-to-one correspondence with partitions of $n$. If $\lambda \vdash n$ is a partition, we write $V^\lambda$ for the associated $\symm_n$-irreducible. The vector space dimension of $V^\lambda$ equals the number $f^\lambda$ of standard Young tableaux of shape $\lambda$. Let $\chi^\lambda: \symm_n \to \CC$ be the character of $V^\lambda$. The character $\chi^\lambda$ is a class function and we have
$\ch_n(\chi^\lambda) = s_\lambda.$
Given two partitions $\lambda, \mu \vdash n$, let $\chi^\lambda_\mu$ be the common value of $\chi^\lambda(w)$ on any permutation $w \in \symm_n$ of cycle type $\mu$. The power sum and Schur bases of $\Lambda_n$ are related by 
\begin{equation}
\label{eq:classical-power-to-schur}
    p_\mu = \sum_{\lambda \, \vdash \, n} \chi^\lambda_\mu \cdot s_\lambda.
\end{equation}

The {\em Murnaghan--Nakayama rule} is a signed formula for the Schur expansion of the power sum $p_\mu$. In light of \eqref{eq:classical-power-to-schur}, this gives a method for computing the irreducible characters of $\symm_n$. In order to state this rule, we need some combinatorial definitions.

 A {\em ribbon} $\xi$ is an edgewise connected skew partition whose Young diagram contains no $2 \times 2$ square. The {\em height} $\height(\xi)$
 of a ribbon $\xi$ is the number of rows occupied by $\xi$, less one. The {\em sign} of $\xi$ is $\sign(\xi) := (-1)^{\height(\xi)}$.
For example, the figure
 \begin{center}
 \begin{tikzpicture}[scale = .3]
  \begin{scope}
    \clip (0,0) -| (1,1) -| (2,2) -| (3,3) -| (5,5) -| (0,0);
    \draw [color=black!25] (0,0) grid (5,5);
  \end{scope}

  \draw [thick] (0,0) -| (1,1) -| (2,2) -| (3,3) -| (5,5) -| (0,0);

  \draw [thick, rounded corners] (1.5,1.5) |- (2.5,2.5) |- (4.5,3.5) ;
  \draw [color=black,fill=black,thick] (4.5,3.5) circle (.4ex);
  \node [draw, circle, fill = white, inner sep = 2pt] at (1.5,1.5) { };
  \end{tikzpicture}
 \end{center}
 shows a ribbon $\xi$ of size 6, satisfying $\height(\xi) = 2$ and $\sign(\xi) = (-1)^2 = +1$ embedded inside the 
 Young diagram of $(5,5,3,2,1)$.
 Ribbons whose removal from a Young diagram yields a Young diagram (such as above) are called {\em rim hooks}.
 The southwesternmost cell in a ribbon $\xi$ is called its {\em tail}; the tail is decorated with a white circle $\circ$ in the above picture.

Let $\lambda/\rho$ be a skew shape. A {\em standard ribbon tableau} $T$ of shape $\lambda/\rho$ is a chain
\[
T = (\rho = \lambda^{(0)} \subset \lambda^{(1)} \subset \cdots \subset \lambda^{(r)} = \lambda)
\]
of partitions such that each difference $\xi^{(i)} := \lambda^{(i)}/\lambda^{(i-1)}$ is a ribbon. The {\em type} of $T$ is the composition $(\mu_1, \dots, \mu_r)$ where $\mu_i := |\xi^{(i)}|$. The {\em sign} of $T$ is the product 
\[
\sign(T) := \sign(\xi^{(1)}) \cdots \sign(\xi^{(r)})
\]
of the constituent ribbons of $T$.

\begin{theorem}
    \label{thm:classical-mn}
    {\em (Murnaghan--Nakayama rule)} 
    Let $\lambda$ be a partition and let $r > 0$. We have 
    \[
    s_\lambda \cdot p_r = \sum_\nu \sign(\nu/\lambda) \cdot s_\nu
    \]
    where the sum is over all $\nu \supseteq \lambda$ such that $\nu/\lambda$ is a ribbon of size $r$.
\end{theorem}

Let $\lambda, \mu \vdash n$. Iterating Theorem~\ref{thm:classical-mn} yields
\[ \chi^\lambda_\mu = \sum_T \sign(T)\]
 where the sum is over standard ribbon tableaux $T$ of shape $\lambda$ and type $\mu$. 
For example, let $\mu = (3,2,2,1)$. The coefficient $\chi^{431}_{3221}$ of 
$s_{431}$ in $p_{3221}$ is the signed sum of standard ribbon tableaux of 
shape $\lambda = (4,3,1)$ and type $\mu = (3,2,2,1)$.  The five ribbon tableaux in question are shown in Figure~\ref{fig:classical-mn}; we conclude that
$\chi^{431}_{3221} = 1 - 1 + 1 - 1 - 1 = -1.$

\begin{figure}
\begin{center}
\begin{tikzpicture}[scale = .4]
  \begin{scope}
    \clip (0,0) -| (1,1) -| (3,2) -| (4,3) -| (0,0);
    \draw [color=black!25] (0,0) grid (4,3);
  \end{scope}

  \draw [thick] (0,0) -| (1,1) -| (3,2) -| (4,3) -| (0,0);

  \draw [thick, rounded corners] (0.5,0.5) -- (0.5,2.5);
  \draw [color=black,fill=black,thick] (0.5,2.5) circle (.4ex);
  \node [draw, circle, fill = white, inner sep = 1pt] at (0.5,0.5)
  {\scriptsize 1};

  \draw [thick, rounded corners] (1.5,1.5) -- (1.5,2.5);
  \draw [color=black,fill=black,thick] (1.5,2.5) circle (.4ex);
  \node [draw, circle, fill = white, inner sep = 1pt] at (1.5,1.5)
  {\scriptsize 2};

  \draw [thick, rounded corners] (2.5,1.5) -- (2.5,2.5);
  \draw [color=black,fill=black,thick] (2.5,2.5) circle (.4ex);
  \node [draw, circle, fill = white, inner sep = 1pt] at (2.5,1.5)
  {\scriptsize 3};

  \node [draw, circle, fill = white, inner sep = 1pt] at (3.5,2.5)
  {\scriptsize 4};

 \node at (2.5,0) {$+1$};

   \begin{scope}
    \clip (6,0) -| (7,1) -| (9,2) -| (10,3) -| (6,0);
    \draw [color=black!25] (6,0) grid (10,3);
  \end{scope}

  \draw [thick] (6,0) -| (7,1) -| (9,2) -| (10,3) -| (6,0);

  \draw [thick, rounded corners] (6.5,0.5) -- (6.5,2.5);
  \draw [color=black,fill=black,thick] (6.5,2.5) circle (.4ex);
  \node [draw, circle, fill = white, inner sep = 1pt] at (6.5,0.5)
  {\scriptsize 1};

  \draw [thick, rounded corners] (7.5,1.5) -- (7.5,2.5);
  \draw [color=black,fill=black,thick] (7.5,2.5) circle (.4ex);
  \node [draw, circle, fill = white, inner sep = 1pt] at (7.5,1.5)
  {\scriptsize 2};

  \draw [thick, rounded corners] (8.5,2.5) -- (9.5,2.5);
  \draw [color=black,fill=black,thick] (9.5,2.5) circle (.4ex);
  \node [draw, circle, fill = white, inner sep = 1pt] at (8.5,2.5)
  {\scriptsize 3};

  \node [draw, circle, fill = white, inner sep = 1pt] at (8.5,1.5)
  {\scriptsize 4};
  
  \node at (8.5,0) {$-1$};

 \begin{scope}
    \clip (12,0) -| (13,1) -| (15,2) -| (16,3) -| (12,0);
    \draw [color=black!25] (12,0) grid (16,3);
  \end{scope}

  \draw [thick] (12,0) -| (13,1) -| (15,2) -| (16,3) -| (12,0);

  \draw [thick, rounded corners] (12.5,0.5) -- (12.5,2.5);
  \draw [color=black,fill=black,thick] (12.5,2.5) circle (.4ex);
  \node [draw, circle, fill = white, inner sep = 1pt] at (12.5,0.5)
  {\scriptsize 1};

  \draw [thick, rounded corners] (13.5,2.5) -- (14.5,2.5);
  \draw [color=black,fill=black,thick] (14.5,2.5) circle (.4ex);
  \node [draw, circle, fill = white, inner sep = 1pt] at (13.5,2.5)
  {\scriptsize 2};

  \draw [thick, rounded corners] (13.5,1.5) -- (14.5,1.5);
  \draw [color=black,fill=black,thick] (14.5,1.5) circle (.4ex);
  \node [draw, circle, fill = white, inner sep = 1pt] at (13.5,1.5)
  {\scriptsize 3};

  \node [draw, circle, fill = white, inner sep = 1pt] at (15.5,2.5)
  {\scriptsize 4};
  
  \node at (14.5,0) {$+1$};

   \begin{scope}
    \clip (18,0) -| (19,1) -| (21,2) -| (22,3) -| (18,0);
    \draw [color=black!25] (18,0) grid (22,3);
  \end{scope}

  \draw [thick]  (18,0) -| (19,1) -| (21,2) -| (22,3) -| (18,0);

  \draw [thick, rounded corners] (18.5,2.5) -- (20.5,2.5);
  \draw [color=black,fill=black,thick] (20.5,2.5) circle (.4ex);
  \node [draw, circle, fill = white, inner sep = 1pt] at (18.5,2.5)
  {\scriptsize 1};

  \draw [thick, rounded corners] (18.5,0.5) -- (18.5,1.5);
  \draw [color=black,fill=black,thick] (18.5,1.5) circle (.4ex);
  \node [draw, circle, fill = white, inner sep = 1pt] at (18.5,0.5)
  {\scriptsize 2};

  \draw [thick, rounded corners] (19.5,1.5) -- (20.5,1.5);
  \draw [color=black,fill=black,thick] (20.5,1.5) circle (.4ex);
  \node [draw, circle, fill = white, inner sep = 1pt] at (19.5,1.5)
  {\scriptsize 3};

  \node [draw, circle, fill = white, inner sep = 1pt] at (21.5,2.5)
  {\scriptsize 4};
  
  \node at (20.5,0) {$-1$};

     \begin{scope}
    \clip (24,0) -| (25,1) -| (27,2) -| (28,3) -| (24,0);
    \draw [color=black!25] (24,0) grid (28,3);
  \end{scope}

  \draw [thick]  (24,0) -| (25,1) -| (27,2) -| (28,3) -| (24,0);

  \draw [thick, rounded corners] (24.5,1.5) |- (25.5,2.5);
  \draw [color=black,fill=black,thick] (25.5,2.5) circle (.4ex);
  \node [draw, circle, fill = white, inner sep = 1pt] at (24.5,1.5)
  {\scriptsize 1};

  \draw [thick, rounded corners] (26.5,2.5) -- (27.5,2.5);
  \draw [color=black,fill=black,thick] (27.5,2.5) circle (.4ex);
  \node [draw, circle, fill = white, inner sep = 1pt] at (26.5,2.5)
  {\scriptsize 2};

  \draw [thick, rounded corners] (25.5,1.5) -- (26.5,1.5);
  \draw [color=black,fill=black,thick] (26.5,1.5) circle (.4ex);
  \node [draw, circle, fill = white, inner sep = 1pt] at (25.5,1.5)
  {\scriptsize 3};

  \node [draw, circle, fill = white, inner sep = 1pt] at (24.5,0.5)
  {\scriptsize 4};

\node at (26.5,0) {$-1$};
  
\end{tikzpicture}
\end{center}
\caption{The standard ribbon tableaux computing $\chi^{431}_{3221}$ together
with their signs. The numbers in the tails indicate the order in which ribbons are added.}
\label{fig:classical-mn}
\end{figure}
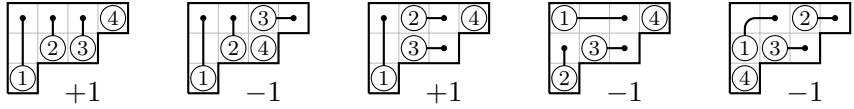

The status of $\chi^{\lambda}_{\mu}$ as a character evaluation guarantees that $\chi^{\lambda}_{\mu}$ is invariant
under permutations of $\mu = (\mu_1, \dots, \mu_r)$, although this is not evident from
Theorem~\ref{thm:classical-mn}.
Changing the order of the sequence $\mu = (\mu_1, \dots, \mu_r)$ can result in different sets of standard ribbon tableaux
(and these sets can have different sizes), but the signed ribbon tableau count is independent of the order of $(\mu_1, \dots, \mu_r)$.
Mendes  \cite{Mendes} gave a combinatorial proof of this result.
We extend the notation $\chi^\lambda_\mu$ by defining
\begin{equation}
    \chi^{\lambda/\rho}_\mu := \sum_T \sign(T)
\end{equation}
where $\lambda/\rho$ is a skew shape and
the sum is over standard ribbon tableaux of shape $\lambda/\rho$ and type $\mu$.

\section{Partial permutations and local statistics}
\label{sec:Partial}

This section develops basic properties of the space $\Loc_k(\symm_n,\CC)$ of $k$-local statistics $\symm_n \to \CC$. The results are proven in various places and in various guises throughout the probability literature. We state and prove them here in using the language and terminology of algebraic combinatorics.

\subsection{The locality filtration} Recall that $\Loc_k(\symm_n,\CC)$ is the vector space spanned by the indicator statistics $\one_{I,J}: \symm_n \to \CC$ for $(I,J) \in \symm_{n,k}$. The pointwise product of these indicator functions is as follows.

\begin{lemma}
    \label{lem:multiply-by-singleton}
    Let $(I,J) \in \symm_{n,k}$ and $1 \leq i, j \leq n$. Write $I = (i_1, \dots, i_k)$ and $J = (j_1, \dots, j_k)$. We have
    \[
    \one_{I,J} \cdot \one_{i,j} = \begin{cases}
        \one_{I \, i,  \, J \, j} & \text{if $i \notin I$ and $j \notin J$,} \\
        \one_{I,J} & \text{if $i = i_r$ and $j = j_r$ for some $1 \leq r \leq k$,} \\
        0 & \text{otherwise.}
    \end{cases}
    \]
    Here $I \, i := (i_1, \dots, i_k,i)$ and $J \, j := (j_1, \dots, j_k,j)$.
\end{lemma}

\begin{proof}
    For any $w \in \symm_n$, we have 
    \[
    (\one_{I,J} \cdot \one_{i,j})(w) = \one_{I,J}(w) \cdot \one_{i,j}(w) = \begin{cases}
        1 & \text{if $w(I) = J$ and $w(i) = j$,} \\
        0 & \text{otherwise,}
    \end{cases}
    \]
    and the lemma follows.
\end{proof}

Lemma~\ref{lem:multiply-by-singleton} implies that the spaces $\Loc_k(\symm_n,\CC)$ nest as $k$ varies.

\begin{proposition}
    \label{prop:local-nesting}
    Any $k$-local statistic $f: \symm_n \to \CC$ is $(k+1)$-local.
\end{proposition}

\begin{proof}
    Let $(I,J) \in \symm_{n,k}$. It suffices to show that $\one_{I,J}: \symm_n \to \CC$ is $(k+1)$-local. Indeed, fix $1 \leq j \leq n$ with $j \notin J$. Since $\one_{1,j} + \cdots + \one_{n,j} = 1$, Lemma~\ref{lem:multiply-by-singleton} implies
    \[ \one_{I,J} = \one_{I,J} \cdot \left( \one_{1,j} + \cdots + \one_{n,j}  \right) = \sum_{i \, \notin \, I} \one_{I \, i, \, J \, j}\]
    which is visibly $(k+1)$-local.
\end{proof}

Proposition~\ref{prop:local-nesting} gives the filtration
\begin{equation}
\label{eq:locality-filtration}
\CC = \Loc_0(\symm_n,\CC) \subseteq \Loc_1(\symm_n,\CC) \subseteq \cdots \subseteq \Loc_{n-1}(\symm_n,\CC) = \Fun(\symm_n,\CC)
\end{equation}
where the equality $\Loc_{n-1}(\symm_n,\CC) = \Fun(\symm_n,\CC)$ follows from
\begin{equation}
    \one_{(1, 2, \dots, n-1), \, (w(1), w(2), \dots, w(n-1))}: v \mapsto \begin{cases} 1 & v = w \\ 0  & v \neq w \end{cases} \quad \text{for all $w,v \in \symm_n$.}
\end{equation}
The filtration \eqref{eq:locality-filtration} behaves well under pointwise product.

\begin{proposition}
    \label{prop:local-product}
    Let $f, g: \symm_n \to \CC$ be statistics. If $f$ is $k$-local and $g$ is $\ell$-local then $f \cdot g: \symm_n \to \CC$ is $(k + \ell)$-local.
\end{proposition}

\begin{proof}
    Let $(I,J) \in \symm_{n,m}$ with $I = (i_1, \dots, j_m)$ and $J = (j_1, \dots, j_m)$. Lemma~\ref{lem:multiply-by-singleton} implies
    \[ \one_{I,J} = \one_{i_1,j_1} \cdots \one_{i_m,j_m}.\]
    If $(I_1,J_1) \in \symm_{n,k}$ and $(I_2,J_2) \in \symm_{n,\ell}$, we deduce that there exist $1 \leq a_r, b_r \leq n$ for $r = 1, \dots, k + \ell$ so that
    \[
    \one_{I_1,J_1} \cdot \one_{I_2,J_2} = \prod_{r \, = \, 1}^{k + \ell} \one_{a_r,b_r}.
    \]
    Lemma~\ref{lem:multiply-by-singleton} and Proposition~\ref{prop:local-nesting} imply that $\one_{I_1,J_1} \cdot \one_{I_2,J_2}$ is $(k + \ell)$-local. The result follows by linearity.
\end{proof}

The product group $\symm_n \times \symm_n$ acts on the group algebra $\CC[\symm_n]$ by left and right multiplication. The induced action on the function space $\Fun(\symm_n,\CC)$ is given by
\begin{equation}
    ((u,v) \cdot f)(w) = f(u^{-1} w v) \quad \quad \text{ for all } u,v,w \in \symm_n \text{ and } f: \symm_n \to \CC.
\end{equation}
The filtration \eqref{eq:locality-filtration} is stable under this action.

\begin{lemma}
    \label{lem:action-on-indicators}
    Let $u,v \in \symm_n$ and let $(I,J) \in \symm_{n,k}$. We have
    \[
    (u,v) \cdot \one_{I,J} = \one_{v(I),u(J)}.
    \]
\end{lemma}

\begin{proof}
    Let $w \in \symm_n$. We have 
    \begin{equation}
        ((u,v) \cdot \one_{I,J})(w) = \one_{I,J}(u^{-1}wv) = \begin{cases} 1 & u^{-1}wv(I) = J \\ 0 & u^{-1}wv(I) \neq J\end{cases} = 
        \begin{cases} 1 & wv(I) = u(J) \\ 0 & wv(I) \neq u(J) \end{cases} = \one_{v(I),u(J)}(w).
    \end{equation}
\end{proof}

\subsection{Artin-Wedderburn and increasing subsequences}
Recall that $V^\lambda$ is the irreducible representation of $\symm_n$ associated to a partition $\lambda \vdash n$. The Artin-Wedderburn Theorem~\ref{thm:aw} in the case of $G = \symm_n$ therefore gives an isomorphism of $\CC$-algebras
\begin{equation}
    \label{eq:aw-sn}
    \Phi: \CC[\symm_n] \xrightarrow{ \, \, \sim \, \, } \bigoplus_{\lambda \, \vdash \, n} \End_\CC(V^\lambda)
\end{equation}
where $\Phi(a)$ acts on the $\lambda$-component by $v \mapsto a \cdot v$ for all $v \in V^\lambda$.
If we identify $\CC[\symm_n] = \Fun(\symm_n,\CC)$, for any $k$ the subspace $\Loc_k(\symm_n,\CC) \subseteq \End_\CC(V^\lambda)$ maps to a subspace of  $\bigoplus_{\lambda \, \vdash \, n} \End_\CC(V^\lambda)$ under $\Phi$.  This image may be described explicitly as follows.

\begin{theorem}
    \label{thm:aw-image}
    For $n,k \geq 0$, the image of $\Loc_k(\symm_n, \CC)$ under the Artin-Wedderburn isomorphism $\Phi$ is 
    \[ \Phi \left(\Loc_k(\symm_n,\CC)  \right) = \bigoplus_{\substack{\lambda \,  \vdash \, n \\ \lambda_1 \, \geq \, n-k}} \End_\CC(V^\lambda).\]
\end{theorem}

\begin{proof}
    Consider the embedding of symmetric groups $\symm_{n-k} \subseteq \symm_n$ where $\symm_{n-k}$ acts on the first $n-k$ letters. Let $\eta_{n-k}$ be the group algebra element
    \begin{equation}
        \eta_{n-k} := \sum_{w \, \in \, \symm_{n-k}} w \in \CC[\symm_n].
    \end{equation}
    In the correspondence $\CC[\symm_n] \leftrightarrow \Fun(\symm_n,\CC)$, we have $\eta_{n-k} \leftrightarrow \one_{I_0,J_0}$ where \[I_0 = J_0 = (n-k+1,\dots,n-1,n).\]
    Lemma~\ref{lem:action-on-indicators} implies that the subspace $\Loc_k(\symm_n,\CC) \subseteq \Fun(\symm_n,\CC)$ of $k$-local statistics corresponds to the two-sided ideal $\III_{n,k} \subseteq \CC[\symm_n]$ generated by $\eta_{n-k}$. We therefore want to compute the image
    \[
    \Phi(\III_{n,k}) \subseteq \bigoplus_{\lambda \, \vdash \, n} \End_\CC(V^\lambda)
    \]
    of this ideal under $\Phi$. Since $\Phi$ is an algebra isomorphism, the subset $\Phi(\III_{n,k})$ is an ideal in $\bigoplus_{\lambda \vdash n} \End_\CC(V^\lambda)$.

    For any $\lambda \vdash n$, an easy computation shows that the matrix ring $\End_\CC(V^\lambda) \cong \mathrm{Mat}_{f^\lambda}(\CC)$ is {\em simple}: it has no nontrivial two-sided ideals. Consequently, the ideals of the direct sum $\bigoplus_{\lambda \vdash n} \End_\CC(V^\lambda)$ correspond to subsets of $\{ \lambda \,:\, \lambda \vdash n \}$. In particular, there is a set $P(n,k)$ of partitions of $n$ such that
    \begin{equation}
        \Phi(\III_{n,k}) = \bigoplus_{\lambda \, \in \, P(n,k)} \End_\CC(V^\lambda).
    \end{equation}
    The definition of $\Phi$ implies that 
    \begin{equation}
        P(n,k) = \{ \lambda \vdash n \,:\, \eta_{n-k} \cdot V^\lambda \neq 0 \}.
    \end{equation}

    If $V$ is an $\symm_n$-module, it is not hard to see that \[ \eta_{n-k} \cdot V = V^{\symm_{n-k}} := \{ v \in V \,:\, w \cdot v 
= v \text{ for all $w \in \symm_{n-k}$} \}.\]
In particular, we have $\eta_{n-k} \cdot V \neq 0$ if and only if the trivial $\symm_{n-k}$-module appears in the restriction $\mathrm{Res}^{\symm_n}_{\symm_{n-k}}(V)$ of $V$ from $\symm_n$ to $\symm_{n-k}$. We conclude that
\begin{equation}
    P(n,k) = \left\{ \lambda \vdash n \,:\, \mathrm{triv}_{\symm_{n-k}} \text{ appears in } \mathrm{Res}^{\symm_n}_{\symm_{n-k}}(V^\lambda) \right\} = \{ \lambda \vdash n\,:\, \lambda_1 \geq n-k \}
\end{equation}
where the second equality uses $\mathrm{triv}_{\symm_{n-k}} = V^{(n-k)}$ and the branching rule for symmetric group modules. This completes the proof.
\end{proof}

\begin{remark}
    \label{rmk:local-literature}
    Statements equivalent to Theorem~\ref{thm:aw-image} (or portions thereof) have appeared in the literature. In machine learning, Huang, Guestrin, and Guibas \cite[Appendix C]{HGG} proved the inclusion $\subseteq$ using the orthogonal basis of $\symm_n$-irreducibles.
    Ellis, Friedgut and Pilpel proved \cite[Thm. 7]{EFP} the reverse inclusion $\supseteq$ using the branching rule for $\symm_n$-modules. In his study of permutation patterns, Even-Zohar proved a `dual' result \cite[Lem. 7]{EZ} which classifies when functions $f: \symm_n \to \CC$ vanish when summed over any two-sided coset $u \cdot \symm_r \cdot v$ for $r \leq n$. The Peter-Weyl Theorem shows that \cite[Lem. 7]{EZ} is equivalent to Theorem~\ref{thm:aw-image}. We have included a proof of Theorem~\ref{thm:aw-image} for convenience and in a form which we hope will be the most accessible to the reader in algebraic combinatorics.
\end{remark}

As a consequence of Theorem~\ref{thm:aw-image}, we get a formula for the dimension of the space of $k$-local statistics $\symm_n \to \CC$. For $\lambda \vdash n$, let $f^\lambda := \dim V^\lambda$.

\begin{corollary}
    \label{cor:local-dimension}
    For $n,k \geq 0$, the dimension of $\Loc_k(\symm_n,\CC)$ is 
    \[ \dim \Loc_k(\symm_n,\CC) = \sum_{\substack{\lambda \, \vdash \, n \\ \lambda_1 \, \geq \, n-k}} (f^\lambda)^2.\]
    This quantity also equals the number of permutations $w = [w(1), \dots, w(n)] \in \symm_n$ whose longest increasing subsequence has length $\geq n-k$.
\end{corollary}

The second author described \cite{RhoadesShadow} an explicit basis of $\Loc_k(\symm_n,\CC)$ in terms of the Viennot shadow line avatar \cite{Viennot} of the Schensted correspondence.

\begin{proof}
    The first statement follows from taking dimensions in Theorem~\ref{thm:aw-image}. It is well-known that $f^\lambda$ counts standard Young tableaux of shape $\lambda$, so the second statement follows from standard properties of the Schensted correspondence \cite{Schensted}.
\end{proof}

Under the identification $\Fun(\symm_n,\CC) = \CC[\symm_n]$, the subspace $\Class(\symm_n\CC) \subseteq \Fun(\symm_n,\CC)$ of class functions corresponds to the center $Z(\CC[\symm_n])$ of the symmetric group algebra. The image of the space of $k$-local class functions under the Artin-Wedderburn isomorphism has the following description.

\begin{corollary}
    \label{cor:local-class}
    For $n,k \geq 0$, the image of the space of $k$-local class functions $f: \symm_n \to \CC$ under the Artin-Wedderburn isomorphism $\Phi$ is given by
    \[ \Phi \left(\Loc_k(\symm_n,\CC) \cap \Class(\symm_n,\CC  \right)) = \bigoplus_{\substack{\lambda \,  \vdash \, n \\ \lambda_1 \, \geq \, n-k}} \CC \cdot \mathrm{id}_{V^\lambda}\]
    where $\CC \cdot \mathrm{id}_{V^\lambda}$ is the 1-dimensional vector space of scalar maps $V^\lambda \to V^\lambda$.
\end{corollary}

\begin{proof}
    Apply Corollary~\ref{cor:aw-central} and Theorem~\ref{thm:aw-image}.
\end{proof}

Recall that $\chi^\mu: \symm_n \to \CC$ is the irreducible character associated to $\mu \vdash n$. It is well-known that the group algebra element $\sum_{w \in \symm_n} \chi^\mu(w) \cdot w \in \CC[\symm_n]$ acts on $\symm_n$-modules by projection onto the $\mu$-isotypic component (up to a nonzero scalar).\footnote{The group algebra element $\sum_{w \in \symm_n} \chi^\mu(w) \cdot w \in \CC[\symm_n]$ is central, and so (by Schur's lemma) acts by a scalar $c_{\lambda,\mu} \in \CC$ on any $\symm_n$-irreducible $V^\lambda$. Character orthogonality implies that the trace of $\sum_{w \in \symm_n} \chi^\mu(w) \cdot w$ acting on $V^\lambda$ is $c_{\lambda, \mu} \cdot \dim(V^\lambda) = \sum_{w \in \symm_n} \chi^\mu(w) \cdot \chi^\lambda(w) = n! \cdot \delta_{\mu,\lambda}$ where $\delta_{\mu,\lambda}$ is the Kronecker delta.} In particular, the image 
\[
\Phi \left( \sum_{w \, \in \, \symm_n} \chi^\mu(w) \cdot w  \right) \in \bigoplus_{\lambda \, \vdash \, n} \End_\CC(V^\lambda)
\]
of this element under the Artin-Wedderburn isomorphism acts by a nonzero scalar on $V^\mu$ and annihilates $V^\lambda$ for all $\lambda \neq \mu$. Corollary~\ref{cor:local-class} therefore shows 
\begin{equation} 
\label{eq:irreducible-character-locality}
\text{ the irreducible character }
    \chi^\mu: \symm_n \to \CC \text{ is $k$-local but not $(k-1)$-local where } k = n - \mu_1.
\end{equation}
By the same reasoning, we see that 
\begin{multline}
    \label{eq:local-class-basis}
    \{ \chi^\lambda: \symm_n \to \CC \,:\, \lambda_1 \geq n-k \} \text{ is a basis of the vector space } \\ \Loc_k(\symm_n,\CC) \cap \Class(\symm_n,\CC) \text{ of $k$-local class functions $\symm_n \to \CC$.}
\end{multline}
We deduce the following classical result of Murnaghan.

\begin{corollary}
    \label{cor:kronecker} {\em (Murnaghan \cite{Murnaghan})}
    Let $\lambda, \mu \vdash n$ and consider the irreducible character expansion
    \[\chi^\lambda \cdot \chi^\mu = \sum_{\nu \, \vdash \, n} g_{\lambda,\mu,\nu} \cdot \chi^\nu\]
    of the pointwise product $\chi^\lambda \cdot \chi^\nu$. We have $g_{\lambda,\mu,\nu} = 0$ unless 
    \[ (n - \nu_1) \leq (n - \lambda_1) + (n - \mu_1).\]
\end{corollary}

The coefficients $g_{\lambda,\mu,\nu}$ in Corollary~\ref{cor:kronecker} are called {\em Kronecker coefficients}.

\begin{proof}
    Combine \eqref{eq:local-class-basis} with Proposition~\ref{prop:local-product}.
\end{proof}

\section{Atomic symmetric functions and path power sums}
\label{sec:Atomic}

\subsection{Atomic symmetric functions} In this section we turn our focus from general $k$-local statistics $f: \symm_n \to \CC$ to $k$-local class functions. Recall that if $f: \symm_n \to \CC$ is a function, its image $R\, f$ under the Reynolds operator is given by
\[
R \, f: w  \mapsto \frac{1}{n!} \sum_{v \, \in \, \symm_n} f(v^{-1} w v).
\]
This formula defines a projection $R: \Fun(\symm_n,\CC) \to \Class(\symm_n,\CC)$. In particular, the space $\Loc_k(\symm_n,\CC) \cap \Class(\symm_n,\CC)$ of $k$-local class functions is spanned by $\{ R \, \one_{I,J} \,:\, (I,J) \in \symm_{n,k} \}$. We use the Frobenius characteristic $\ch_n: \Class(\symm_n, \CC) \xrightarrow{ \, \sim \, } \Lambda_n$ to encode $R \, \one_{I,J}$ as a symmetric function.

\begin{defn}
    \label{def:atomic}
    Let $(I,J) \in \symm_{n,k}$. The {\em atomic symmetric function} $A_{n,I,J}$ is 
    \[ A_{n,I,J} := n! \cdot \ch_n (R \, \one_{I,J}).\]
\end{defn}

The factor of $n!$ in Definition~\ref{def:atomic} makes our results more aesthetic. For example, it banishes a factor of $\frac{1}{n!}$ from the power sum expansion of $A_{n,I,J}$.

\begin{proposition}
    \label{prop:power-sum-expansion}
    Let $(I,J) \in \symm_{n,k}$. We have the power sum expansion
    \[ A_{n,I,J} = \sum_{\substack{w \, \in \, \symm_n \\ w(I) \, = \,J}} p_{\mathrm{cyc}(w)}\]
    where $\mathrm{cyc}(w) \vdash n$ is the cycle type of the permutation $w \in \symm_n$.
\end{proposition}

\begin{proof}
    For any function $f: \symm_n \to \CC$, we claim 
    \begin{equation}
        \label{eq:ch-R-formula}
        \ch_n(R \, f) = \frac{1}{n!} \sum_{w \, \in \, \symm_n} f(w) \cdot p_{\cyc(w)}.
    \end{equation}
    Equation~\eqref{eq:ch-R-formula} follows from the computation
    \begin{multline}
        \ch_n(R \, f) = \frac{1}{n!} \cdot \sum_{x \, \in \, \symm_n} R \, f(x) \cdot p_{\cyc(x)} = \frac{1}{n!} \cdot \sum_{x \, \in \symm_n} \left( \frac{1}{n!} \sum_{v \in \symm_n} f(v^{-1}xv)  \right) \cdot p_{\cyc(x)} \\ 
        = \left( \frac{1}{n!} \right)^2 \cdot \sum_{x,v \, \in \, \symm_n} f(v^{-1}xv) \cdot p_{\cyc(x)} =
        \left( \frac{1}{n!} \right)^2 \cdot \sum_{w,v \, \in \, \symm_n} f(w) \cdot p_{\cyc(vwv^{-1})} \\
        =  
        \left( \frac{1}{n!} \right)^2 \cdot \sum_{w,v \, \in \, \symm_n} f(w) \cdot p_{\cyc(w)} = \frac{1}{n!} \sum_{w \, \in \, \symm_n} f(w) \cdot p_{\cyc(w)}
    \end{multline}
    where we made the change of variables $w = v^{-1}xv$ and used $\cyc(v w v^{-1}) = \cyc(w)$. The proposition is proven by taking $f = \one_{I,J}$.
\end{proof}

The factor of $n!$ in Definition~\ref{def:atomic} also makes the Schur expansion of $A_{n,I,J}$ look nicer. If $(I,J)$ is a partial permutation of $[n]$, recall the group algebra element $[I,J] = \sum_{w(I) = J} w \in \CC[\symm_n]$.

\begin{proposition}
    \label{prop:atomic-character-interpretation}
    Let $(I,J)$ be a partial permutation of $[n]$. In the Schur expansion
    \[ A_{n,I,J} = \sum_{\lambda \, \vdash \, n} c_\lambda \cdot s_\lambda\]
    the coefficient $c_\lambda$ is the trace of the linear operator $[I,J]: V^\lambda \to V^\lambda$. Equivalently, we have
    \[ c_\lambda =  \chi^\lambda([I,J]) =\sum_{\substack{w \, \in \, \symm_n \\ w(I) \, = \, J}} \chi^\lambda(w). \]
\end{proposition}

\begin{proof}
    If $\mu \vdash n$, we have $p_\mu = \sum_{\lambda \vdash n} \chi^{\lambda}_\mu \cdot s_\lambda$. Applying Proposition~\ref{prop:power-sum-expansion}, we see that
    \begin{multline}
        A_{n,I,J} = \sum_{\substack{w \, \in \, \symm_n \\ w(I) \, = \, J}} p_{\cyc(w)} = 
        \sum_{\substack{w \, \in \, \symm_n \\ w(I) \, = \, J}} \left( \sum_{\lambda \, \vdash \, n} \chi^\lambda(w) \cdot s_\lambda \right) \\ =
        \sum_{\lambda \, \vdash \, n} \chi^\lambda \left(\sum_{\substack{w \, \in \, \symm_n \\ w(I) \, = \, J}} w  \right) \cdot s_\lambda = \sum_{\lambda \, \vdash \, n} \chi^\lambda([I,J]) \cdot s_\lambda
    \end{multline}
    and the result is proven.
\end{proof}

Proposition~\ref{prop:atomic-character-interpretation} gives a restriction on the support of the Schur expansion of $A_{n,I,J}$. The following result will also follow from our path Murnaghan--Nakayama formula.

\begin{corollary}
    \label{cor:atomic-support-bound}
    Let $(I,J) \in \symm_{n,k}$. In the Schur expansion
    \[ A_{n,I,J} = \sum_{\lambda \, \vdash \, n} c_\lambda \cdot s_\lambda\]
    we have $c_\lambda = 0$ unless $\lambda_1 \geq n-k$.
\end{corollary}

\begin{proof}
    Theorem~\ref{thm:aw-image} implies that $[I,J] \in \CC[\symm_n]$ acts as the zero operator on $V^\lambda$ whenever $\lambda_1 < n-k$. Proposition~\ref{prop:atomic-character-interpretation} yields the corollary.
\end{proof}

\subsection{Path power sums and path-cycle factorization} In this subsection we prove an important factorization property of the atomic symmetric functions. To begin, we show that atomic functions are invariant under the action of $\symm_n$ on partial permutations.

\begin{proposition}
    \label{prop:atomic-conjugacy-independence}
    Let $(I,J)$ be a partial permutation of $[n]$. For any $w \in \symm_n$ we have
    \[A_{n,I,J} = A_{n,w(I),w(J)}.\]
\end{proposition}

\begin{proof}
    Lemma~\ref{lem:action-on-indicators} implies  
    \[ R \, \one_{I,J} = \frac{1}{n!} \sum_{v \, \in \, \symm_n} \one_{v(I),v(J)} = \frac{1}{n!} \sum_{v \, \in \, \symm_n} \one_{vw(I),vw(J)} = R \, \one_{w(I),w(J)} \]
    and the result follows from Definition~\ref{def:atomic}.
\end{proof}

Proposition~\ref{prop:atomic-conjugacy-independence} gives an alternative way to index atomic symmetric functions. 
Let $\GGG_n$ be the family of directed graphs $G$ on $n$ unlabelled vertices which are disjoint unions of directed paths and directed cycles. For $G \in \GGG_n$, we set 
\begin{equation}
    A_G := A_{n,I,J}
\end{equation}
where $(I,J)$ is any partial permutation of $[n]$ whose graph $G(I,J)$ coincides with $G$ after erasing labels. Proposition~\ref{prop:atomic-conjugacy-independence} guarantees that $A_G$ is independent of the choice of $(I,J)$. Graphs which consist only of paths will play a special role. In the next definition, we write $P_n$ for the directed path on $n$ unlabeled vertices.

\begin{defn}
    \label{def:path-power-sum}
    Let $\mu = (\mu_1, \dots, \mu_k) \models n$ be a composition of $n$ and let $G \in \GGG_n$ be the disjoint union
    $G := P_{\mu_1} \sqcup \cdots \sqcup P_{\mu_k}$
    of directed paths. The {\em path power sum} $\vec{p}_\mu \in \Lambda_n$ is the symmetric function 
    \[\vec{p}_\mu := A_G.\]
\end{defn}

The notation $\vec{p}_\mu$ is meant to evoke of the shape of the graph whose atomic function equals $\vec{p}_\mu$. Every atomic function factors as a product of a path and classical power sum.

\begin{proposition}
    \label{prop:path-cycle-factorization}
    Let $(I,J) \in \symm_{n,k}$ be a partial permutation, let $\mu$ be the path type of $(I,J)$, and let $\nu$ be the cycle type of $(I,J)$. We have 
    \[ A_{n,I,J} = \vec{p}_\mu \cdot p_\nu.\]
\end{proposition}

\begin{proof}
    Let $G \in \GGG_n$, let $C$ be a directed cycle in $G$, and let $G - C$ be the graph obtained by removing the cycle $C$ from $G$. Proposition~\ref{prop:power-sum-expansion} implies that 
    \[ A_G = A_{G - C} \cdot p_c \]
    where $c$ is the number of vertices in $C$. The result follows by induction.
\end{proof}

Proposition~\ref{prop:power-sum-expansion} leads to an expansion of path power sums into classical power sums. To state this expansion, we introduce some notation which will be useful in this section and the next.

\begin{notation}
\label{not:mu-parts}
Let $\mu = (\mu_1, \dots, \mu_r)$ be a composition with $r$ parts. If $B \subseteq [r]$, we set 
    \[ \mu_B := \sum_{b \, \in \, B} \mu_b.\]
    Similarly, if $C$ is a cycle of a permutation in $\symm_r$, we set
    \[ \mu_C := \sum_{c \, \in \, C} \mu_c.\]
    Finally, if $w = w_1 \dots w_k$ is a word over the alphabet $[r]$, we set
    \[ \mu_w := \sum_{i \, = \, 1}^k \mu_{w_i}.\]    
\end{notation}

For example, we have $\mu_{\{2,4,5\}} = \mu_{(2,5,4)} = \mu_{425} = \mu_2 + \mu_4 + \mu_5$. Recall that $\Pi_r$ is the lattice of set partitions of $[r]$.

\begin{proposition}
    \label{prop:path-to-classical}
    Let $\mu = (\mu_1, \dots, \mu_r)$ be a composition with $r$ parts. We have 
    \[
    \vec{p}_\mu = \sum_{w \, \in \, \symm_r} \prod_{C \, \in \, w} p_{\mu_C} = \sum_{\pi \, \in \, \Pi_r} \prod_{B \, \in \, \pi} (\# B  - 1)!  \cdot p_{\mu_B}
    \]
    where the product in the middle expression is over all cycles $C$ belonging to $w \in \symm_r$.
\end{proposition}

\begin{proof}
    The first equality follows from Proposition~\ref{prop:power-sum-expansion}. The second equality applies the fact that there are $(\# S - 1)!$ ways to cyclically order a finite set $S$.
\end{proof}

For example, if $\mu = (a,b,c)$ has three parts we have
\[\vec{p}_{(a,b,c)} = p_{(a,b,c)} + p_{(a+b,c)} + p_{(a+c,b)} + p_{b+c,a)} + 2 \cdot p_{(a+b+c)}.\]
For a partition $\mu \vdash n$, the terms $p_\rho$ appearing in the $p$-expansion of $\vec{p}_\mu$ are indexed by partitions $\rho$ obtained by combining parts of $\mu$. Since the coefficient of $p_\mu$ in $\vec{p}_\mu$ is 1, the following corollary is a consequence of Proposition~\ref{prop:path-to-classical}.

\begin{corollary}
    \label{cor:path-basis}
    The set $\{ \vec{p}_\mu \,:\, \mu \vdash n \}$ is a basis of $\Lambda_n$.
\end{corollary}

There is a clean formula for the expansion of a classical power sum in the path power sum basis. If $\mu = (\mu_1, \dots, \mu_r)$ is a composition with $r$ parts we have
\begin{equation}
\label{eq:classical-to-path}
    p_\mu = \sum_{\pi \, \in \, \Pi_r} (-1)^{r \, - \, \# \pi} \prod_{B \, \in \, \pi} \vec{p}_{\mu_B}.
\end{equation}
The proof of Equation~\eqref{eq:classical-to-path} uses the formula
\begin{equation}
    \mu_{\Pi_r}(\hat{0},\pi) = (-1)^{r \, - \, \# \pi} \prod_{B \, \in \, \pi} (\#B - 1)!
\end{equation}
for the M\"obius function of $\Pi_r$ together with Proposition~\ref{prop:path-to-classical} and M\"obius inversion.
We will have no use for Equation~\eqref{eq:classical-to-path} and leave the details of its verification to the reader.

\section{Path Murnaghan--Nakayama Rule}
\label{sec:Path}

\subsection{Classical Murnaghan--Nakayama via alternants}
The classical Murnaghan--Nakayama rule describes the coefficients $\chi^{\lambda}_{\mu}$ in the expansion
$p_{\mu} = \sum_{\lambda \vdash n} \chi^{\lambda}_{\mu} \cdot s_{\lambda}$ in terms of ribbon tableaux.
There are several ways to prove this result. The proof that extends best to the path power sum setting uses alternants;
we recall this proof before  treating the more elaborate case of
 $\vec{p}_{\mu}$
 
Restricting to a finite variable set $\{x_1, \dots, x_N\}$,  for $\lambda$ with $\leq N$ parts, one has
\begin{equation}
s_{\lambda}(x_1, \dots, x_N) = \frac{ \varepsilon \cdot x^{\lambda + \delta} }{ \varepsilon \cdot x^{\delta}}
\end{equation}
where  $\varepsilon = \sum_{w \in \symm_N} \sign(w) \cdot w \in \CC[\symm_N]$ and $\delta = (N-1,N-2,\dots,1,0)$.
The numerator is the alternant
$a_{\lambda}(x_1, \dots, x_N) = \varepsilon \cdot x^{\lambda + \delta}.$
The set $\{ a_{\lambda}(x_1,\dots,x_N) \,:\, \ell(\lambda) \leq N \}$ 
 is a basis of the vector space $\varepsilon \cdot \CC[x_1,\dots,x_N]$  of alternating 
polynomials.
Since $p_\mu(x_1,\dots,x_N)$ is symmetric, we have
\begin{equation}
\label{eqn:first-classical}
p_{\mu}(x_1, \dots, x_N) \times \varepsilon \cdot (x^{\delta}) =
\varepsilon \cdot \left( p_{\mu}(x_1, \dots, x_N) \times x^{\delta} \right).
\end{equation}
Multiplying through by $\varepsilon \cdot x^\delta$, the Murnaghan--Nakayama rule is equivalent to finding the expansion of 
\eqref{eqn:first-classical} in the alternant basis $\{ a_\lambda(x_1,\dots,x_n) \,:\, \ell(\lambda) \leq N \}$.

Since the $p$-basis is multiplicative, the alternant expansion of \eqref{eqn:first-classical} can be understood inductively. For $k \geq 1$ and  $\lambda = (\lambda_1, \dots, \lambda_N)$ we have
\begin{multline}
\label{eqn:inductive-step}
p_k(x_1, \dots, x_N) \cdot a_{\lambda}(x_1, \dots, x_N) = 
\varepsilon \cdot \left( p_k(x_1, \dots, x_N) \cdot x^{\lambda + \delta}  \right) \\ =
\sum_{i \, = \, 1}^N \varepsilon \cdot \left( x_1^{\lambda_1 + N-1} \cdots x_i^{k + \lambda_i + N-i} \cdots x_N^{\lambda_N} \right).
\end{multline}
Each term 
$\varepsilon \cdot \left( x_1^{\lambda_1 + N-1} \cdots x_i^{k + \lambda_i + N-i} \cdots x_N^{\lambda_N} \right)$ in this sum 
is an alternant, the negative of an alternant, or zero according to the following combinatorial criterion.

\begin{observation}
\label{obs:ribbon-addition}
For $1 \leq i \leq N$, we have 
\begin{equation*}
\varepsilon \cdot \left( x_1^{\lambda_1 + N-1} \cdots x_i^{k + \lambda_i + N-i} \cdots x_N^{\lambda_N} \right) \neq 0
\end{equation*}
if and only if it is possible to add a size $k$ ribbon $\xi$ to the Young diagram of $\lambda$ such that the tail of $\xi$ is in row $i$
and $\lambda \cup \xi$ is the Young diagram of a partition. In this case, we have
\begin{equation}
\varepsilon \cdot \left( x_1^{\lambda_1 + N-1} \cdots x_i^{k + \lambda_i + N-i} \cdots x_N^{\lambda_N} \right) =
\sign(\xi) \cdot a_{\lambda \cup \xi}.
\end{equation}
\end{observation}

For example, let $k = 3, \lambda = (3,3,1),$ and $N = 6$. Starting with the length
$N = 6$ sequence $(\lambda_1, \dots, \lambda_6) + (\delta_1, \dots, \delta_6) = (8,7,4,2,1,0)$, adding $k = 3$ in all possible positions yields
\begin{multline*}
(11,7,4,2,1,0), \quad (8,10,4,2,1,0), \quad (8,7,7,2,1,0), \\ \quad (8,7,4,5,1,0),  \quad (8,7,4,2,4,0), \quad \text{and} \quad
(8,7,4,2,1,3).
\end{multline*}
The third and fifth sequences above have repeated terms, so their corresponding monomials in $\CC[x_1, \dots, x_6]$
are annihilated by $\varepsilon$. Sorting the remaining sequences into decreasing order and applying $\varepsilon$
introduces signs of $+1, -1, -1,$ and $+1$ into the first, second, fourth, and sixth sequences (respectively).
At the level of partitions, these are the four signs associated to the ways to add size $3$ ribbons to $\lambda = (3,3,1)$.
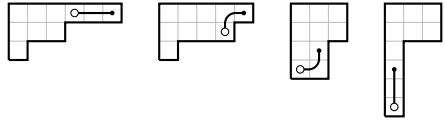
\begin{center}
\begin{tikzpicture}[scale = 0.25]

  \begin{scope}
    \clip (37,0) -| (38,4) -| (40,6) -| (37,0);
    \draw [color=black!25] (37,0) grid (40,6);
  \end{scope}

  \draw [thick] (37,0) -| (38,4) -| (40,6) -| (37,0);

  \draw [thick, rounded corners] (37.5,0.5) -- (37.5,2.5);
  \draw [color=black,fill=black,thick] (37.5,2.5) circle (.4ex);
  \node [draw, circle, fill = white, inner sep = 1pt] at (37.5,0.5) { };

  \begin{scope}
    \clip (32,2) -| (34,4) -| (35,6) -| (32,2);
    \draw [color=black!25] (32,2) grid (35,6);
  \end{scope}

  \draw [thick] (32,2) -| (34,4) -| (35,6) -| (32,2);

  \draw [thick, rounded corners] (32.5,2.5) -| (33.5,3.5);
  \draw [color=black,fill=black,thick] (33.5,3.5) circle (.4ex);
  \node [draw, circle, fill = white, inner sep = 1pt] at (32.5,2.5) { };

 \begin{scope}
    \clip (25,3) -| (26,4) -| (29,5) -| (30,6) -| (25,3);
    \draw [color=black!25] (25,3) grid (30,6);
  \end{scope}

  \draw [thick] (25,3) -| (26,4) -| (29,5) -| (30,6) -| (25,3);

   \draw [thick, rounded corners] (28.5,4.5) |- (29.5,5.5);
  \draw [color=black,fill=black,thick] (29.5,5.5) circle (.4ex);
  \node [draw, circle, fill = white, inner sep = 1pt] at (28.5,4.5) { };

 \begin{scope}
    \clip (17,3) -| (18,4) -| (20,5) -| (23,6) -| (17,3);
    \draw [color=black!25] (17,3) grid (23,6);
  \end{scope}

  \draw [thick]  (17,3) -| (18,4) -| (20,5) -| (23,6) -| (17,3);
  
   \draw [thick, rounded corners] (20.5,5.5) -- (22.5,5.5);
  \draw [color=black,fill=black,thick] (22.5,5.5) circle (.4ex);
  \node [draw, circle, fill = white, inner sep = 1pt] at (20.5,5.5) { };

\end{tikzpicture}
\end{center}
The tails of these added ribbons are in rows 1, 2, 4, and 6.  These agree with the relative values 
of the terms in $\lambda + \delta = (8,7,4,2,1,0)$ to which $k  =3$ was added. The illegal ribbon additions
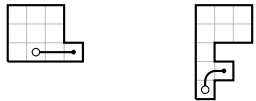
\begin{center}
\begin{tikzpicture}[scale = 0.25]

\begin{scope}
   \clip (0,0) -| (4,1) -| (3,3) -| (0,0);
   \draw [color=black!25] (0,0) grid (4,3);
\end{scope}

\draw [thick] (0,0) -| (4,1) -| (3,3) -| (0,0);

  \draw [thick, rounded corners] (1.5,0.5) -- (3.5,0.5);
  \draw [color=black,fill=black,thick] (3.5,0.5) circle (.4ex);
  \node [draw, circle, fill = white, inner sep = 1pt] at (1.5,0.5) { };

\begin{scope}
   \clip (10,-2) -|  (11,-1) -| (12,0) -| (11,1) -| (13,3) -| (10,-2);
   \draw [color=black!25] (13,3) grid (10,-2);
\end{scope}  

\draw [thick] (10,-2) -|  (11,-1) -| (12,0) -| (11,1) -| (13,3) -| (10,-2);

  \draw [thick, rounded corners] (10.5,-1.5) |- (11.5,-0.5);
  \draw [color=black,fill=black,thick] (11.5,-0.5) circle (.4ex);
  \node [draw, circle, fill = white, inner sep = 1pt] at (10.5,-1.5) { };

\end{tikzpicture}
\end{center}
obtained by adding a 3-ribbon with tail in rows 3 and 5 correspond precisely to the sequences 
$\lambda + \delta + (0,0,3,0,0,0)$ and $\lambda + \delta + (0,0,0,0,3,0)$ having repeated terms.

Observation~\ref{obs:ribbon-addition} and induction on the number of parts of $\mu$ show that
multiplying $p_{\mu} = p_{\mu_1} p_{\mu_2} \cdots $ by $\varepsilon \cdot x^{\delta}$ corresponds to building up a standard ribbon
tableau of type $\mu$ ribbon by ribbon, starting with the empty tableau $\varnothing$. That is, we have
\begin{equation}
p_{\mu}(x_1, \dots, x_N) \times \varepsilon \cdot x^\delta =
\sum_T \sign(T) \times  a_{\shape(T)}(x_1, \dots, x_N)
\end{equation}
where the sum is over all standard ribbon tableaux $T$ of type $\mu$. Dividing by  $\varepsilon \cdot x^\delta$ and taking the limit
as $N \rightarrow \infty$ yields the Murnaghan--Nakayama rule  (Theorem~\ref{thm:classical-mn}).

\subsection{Word arrays}
We embark on finding the $s$-expansion of $\vec{p}_{\mu}$.  The difficulty in adapting the proof
in the previous subsection to path power sums is that $\vec{p}_{\mu}$ is not multiplicative in $\mu$
so the induction argument breaks down. In spite of this, the $s$-expansion of the path power sum $\vec{p}_{\mu}$ will admit 
a reasonably compact expansion in terms of certain ribbon tilings.
In this section we set the stage for this result by giving a combinatorial description of the alternating
polynomial $\vec{p}_{\mu}(x_1, \dots, x_N) \times \varepsilon \cdot x^\delta$.

Fix an integer $N \gg 0$ and a partition $\mu = (\mu_1, \dots, \mu_r)$ with $r$ parts.  A {\em $\mu$-word array} of length $N$
is  a sequence $\omega = ( w_1 \mid \cdots \mid w_N)$ of finite (possibly empty) words $w_1, \dots, w_N$
over the alphabet $\{1, \dots, r \}$.
The word array $\omega$ is {\em standard} if each letter $1, \dots, r$ appears exactly once among the words 
$w_1,  \dots, w_N$.
We will mainly be interested in standard word arrays, but will pass through more general word arrays 
in the course of our inductive arguments.

If $\omega = (w_1 \mid w_2 \mid \cdots \mid w_{N-1} \mid w_N)$ is a $\mu$-word array of length $N$, the {\em weight}
$\wt(\omega)$ is the sequence of integers
\begin{multline}
\wt(\omega) := ( \mu_{w_1} + N - 1, \mu_{w_2} + N - 2, \dots, \mu_{w_{N-1}} + 1, \mu_{w_N} ) \\ =
( \mu_{w_1}, \mu_{w_2}, \dots, \mu_{w_{N-1}}, \mu_{w_N}) + (N-1,N-2, \dots, 1,0)
\end{multline}
where the addition in the second line is componentwise.  Recall that $\mu_{w_i}$ is shorthand for the sum of the parts $\mu_j$
corresponding to the letters $j$ of the word $w_i$.
The sequence $\wt(\omega)$ depends on $\mu$; we leave this dependence implicit to reduce clutter.
Word arrays have the following connection to path power sums.

\begin{lemma}
\label{lem:path-to-array}
Let $\mu = (\mu_1, \dots, \mu_r)$ be a partition. 
We have
\begin{equation}
\vec{p}_{\mu}(x_1, \dots, x_N) \times \varepsilon \cdot x^\delta = \sum_{\omega}
 \varepsilon \cdot x^{\wt(\omega)}
\end{equation}
where the sum is over all standard $\mu$-word arrays $\omega = (w_1 \mid \cdots \mid w_N)$ of length $N$.
\end{lemma}

\begin{proof}
Proposition~\ref{prop:path-to-classical} implies  
\begin{equation}
\label{path-to-array-one}
\vec{p}_{\mu}(x_1, \dots, x_N) = \sum_{w \, \in \, \symm_r} \prod_{C \, \in \, w} p_{\mu_C}(x_1, \dots, x_N)
\end{equation}
where the notation $C \in w$ indicates that $C$ is a cycle of the permutation $w \in \symm_r$.
Multiplying both sides of Equation~\eqref{path-to-array-one} by $\varepsilon \cdot x^\delta$ yields
\begin{equation*}
\vec{p}_{\mu}(x_1, \dots, x_N) \times
\varepsilon \cdot x^\delta  = 
\varepsilon \cdot \left(
\vec{p}_{\mu}(x_1, \dots, x_N) \times
x^\delta  \right)  = 
\varepsilon \cdot \left( \sum_{w \in \symm_r} \prod_{C \in w} p_{\mu_C}(x_1, \dots, x_N) \times x^{\delta} \right).
\end{equation*}
Consider the monomial expansion of 
\begin{equation}
\label{eqn:path-to-array-two}
 \sum_{w \, \in \, \symm_r} 
 \prod_{C \, \in \, w} p_{\mu_C}(x_1, \dots, x_N) \times x^{\delta} = 
  \sum_{w \, \in \, \symm_r} \prod_{C \, \in \, w} (x_1^{\mu_C} + \cdots + x_N^{\mu_C}) \times x^{\delta}.
 \end{equation}
 A typical term in this expansion is obtained by first selecting a permutation $w \in \symm_r$, then assigning the cycles $C$ of $w$
 to the $N$ exponents of $x_1, \dots, x_N$ (where a given exponent can get no, one, or multiple cycles), and finally 
 multiplying by $x^{\delta}$. Applying  $\varepsilon$ to both sides of 
 Equation~\eqref{eqn:path-to-array-two} completes the proof.
\end{proof}

Each term $\varepsilon \cdot x^{\wt(\omega)}$ on the right-hand side in Lemma~\ref{lem:path-to-array}
 is 0 when $\wt(\omega)$ has repeated entries, 
 or $\pm$ some alternant 
 $a_{\lambda(\omega)}(x_1, \dots, x_N)$
 where $\lambda(\omega)$ is a partition.
Dividing by $\varepsilon \cdot x^\delta$ and taking the limit as $N \rightarrow \infty$
gives an $s$-expansion of $\vec{p}_\mu$, but this expansion is  too large and
involves too much cancellation to be of much use.
Roughly speaking, the $s$-expansion
coming from Lemma~\ref{lem:path-to-array} arises
by computing the $p$-expansion of $\vec{p}_\mu$,
and then applying classical Murnaghan--Nakayama
to each term.

\subsection{A sign-reversing involution}
This subsection gives a sign-reversing involution $\iota$ on word arrays
which removes a large part (but not all) of the cancellation
in Lemma~\ref{lem:path-to-array}. The involution $\iota$ will act by swapping `unstable pairs' in word arrays.  

Let $\omega = (w_1 \mid \cdots \mid w_N)$ 
be a $\mu$-word array. An {\em unstable pair} in $\omega$ consists of two prefix-suffix factorizations
\begin{equation}
w_i = u_i v_i \quad \text{and} \quad w_j = u_j v_j 
\end{equation}
of the words $w_i$ and $w_j$
at distinct positions $1 \leq i < j \leq N$ such that we have the equality
\begin{equation}
\label{unstable-equality}
\mu_{v_i} - i = \mu_{v_j} - j
\end{equation}
involving the suffixes of these factorizations.
A word array $\omega$ can admit unstable pairs at multiple pairs of positions $i < j$, and even at the same pair of positions
there could be more than one factorization
$w_i = u_i v_i$ and $w_j = u_j v_j$ of the words $w_i$ and $w_j$  witnessing an unstable pair. 
It is possible for either or both of the prefixes $u_i$ and $u_j$ in an unstable pair to be empty, but at least one of the suffixes
$v_i$ or $v_j$ must be nonempty.
We define the {\em score}
of an unstable pair to be the common value $\mu_{v_i} - i = \mu_{v_j} - j$ of \eqref{unstable-equality}.

\begin{example}
\label{ex:unstable-score-example}
Let $\mu = (\mu_1, \dots, \mu_9) = (3,3,3,3,2,2,2,1,1)$ and $N = 6$.  We have the standard $\mu$-word array
\begin{equation*}
\omega = (w_1 \mid w_2 \mid w_3 \mid w_4 \mid w_5 \mid w_6) = ( \varnothing \mid \varnothing \mid 3 \, 7 \mid 8 \, 1 \, 4 \mid \varnothing \mid 9 \, 5 \, 2 \, 6).
\end{equation*}
The ten unstable pairs in $\omega$ correspond to the prefix-suffix factorizations
\begin{equation*}
(w_1, w_3) = ( \varnothing , 3 \, \cdot \, 7),  \quad  
(w_1, w_4) = ( \varnothing, 8 \, 1 \, \cdot \, 4), \quad 
(w_1, w_6) = (\varnothing ,  9 \, \cdot \,  5 \, 2 \, 6),  
\end{equation*}
\begin{equation*}
(w_3, w_4) = ( 3  \, \cdot \,  7 , 8 \, 1 \, \cdot \, 4),  \quad 
(w_3, w_4) = ( \cdot \, 3 \,  7 , 8 \, \cdot \, 1 \, 4),  \quad 
(w_3, w_6) = (3 \cdot 7, 9 \, 5 \cdot  2 \, 6),
\end{equation*}
\begin{equation*}
(w_3, w_6) = ( \cdot \,  3 \, 7 , \cdot \, 9 \, 5 \, 2 \, 6),   \quad 
(w_4, w_6) = ( 8 \, 1 \, 4 \, \cdot \, , 9 \, 5 \, 2 \, \cdot \, 6), \quad 
(w_4, w_6) = (8 \, 1 \, \cdot \, 4, 9 \, 5 \, \cdot \, 2 \, 6),
\end{equation*}
\begin{equation*}
\text{ and } \quad 
(w_4, w_6) = (8 \, \cdot \, 1 \, 4, \, \cdot \, 9 \, 5 \, 2 \, 6).
\end{equation*}
In reading order from the top left, the scores of these unstable pairs are 
\begin{equation*}
-1, -1, -1, -1, 2, -1, 2, -4, -1, \text{ and } 2.
\end{equation*}
\end{example}

A $\mu$-word array $\omega = (w_1 \mid \cdots \mid w_N)$ is {\em stable} if $\omega$ has no unstable pairs.
Stability is sufficient to detect when  the monomial $x^{\wt(\omega)}$ is annihilated by the antisymmetrizer $\varepsilon$.

\begin{lemma}
\label{lem:collision-instability}
Let $\mu$ be a partition and let
 $\omega = (w_1 \mid \cdots \mid w_N)$ be a $\mu$-word array with weight sequence
 $\wt(\omega) = (\wt(\omega)_1, \dots, \wt(\omega)_N)$.  If there exist $i < j$ such that $\wt(\omega)_i = \wt(\omega)_j$
 then $\omega$ is unstable.
\end{lemma}

\begin{proof}
This instability is witnessed by the prefix-suffix
 factorizations $w_i = u_i v_i$ and $w_j = u_j v_j$ where $u_i = u_j = \varnothing$ are empty prefixes.
\end{proof}

Lemma~\ref{lem:collision-instability} says that any individual term in the expansion
\begin{equation*}
 \vec{p}_{\mu}(x_1, \dots, x_N) \times \varepsilon \cdot x^\delta = \sum_{\text{$\omega$ standard}} \varepsilon \cdot x^{\wt(\omega)}
 \end{equation*}
of Lemma~\ref{lem:path-to-array} which vanishes corresponds to an unstable
word array.  We will use stability to cancel more terms of this expansion pairwise.
Before doing so, we record
the following hereditary property of stable word arrays which will be crucial for our inductive arguments.

\begin{lemma}
\label{lem:hereditary-stability}
Let $\mu$ be a partition and let $\omega = (w_1 \mid \cdots \mid w_N)$ be a stable $\mu$-word array.  
For each position $i$, let $w_i = u_i v_i$ be a prefix-suffix factorization of the word $w_i$.  
The  $\mu$-word array
$(v_1 \mid \cdots \mid v_N)$ formed by the suffices $v_i$ of the words $w_i$ is also stable.
\end{lemma}

\begin{proof}
The suffix  array $(v_1 \mid \cdots \mid v_N)$ is stable because
 the condition \eqref{unstable-equality} defining unstable pairs
depends only on the suffixes of the words $w_1, \dots, w_N$.
\end{proof}

We define a swapping operation on unstable pairs as follows.
Consider an unstable $\mu$-word array $\omega = ( w_1 \mid \cdots \mid w_N)$ at positions $1 \leq i < j \leq N$ 
with prefix-suffix factorizations 
$w_i = u_i v_i$ and $w_j = u_j v_j$. The {\em swap} $\sigma(\omega)$ of $\omega$ at this unstable pair is the 
$\mu$-word array
\begin{equation}
\omega = (w_1 \mid \cdots \mid w_i \mid \cdots \mid w_j \mid \cdots \mid w_N) \quad \leadsto \quad
\sigma(\omega) := (w_1 \mid \cdots \mid w'_i \mid \cdots \mid w'_j \mid \cdots \mid w_N)
\end{equation}
where 
\begin{equation}
w'_i := u_j v_i \quad \text{and} \quad w'_j := u_i v_j. 
\end{equation}
That is, the swap $\sigma(\omega)$ is obtained from $\omega$ by interchanging the 
prefixes $u_i$ and $u_j$.
It is possible for $\sigma$ to leave the 
word array $\omega$ unchanged:
if both prefixes $u_i$ and $u_j$ are empty, we have $\sigma(\omega) = \omega$.
The instability condition \eqref{unstable-equality} implies
\begin{equation}
\mu_{w_i} + N - i = \mu_{u_i} + \mu_{v_i} + N - i = \mu_{u_i} + \mu_{v_j} + N - j = \mu_{w'_j} + N - j
\end{equation}
and
\begin{equation}
\mu_{w_j} + N - j = \mu_{u_j} + \mu_{v_j} + N - j = \mu_{u_j} + \mu_{v_i} + N - i = \mu_{w'_i} + N - i
\end{equation}
We record this crucial property of swapping in the following observation.

\begin{observation}
\label{obs:swap-reverse}
Let $\omega = (w_1 \mid \cdots \mid w_N)$ be a 
$\mu$-word array which is unstable
at positions $1 \leq i < j \leq n$ with respect to some prefix-suffix factorizations
$w_i = u_i v_i$ and $w_j = u_j v_j$
 of $w_i$ and $w_j$.  Let $\sigma(\omega)$ be the swapped $\mu$-word array.

The weight sequence $\wt(\sigma(\omega)) = (\wt(\sigma(\omega))_1, \dots, \wt(\sigma(\omega))_N)$ of the swapped word array $\sigma(\omega)$
 is obtained from the weight sequence
$\wt(\omega) = (\wt(\omega)_1, \dots, \wt(\omega)_N)$ of the original word array $\omega$ 
by interchanging the entries in positions $i$ and $j$.
\end{observation}

Observation~\ref{obs:swap-reverse}
 implies that
\begin{equation}
\label{eqn:basic-swapping-cancellation}
\varepsilon \cdot x^{\wt(\omega)} + \varepsilon \cdot x^{\wt(\sigma(\omega))} = 0
\end{equation}
so that (for standard arrays) the terms corresponding to $\omega$ and $\sigma(\omega)$ appearing in 
Lemma~\ref{lem:path-to-array} cancel.
If both prefixes in the swap $\sigma$ are empty so that $\sigma(\omega) = \omega$, the corresponding entries of $\wt(\omega) = \wt(\sigma(\omega))$ coincide
and
Equation~\eqref{eqn:basic-swapping-cancellation} reads $0 + 0 = 0$.

As Example~\ref{ex:unstable-score-example} illustrates, an unstable word array $\omega$ can have multiple unstable pairs.
Our sign-reversing involution $\iota$ will act on unstable arrays by making a  swap
$\iota: \omega \mapsto \sigma(\omega)$ at a strategically chosen unstable pair of $\omega$.
Since the positions and prefix-suffix factorizations of unstable pairs in unstable arrays can be affected by swapping,
care must be taken to ensure the map $\iota$ is actually an involution. 
The score of an unstable pair was introduced to solve this problem.

In order to state the next result, we need one more definition. If $\mu = (\mu_1, \dots, \mu_r)$
is a partition with $r$ parts, the {\em content} of a $\mu$-word
array $\omega = (w_1 \mid \cdots \mid w_N)$ is the sequence $\content(\omega) = (\content(\omega)_1, \dots, \content(\omega)_r)$
where $\content(\omega)_i$ counts the total number of copies of the letter $i$ among the words $w_1, \dots, w_N$.
In particular, a standard word array is simply a word array of content $(1^r)$.

\begin{lemma}
\label{lem:sign-reversing}
Let $\mu = (\mu_1, \dots, \mu_r)$ be a partition with $r$ parts and let $\gamma = (\gamma_1, \dots, \gamma_r)$ be a length
$r$ sequence of nonnegative integers.  We have
\begin{equation}
\sum_{\omega}  \varepsilon \cdot x^{\wt(\omega)} = \sum_{\text{$\omega$ {\em stable}}} \varepsilon \cdot x^{\wt(\omega)}
\end{equation}
where the sum on the left-hand side is over all $\mu$-word arrays of length $N$ and content $\gamma$ and the sum
on the right-hand side is over all {\em stable} $\mu$-word arrays of length $N$ and content $\gamma$.
\end{lemma}

In particular, Lemmas~\ref{lem:path-to-array} and \ref{lem:sign-reversing} imply that $\vec{p}_{\mu} \times \varepsilon \cdot x^\delta$ may be expressed
\begin{equation}
\label{eqn:sign-standard-case}
\vec{p}_{\mu}(x_1, \dots, x_N) \times 
\varepsilon \cdot x^\delta = \sum_{\substack{ \text{$\omega$ standard} \\ \text{$\omega$ stable}}} 
\varepsilon \cdot x^{\wt(\omega)}.
\end{equation}
in terms of $\mu$-word arrays $\omega = (w_1 \mid \cdots \mid w_N)$ which are both standard and stable. Upon division by $\varepsilon \cdot x^\delta$,
Equation~\eqref{eqn:sign-standard-case} expresses the $s$-expansion of $\vec{p}_{\mu}$ in the most efficient way known to the authors.
In the next section we give a combinatorially amenable avatar of standard stable arrays.

\begin{proof}
Let $\WWW$ be the family of all $\mu$-word arrays $\omega = ( w_1 \mid \cdots \mid w_N)$ of length $N$ and content $\gamma$. Let 
 $\UUU \subseteq \WWW$ be the subfamily unstable arrays.
 We define a function $\iota: \WWW \rightarrow \WWW$ as follows.
 If $\omega \in \WWW - \UUU$ is stable, we set $\iota(\omega) := \omega$.
 The definition of $\iota(\omega)$ for $\omega \in \UUU$ unstable requires a fact about unstable pairs of minimal score.

Let $\omega = (w_1 \mid \cdots \mid w_N) \in \UUU$. There may be multiple factorization pairs $(w_i = u_i v_i, w_j = u_j v_j)$ 
and multiple pairs of positions which are unstable in $\omega$.  Choose one such unstable pair such that the score
$m := \mu_{v_i} - i = \mu_{v_j} - j$ is as small as possible. Let 
$\sigma(\omega) = (w_1 \mid \cdots \mid w'_i \mid \cdots \mid w'_j \mid \cdots \mid w_N)$ be the result of swapping this pair,
so that $w'_i = u_j v_i$ and $w'_j = u_i v_j$.

{\bf Claim:} {\em The swapped array $\sigma(\omega) = (w_1 \mid \cdots \mid w'_i \mid \cdots \mid w'_j \mid \cdots \mid w_N)$ 
has an unstable pair at positions $i < j$ of score $m$ given by the factorizations $w'_i = u_j v_i$ and $w'_j = u_i v_j$.
Furthermore, the array $\sigma(\omega)$ has no unstable pairs 
of score $< m$. Finally, the unstable pairs in the array $\sigma(\omega)$ of score $m$ occur at the same positions as the unstable pairs of score $m$ in 
$\sigma$.}

The first part of the claim is immediate from the defining condition \eqref{unstable-equality} of unstable pairs. For the second part,
suppose $\sigma(\omega)$ admitted an unstable pair of score $< m$. By the minimality of $m$, this unstable pair must involve 
one of the positions $i, j$ and some other position $k \neq i,j$.  Without loss of generality, assume that this unstable pair
involves the positions $i$ and $k$ and  corresponds to the
prefix-suffix factorizations $w'_i = u'_i v'_i$ and $w_k = u_k v_k$ of the words in at positions $i$ and $k$ of $\sigma(\omega)$.
By assumption the score of this pair is
\begin{equation}
\mu_{v'_i} - i < m = \mu_{v_i} - i 
\end{equation}
which implies that $v'_i$ is a (proper) suffix of $v_i$. If we write $v_i = t_i v'_i$ then $w_i = u_i t_i v'_i$ and
 $v'_i$ is also a suffix of $w_i$. The factorizations $w_i = (u_i t_i) v'_i$ and $w_k = u_k v_k$ witness an unstable pair in $\omega$
 of score $< m$, which contradicts the choice of $m$.  
 For the last part of the claim, if there is an unstable pair in $w$ of score $m$ at positions $i$ and $k$ with corresponding factorization
 $w_k = u_k v_k$, we have 
 \begin{equation}
 \mu_{v_k} - k = m = \mu_{v_i} - i = \mu_{v_j} - j,
 \end{equation}
 which implies that $w$ also has an unstable pair of score $m$ at positions $j$ and $k$.  Since the suffices $v_i, v_j,$ and $v_k$ remain unchanged
 by the operation $\omega \mapsto \sigma(\omega)$, we see that $\sigma(\omega)$ also has unstable pairs of score $m$ at positions $(i,k)$ and $(j,k)$.
 This completes the proof of the claim.
 
 With the claim in hand, the rest of the proof is straightforward. Given $\omega \in \UUU$, define 
 $\iota(\omega) := \sigma(\omega)$ where $\sigma$ swaps the unstable pair of minimal score $m$
 (if there are multiple unstable pairs of score $m$, let $\sigma$ swap the unstable pair of score $m$ at positions 
 $1 \leq i < j \leq N$ which are lexicographically maximal).   Since $\iota$ fixes any stable array, we have a function $\iota: \WWW \rightarrow \WWW$.
 The claim guarantees that $\iota(\iota(\omega)) = \omega$
 for all $\omega \in \UUU$ so that $\iota$ is an involution.
 The discussion following Lemma~\ref{lem:hereditary-stability} shows that
 \begin{equation}
 \sum_{\omega \, \in \, \UUU} \varepsilon \cdot x^{\wt(\omega)} = 0
 \end{equation}
 which completes the proof.
\end{proof}

To show how the involution $\iota$ in the proof of Lemma~\ref{lem:sign-reversing} works, we consider the unstable word
array of Example~\ref{ex:unstable-score-example}.

\begin{example}
\label{ex:involution-example}
Consider applying $\iota$ to the $\mu$-word array $\omega = (w_1 \mid w_2 \mid w_3 \mid w_4 \mid w_5 \mid w_6)$ 
of Example~\ref{ex:unstable-score-example}.
The unstable pair of lowest score $m = -4$  is $(w_4, w_6) = ( 8 \, 1 \, 4 \, \cdot \, , 9 \, 5 \, 2 \, \cdot \, 6)$. Performing a swap at this unstable pair 
interchanges these two prefixes. The resulting array is 
\begin{equation*}
\iota(\omega) = ( \varnothing \mid \varnothing \mid 3 \, 7 \mid 9 \, 5 \, 2 \mid \varnothing \mid 8 \, 1 \, 4 \, 6).
\end{equation*}
The factorizations $(9 \, 5 \, 2 \, \cdot \, , 8 \, 1 \, 4 \, \cdot \, 6)$ yield an unstable pair in $\iota(\omega)$,
also of score $m = -4$. 
\end{example}

\subsection{Monotonic ribbon tilings}
Lemma~\ref{lem:sign-reversing} is equivalent to our expansion of $\vec{p}_{\mu}$ in the $s$-basis, but we will need 
this expansion in a more combinatorially transparent form.
To achieve this, we consider ribbon tilings of Young diagrams that satisfy a monotone condition on their tails.

\begin{defn}
\label{def:monotonic-definition}
A {\em monotonic ribbon tiling} $T$ of a shape $\lambda$ is a disjoint union decomposition \[\lambda = \xi^{(1)} \sqcup \cdots \sqcup \xi^{(r)}\]
of the Young diagram of $\lambda$ into ribbons $\xi^{(1)}, \dots, \xi^{(r)}$ such that 
\begin{itemize}
\item
the tails of $\xi^{(1)}, \dots, \xi^{(r)}$ occupy distinct columns
$c_1 < \cdots < c_r$ of $\lambda$, and 
\item
each initial union $\xi^{(1)} \sqcup \cdots \sqcup \xi^{(i)}$ of ribbons is the Young diagram of a partition for $0 \leq i \leq r$.
\end{itemize}
\end{defn}

Definition~\ref{def:monotonic-definition} forces the tails of $\xi^{(1)}, \dots, \xi^{(r)}$ to lie on the southern boundary of $\lambda$. In particular, if $b_i$ is the row containing the tail of $\xi^{(r)}$, we have $b_1 \geq \cdots \geq b_r$. 
This is the origin of the term `monotonic'; the ribbons have a left-to-right total order and their tails become shallower when read from left to right.  
We define the {\em tail depth sequence} of $T$ to be 
\begin{equation}
    \depth(T) := (b_1 \geq \cdots \geq b_r)
\end{equation}
where $b_i$ is the row occupied by the tail of the $i^{th}$ ribbon of $T$. 

\begin{example}
    \label{ex:motonoic-and-not}
Figure~\ref{fig:monotonic-and-not} shows two tilings of the Young diagram of $\lambda = (10,9,4,2,2)$ by ribbons. The tiling on the left is monotonic while the tiling on the right is not. The tail depth sequence of the tiling on the left is 
$\depth(T) = (5,5,3,2,2,1)$.
\end{example}

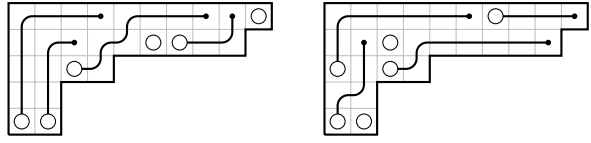
\begin{figure}
\begin{center}
\begin{tikzpicture}[scale = 0.35]

  \begin{scope}
    \clip (0,0) -| (2,2) -| (4,3) -| (9,4) -| (10,5) -| (0,0);
    \draw [color=black!25] (0,0) grid (10,5);
  \end{scope}

  \draw [thick] (0,0) -| (2,2) -| (4,3) -| (9,4) -| (10,5) -| (0,0);

  \draw [thick, rounded corners] (0.5,0.5) |- (3.5,4.5);
  \draw [color=black,fill=black,thick] (3.5,4.5) circle (.4ex);
  \node [draw, circle, fill = white, inner sep = 2pt] at (0.5,0.5) { };
  
  \draw [thick, rounded corners] (1.5,0.5) |- (2.5,3.5);
  \draw [color=black,fill=black,thick] (2.5,3.5) circle (.4ex);
  \node [draw, circle, fill = white, inner sep = 2pt] at (1.5,0.5) { };
  
  \draw [thick, rounded corners] (2.5,2.5) -| (3.5,3.5) -| (4.5,4.5) -- (7.5,4.5);
  \draw [color=black,fill=black,thick] (7.5,4.5) circle (.4ex);
  \node [draw, circle, fill = white, inner sep = 2pt] at (2.5,2.5) { };

    \node [draw, circle, fill = white, inner sep = 2pt] at (5.5,3.5) { };
    
  \draw [thick, rounded corners] (6.5,3.5) -| (8.5,4.5);
  \draw [color=black,fill=black,thick] (8.5,4.5) circle (.4ex);
  \node [draw, circle, fill = white, inner sep = 2pt] at (6.5,3.5) { };   
  
      \node [draw, circle, fill = white, inner sep = 2pt] at (9.5,4.5) { };

   \begin{scope}
    \clip (12,0) -| (14,2) -| (16,3) -| (21,4) -| (22,5) -| (12,0);
    \draw [color=black!25] (12,0) grid (22,5);
  \end{scope}

   \draw [thick] (12,0) -| (14,2) -| (16,3) -| (21,4) -| (22,5) -| (12,0);    
  
  \draw [thick, rounded corners] (12.5,2.5) |- (17.5,4.5);
  \draw [color=black,fill=black,thick] (17.5,4.5) circle (.4ex);
  \node [draw, circle, fill = white, inner sep = 2pt] at (12.5,2.5) { };   
  
  \draw [thick, rounded corners] (12.5,0.5) |- (13.5,1.5) -- (13.5,3.5);
  \draw [color=black,fill=black,thick] (13.5,3.5) circle (.4ex);
  \node [draw, circle, fill = white, inner sep = 2pt] at (12.5,0.5) { };   
  
  \draw [thick, rounded corners] (14.5,2.5) -| (15.5,3.5) -- (20.5,3.5);
  \draw [color=black,fill=black,thick] (20.5,3.5) circle (.4ex);
  \node [draw, circle, fill = white, inner sep = 2pt] at (14.5,2.5) { };   
  
  \draw [thick, rounded corners]  (18.5,4.5) -- (21.5,4.5);
  \draw [color=black,fill=black,thick] (21.5,4.5) circle (.4ex);
  \node [draw, circle, fill = white, inner sep = 2pt] at (18.5,4.5) { };

   \node [draw, circle, fill = white, inner sep = 2pt] at (13.5,0.5) { };

    \node [draw, circle, fill = white, inner sep = 2pt] at (14.5,3.5) { };

\end{tikzpicture}
\end{center}
\caption{A monotonic ribbon tiling (left) and a tiling by ribbons which is not monotonic (right).}
\label{fig:monotonic-and-not}
\end{figure}

\begin{figure}
\begin{center}
\begin{tikzpicture}[scale = 0.3]

\begin{scope}
   \clip (0,0) -| (6,1) -| (0,0);
    \draw [color=black!25] (0,0) grid (6,1);
\end{scope}

\draw [thick] (0,0) -| (6,1) -| (0,0);

  \draw [thick, rounded corners]  (0.5,0.5) -- (2.5,0.5);
  \draw [color=black,fill=black,thick] (2.5,0.5) circle (.4ex);
  \node [draw, circle, fill = white, inner sep = 1.2pt] at (0.5,0.5) { };    
  
  \draw [thick, rounded corners]  (3.5,0.5) -- (4.5,0.5);
  \draw [color=black,fill=black,thick] (4.5,0.5) circle (.4ex);
  \node [draw, circle, fill = white, inner sep = 1.2pt] at (3.5,0.5) { };    
  
    \node [draw, circle, fill = white, inner sep = 1.2pt] at (5.5,0.5) { };    
    
    \node at (3,-1) {$+1$};

  \begin{scope}
   \clip (8,0) -| (14,1) -| (8,0);
    \draw [color=black!25] (8,0) grid (14,1);
\end{scope}

\draw [thick] (8,0) -| (14,1) -| (8,0);

  \draw [thick, rounded corners]  (8.5,0.5) -- (10.5,0.5);
  \draw [color=black,fill=black,thick] (10.5,0.5) circle (.4ex);
  \node [draw, circle, fill = white, inner sep = 1.2pt] at (8.5,0.5) { };    
  
  \draw [thick, rounded corners]  (12.5,0.5) -- (13.5,0.5);
  \draw [color=black,fill=black,thick] (13.5,0.5) circle (.4ex);
  \node [draw, circle, fill = white, inner sep = 1.2pt] at (12.5,0.5) { };    
  
    \node [draw, circle, fill = white, inner sep = 1.2pt] at (11.5,0.5) { };    
    
    \node at (11,-1) {$+1$};

  \begin{scope}
   \clip (16,0) -| (22,1) -| (16,0);
    \draw [color=black!25] (16,0) grid (22,1);
\end{scope}

\draw [thick] (16,0) -| (22,1) -| (16,0);

  \draw [thick, rounded corners]  (16.5,0.5) -- (17.5,0.5);
  \draw [color=black,fill=black,thick] (17.5,0.5) circle (.4ex);
  \node [draw, circle, fill = white, inner sep = 1.2pt] at (16.5,0.5) { };    
  
  \draw [thick, rounded corners]  (18.5,0.5) -- (20.5,0.5);
  \draw [color=black,fill=black,thick] (20.5,0.5) circle (.4ex);
  \node [draw, circle, fill = white, inner sep = 1.2pt] at (18.5,0.5) { };    
  
    \node [draw, circle, fill = white, inner sep = 1.2pt] at (21.5,0.5) { };    
    
    \node at (19,-1) {$+1$};

 \begin{scope}
   \clip (24,0) -| (30,1) -| (24,0);
    \draw [color=black!25] (24,0) grid (30,1);
\end{scope}

\draw [thick] (24,0) -| (30,1) -| (24,0);

  \draw [thick, rounded corners]  (24.5,0.5) -- (25.5,0.5);
  \draw [color=black,fill=black,thick] (25.5,0.5) circle (.4ex);
  \node [draw, circle, fill = white, inner sep = 1.2pt] at (24.5,0.5) { };    
  
  \draw [thick, rounded corners]  (27.5,0.5) -- (29.5,0.5);
  \draw [color=black,fill=black,thick] (29.5,0.5) circle (.4ex);
  \node [draw, circle, fill = white, inner sep = 1.2pt] at (27.5,0.5) { };    
  
    \node [draw, circle, fill = white, inner sep = 1.2pt] at (26.5,0.5) { };    
    
    \node at (27,-1) {$+1$};     
    
 \begin{scope}
   \clip (32,0) -| (38,1) -| (32,0);
    \draw [color=black!25] (32,0) grid (38,1);
\end{scope}

\draw [thick] (32,0) -| (38,1) -| (32,0);

  \draw [thick, rounded corners]  (33.5,0.5) -- (35.5,0.5);
  \draw [color=black,fill=black,thick] (35.5,0.5) circle (.4ex);
  \node [draw, circle, fill = white, inner sep = 1.2pt] at (33.5,0.5) { };    
  
  \draw [thick, rounded corners]  (36.5,0.5) -- (37.5,0.5);
  \draw [color=black,fill=black,thick] (37.5,0.5) circle (.4ex);
  \node [draw, circle, fill = white, inner sep = 1.2pt] at (36.5,0.5) { };    
  
    \node [draw, circle, fill = white, inner sep = 1.2pt] at (32.5,0.5) { };    
    
    \node at (35,-1) {$+1$};

 \begin{scope}
   \clip (40,0) -| (46,1) -| (40,0);
    \draw [color=black!25] (40,0) grid (46,1);
\end{scope}

\draw [thick] (40,0) -| (46,1) -| (40,0);

  \draw [thick, rounded corners]  (41.5,0.5) -- (42.5,0.5);
  \draw [color=black,fill=black,thick] (42.5,0.5) circle (.4ex);
  \node [draw, circle, fill = white, inner sep = 1.2pt] at (41.5,0.5) { };    
  
  \draw [thick, rounded corners]  (43.5,0.5) -- (45.5,0.5);
  \draw [color=black,fill=black,thick] (45.5,0.5) circle (.4ex);
  \node [draw, circle, fill = white, inner sep = 1.2pt] at (43.5,0.5) { };    
  
    \node [draw, circle, fill = white, inner sep = 1.2pt] at (40.5,0.5) { };    
    
    \node at (43,-1) {$+1$};

\end{tikzpicture}
\end{center}

\begin{center}
\begin{tikzpicture}[scale = 0.3]

\begin{scope}
   \clip (0,0) -| (1,1) -| (5,2) -| (0,0);
    \draw [color=black!25] (0,0) grid (5,2);
\end{scope}

\draw [thick] (0,0) -| (1,1) -| (5,2) -| (0,0);

\node at (2.5,-0.5) {$-1$};

  \draw [thick, rounded corners]  (0.5,0.5) |- (1.5,1.5);
  \draw [color=black,fill=black,thick] (1.5,1.5) circle (.4ex);
  \node [draw, circle, fill = white, inner sep = 1.2pt] at (0.5,0.5) { };

    \draw [thick, rounded corners]  (2.5,1.5) -- (3.5,1.5);
  \draw [color=black,fill=black,thick] (3.5,1.5) circle (.4ex);
  \node [draw, circle, fill = white, inner sep = 1.2pt] at (2.5,1.5) { };  
  
    \node [draw, circle, fill = white, inner sep = 1.2pt] at (4.5,1.5) { };

\begin{scope}
   \clip (10,0) -| (11,1) -| (15,2) -| (10,0);
    \draw [color=black!25] (10,0) grid (15,2);
\end{scope}

\draw [thick] (10,0) -| (11,1) -| (15,2) -| (10,0);

\node at (12.5,-0.5) {$-1$};

  \draw [thick, rounded corners]  (10.5,0.5) |- (11.5,1.5);
  \draw [color=black,fill=black,thick] (11.5,1.5) circle (.4ex);
  \node [draw, circle, fill = white, inner sep = 1.2pt] at (10.5,0.5) { };  
  
      \node [draw, circle, fill = white, inner sep = 1.2pt] at (12.5,1.5) { };

      \draw [thick, rounded corners]  (13.5,1.5) -- (14.5,1.5);
  \draw [color=black,fill=black,thick] (14.5,1.5) circle (.4ex);
  \node [draw, circle, fill = white, inner sep = 1.2pt] at (13.5,1.5) { };

\begin{scope}
   \clip (20,0) -| (21,1) -| (25,2) -| (20,0);
    \draw [color=black!25] (20,0) grid (25,2);
\end{scope}

\draw [thick] (20,0) -| (21,1) -| (25,2) -| (20,0);

\node at (22.5,-0.5) {$-1$};

  \draw [thick, rounded corners]  (20.5,0.5) -- (20.5,1.5);
  \draw [color=black,fill=black,thick] (20.5,1.5) circle (.4ex);
  \node [draw, circle, fill = white, inner sep = 1.2pt] at (20.5,0.5) { };

   \draw [thick, rounded corners]  (21.5,1.5) -- (23.5,1.5);
  \draw [color=black,fill=black,thick] (23.5,1.5) circle (.4ex);
  \node [draw, circle, fill = white, inner sep = 1.2pt] at (21.5,1.5) { };  
  
  \node [draw, circle, fill = white, inner sep = 1.2pt] at (24.5,1.5) { };

\begin{scope}
   \clip (30,0) -| (31,1) -| (35,2) -| (30,0);
    \draw [color=black!25] (30,0) grid (35,2);
\end{scope}

\draw [thick] (30,0) -| (31,1) -| (35,2) -| (30,0);

\node at (32.5,-0.5) {$-1$};

  \draw [thick, rounded corners]  (30.5,0.5) -- (30.5,1.5);
  \draw [color=black,fill=black,thick] (30.5,1.5) circle (.4ex);
  \node [draw, circle, fill = white, inner sep = 1.2pt] at (30.5,0.5) { };

   \draw [thick, rounded corners]  (32.5,1.5) -- (34.5,1.5);
  \draw [color=black,fill=black,thick] (34.5,1.5) circle (.4ex);
  \node [draw, circle, fill = white, inner sep = 1.2pt] at (32.5,1.5) { };  
  
    \node [draw, circle, fill = white, inner sep = 1.2pt] at (31.5,1.5) { };  

\end{tikzpicture}
\end{center}

\begin{center}
\begin{tikzpicture}[scale = 0.3]

\begin{scope}
\clip (0,0) -| (2,1) -| (4,2) -| (0,0);
\draw [color=black!25] (0,0) grid (4,2);
\end{scope}

\draw [thick] (0,0) -| (2,1) -| (4,2) -| (0,0);

\node at (2,-1) {$-1$};

  \draw [thick, rounded corners]  (0.5,0.5) |- (1.5,1.5);
  \draw [color=black,fill=black,thick] (1.5,1.5) circle (.4ex);
  \node [draw, circle, fill = white, inner sep = 1.2pt] at (0.5,0.5) { };

    \draw [thick, rounded corners]  (2.5,1.5) -- (3.5,1.5);
  \draw [color=black,fill=black,thick] (3.5,1.5) circle (.4ex);
  \node [draw, circle, fill = white, inner sep = 1.2pt] at (2.5,1.5) { };  
  
    \node [draw, circle, fill = white, inner sep = 1.2pt] at (1.5,0.5) { };

\begin{scope}
\clip (6,0) -| (8,1) -| (10,2) -| (6,0);
\draw [color=black!25] (6,0) grid (10,2);
\end{scope}

\draw [thick] (6,0) -| (8,1) -| (10,2) -| (6,0);

\node at (8,-1) {$+1$};

  \draw [thick, rounded corners]  (6.5,0.5) -- (6.5,1.5);
  \draw [color=black,fill=black,thick] (6.5,1.5) circle (.4ex);
  \node [draw, circle, fill = white, inner sep = 1.2pt] at (6.5,0.5) { };

  \draw [thick, rounded corners]  (7.5,0.5) |- (8.5,1.5);
  \draw [color=black,fill=black,thick] (8.5,1.5) circle (.4ex);
  \node [draw, circle, fill = white, inner sep = 1.2pt] at (7.5,0.5) { };

 \node [draw, circle, fill = white, inner sep = 1.2pt] at (9.5,1.5) { };

\begin{scope}
\clip (14,0) -| (17,2) -| (14,0);
\draw [color=black!25] (14,0) grid (17,2);
\end{scope}

\draw [thick] (14,0) -| (17,2) -| (14,0);

\node at (15.5,-1) {$+1$};

  \draw [thick, rounded corners]  (14.5,0.5) |- (15.5,1.5);
  \draw [color=black,fill=black,thick] (15.5,1.5) circle (.4ex);
  \node [draw, circle, fill = white, inner sep = 1.2pt] at (14.5,0.5) { };  
  
  \draw [thick, rounded corners]  (16.5,0.5) -- (16.5,1.5);
  \draw [color=black,fill=black,thick] (16.5,1.5) circle (.4ex);
  \node [draw, circle, fill = white, inner sep = 1.2pt] at (16.5,0.5) { };    
  
   \node [draw, circle, fill = white, inner sep = 1.2pt] at (15.5,0.5) { };

\begin{scope}
\clip (19,0) -| (22,2) -| (19,0);
\draw [color=black!25] (19,0) grid (22,2);
\end{scope}

\draw [thick] (19,0) -| (22,2) -| (19,0);

\node at (20.5,-1) {$+1$};

  \draw [thick, rounded corners]  (19.5,0.5) -- (19.5,1.5);
  \draw [color=black,fill=black,thick] (19.5,1.5) circle (.4ex);
  \node [draw, circle, fill = white, inner sep = 1.2pt] at (19.5,0.5) { };  
  
  \draw [thick, rounded corners]  (20.5,0.5) |- (21.5,1.5);
  \draw [color=black,fill=black,thick] (21.5,1.5) circle (.4ex);
  \node [draw, circle, fill = white, inner sep = 1.2pt] at (20.5,0.5) { };   
  
   \node [draw, circle, fill = white, inner sep = 1.2pt] at (21.5,0.5) { };

\begin{scope}
\clip (26,-1) -| (27,1) -| (30,2) -| (26,-1);
\draw [color=black!25] (26,-1) grid (30,2);
\end{scope}

\draw [thick] (26,-1) -| (27,1) -| (30,2) -| (26,-1);

\node at (28.5,-1) {$+1$};

  \draw [thick, rounded corners]  (26.5,-0.5) -- (26.5,1.5);
  \draw [color=black,fill=black,thick] (26.5,1.5) circle (.4ex);
  \node [draw, circle, fill = white, inner sep = 1.2pt] at (26.5,-0.5) { };

    \draw [thick, rounded corners]  (27.5,1.5) -- (28.5,1.5);
  \draw [color=black,fill=black,thick] (28.5,1.5) circle (.4ex);
  \node [draw, circle, fill = white, inner sep = 1.2pt] at (27.5,1.5) { };  
  
   \node [draw, circle, fill = white, inner sep = 1.2pt] at (29.5,1.5) { };

\begin{scope}
\clip (32,-1) -| (33,1) -| (36,2) -| (32,-1);
\draw [color=black!25] (32,-1) grid (36,2);
\end{scope}

\draw [thick]  (32,-1) -| (33,1) -| (36,2) -| (32,-1);

\node at (34.5,-1) {$+1$};

  \draw [thick, rounded corners]  (32.5,-0.5) -- (32.5,1.5);
  \draw [color=black,fill=black,thick] (32.5,1.5) circle (.4ex);
  \node [draw, circle, fill = white, inner sep = 1.2pt] at (32.5,-0.5) { };  
  
    \node [draw, circle, fill = white, inner sep = 1.2pt] at (33.5,1.5) { };  
    
  \draw [thick, rounded corners]  (34.5,1.5) -- (35.5,1.5);
  \draw [color=black,fill=black,thick] (35.5,1.5) circle (.4ex);
  \node [draw, circle, fill = white, inner sep = 1.2pt] at (34.5,1.5) { };

\begin{scope}
\clip (40,-1) -| (41,0) -| (42,1) -| (43,2) -| (40,-1);
\draw [color=black!25] (40,-1) grid (43,2);
\end{scope}

\draw [thick] (40,-1) -| (41,0) -| (42,1) -| (43,2) -| (40,-1);

\node at (42.5,-1) {$-1$};

  \draw [thick, rounded corners]  (40.5,-0.5) -- (40.5,1.5);
  \draw [color=black,fill=black,thick] (40.5,1.5) circle (.4ex);
  \node [draw, circle, fill = white, inner sep = 1.2pt] at (40.5,-0.5) { };  
  
  \draw [thick, rounded corners]  (41.5,0.5) -- (41.5,1.5);
  \draw [color=black,fill=black,thick] (41.5,1.5) circle (.4ex);
  \node [draw, circle, fill = white, inner sep = 1.2pt] at (41.5,0.5) { };

  \node [draw, circle, fill = white, inner sep = 1.2pt] at (42.5,1.5) { };

\end{tikzpicture}
\end{center}

\caption{The monotonic ribbon tilings used to calculate the Schur expansion of $\vec{p}_{321}$ using 
Theorem~\ref{thm:path-Murnaghan--Nakayama}.}
\label{fig:tilings}
\end{figure}
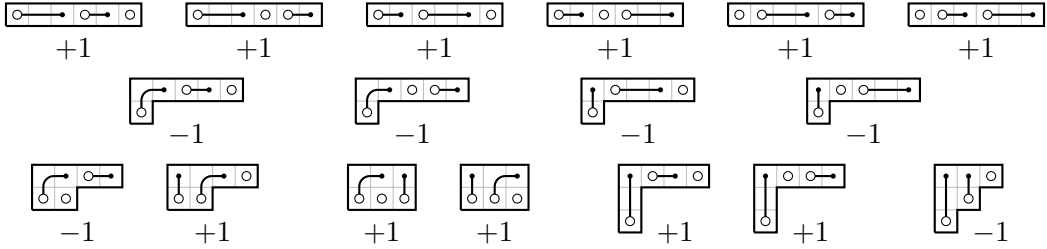

Let $T$ be a monotonic ribbon tiling of a partition $\lambda \vdash n$. We may attach various combinatorial objects to $T$ as in the case of classical ribbon tableaux. We write
$\shape(T) := \lambda$ for the partition tiled by $T$. The {\em sign} of $T$ is 
\begin{equation}
    \sign(T) := \prod_\xi \sign(\xi) = \prod_\xi (-1)^{\height(\xi)}
\end{equation}
where the product is over all ribbons $\xi$ in the tiling $T$. The {\em type} of the tiling $T$ is the composition
\begin{equation}
\type(T) = (k_1, k_2, \dots ) \models n 
\end{equation}
where $k_i$ is the size of the $i^{th}$ ribbon in $T$, read from left to right. The sequence $\sort(\type(T))$ obtained by sorting $\type(T)$ into weakly decreasing order is a partition of $n$. For $i \geq 1$, we let
\begin{equation}
    \rho_i(T) := \text{number of ribbons in $T$ of size $i$}
\end{equation}
and define
\begin{equation}
    \rho(T) := \sum_{i \, \geq \, 1} \rho_i(T) = \text{total number of ribbons in $T$}.
\end{equation}
If $\sort(\type(T)) = \mu \vdash n$, one has $\rho_i(T) = m_i(\mu)$ and $\rho(T) = \ell(\mu)$.

\begin{example}
    Let $T$ be the tiling on the left of Figure~\ref{fig:monotonic-and-not}. We have 
$\shape(T) = (10,9,4,2,2) \vdash 27$. Reading the ribbon sizes in $T$ from left to right, we have $\type(T) = (8,5,8,1,4,1) \models 27$ so that $\sort(\type(T)) = (8,8,5,4,1,1) \vdash 27$. The ribbon size multiplicities are 
\[\rho_1(T) = 2, \quad \rho_4(T) = 1, \quad \rho_5(T) = 1, \quad \rho_8(T) = 2\]
and $\rho_i(T) = 0$ for $i \neq 1,4,5,8$. We have $\rho(T) = 2 + 1 + 1 + 2 = 6$. Taking the signs of the ribbons in $T$ from left to right gives
\[\sign(T) = (+1)(-1)(+1)(+1)(-1)(+1) = +1.\]
\end{example}

A monotonic tiling is uniquely encoded by its type composition and depth sequence.

\begin{observation}
\label{obs:monotonic-characterization}
Let $\alpha$ be a composition with $r$ parts and let $b_1 \geq \cdots \geq b_r$ be a weakly decreasing sequence of $r$ positive integers. There is at most one monotonic ribbon tiling $T$ with $\type(T) = \alpha$ and $\depth(T) = (b_1 \geq \cdots \geq b_r)$.
\end{observation}

Not every pair of $\alpha$ and $b_1 \geq \cdots \geq b_r$ as in Observation~\ref{obs:monotonic-characterization} corresponds to a monotonic
ribbon tiling.
For example, the pair $\alpha = (5,3,3)$ and $(b_1, b_2, b_3) = (3,2,1)$ does not yield a monotonic ribbon tiling. Indeed, the tiling
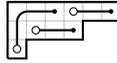
\begin{center}
\begin{tikzpicture}[scale = 0.25]

  \begin{scope}
    \clip (0,0) -| (1,1) -| (4,2) -| (6,3)  -| (0,0);
    \draw [color=black!25] (0,0) grid (6,3);
  \end{scope}

  \draw [thick] (0,0) -| (1,1) -| (4,2) -| (6,3)  -| (0,0);
  
  \draw [thick, rounded corners]  (0.5,0.5) |- (2.5,2.5);
  \draw [color=black,fill=black,thick] (2.5,2.5) circle (.4ex);
  \node [draw, circle, fill = white, inner sep = 1pt] at (0.5,0.5) { };

   \draw [thick, rounded corners]  (3.5,2.5) -- (5.5,2.5);
  \draw [color=black,fill=black,thick] (5.5,2.5) circle (.4ex);
  \node [draw, circle, fill = white, inner sep = 1pt] at (3.5,2.5) { };

     \draw [thick, rounded corners]  (1.5,1.5) -- (3.5,1.5);
  \draw [color=black,fill=black,thick] (3.5,1.5) circle (.4ex);
  \node [draw, circle, fill = white, inner sep = 1pt] at (1.5,1.5) { };    

\end{tikzpicture}
\end{center}
is not monotonic because the union of the ribbons with the two leftmost tails
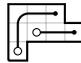
\begin{center}
\begin{tikzpicture}[scale = 0.25]

  \begin{scope}
    \clip (0,0) -| (1,1) -| (4,2) -| (3,3)  |- (2,3) -| (0,0);
    \draw [color=black!25] (0,0) grid (6,3);
  \end{scope}

  \draw [thick] (0,0) -| (1,1) -| (4,2) -| (3,3)  |- (2,3) -| (0,0);
  
  \draw [thick, rounded corners]  (0.5,0.5) |- (2.5,2.5);
  \draw [color=black,fill=black,thick] (2.5,2.5) circle (.4ex);
  \node [draw, circle, fill = white, inner sep = 1pt] at (0.5,0.5) { };

   \draw [thick, rounded corners]  (1.5,1.5) -- (3.5,1.5);
  \draw [color=black,fill=black,thick] (3.5,1.5) circle (.4ex);
  \node [draw, circle, fill = white, inner sep = 1pt] at (1.5,1.5) { };    

\end{tikzpicture}
\end{center}
is not the Young diagram of a partition. Monotonic tilings may be used to calculate the Schur expansion of the path power sum $\vec{p}_\mu$ as follows.

\begin{theorem}
\label{thm:path-Murnaghan--Nakayama}
{\em (Path Murnaghan--Nakayama rule)}  Let $\mu \vdash n$.  We have 
\begin{equation}
\vec{p}_{\mu} = m(\mu)! \cdot \sum_T \sign(T) \cdot s_{\shape(T)}
\end{equation}
where the sum is over all monotonic ribbon tilings $T$ with
$\sort(\type(T)) = \mu.$
\end{theorem}

To prove Theorem~\ref{thm:path-Murnaghan--Nakayama} we will find a bijection between standard stable word arrays and 
monotonic ribbon tilings. This proof is technical and appears in the next subsection.
There are two key differences between using the classical Murnaghan--Nakayama rule to find the 
$s$-expansion of $p_{\mu}$ and using the path Murnaghan--Nakayama rule to calculate the $s$-expansion of 
$\vec{p}_{\mu}$.
\begin{enumerate}
\item  When calculating the $s$-expansion of $p_{\mu}$, one first fixes a definite order of the composition $\mu = (\mu_1, \dots, \mu_r)$
and always adds ribbons of sizes $\mu_1, \dots, \mu_r$ in that order. 
When calculating the $s$-expansion of $\vec{p}_{\mu}$, one must add ribbons of the sizes $\mu_1, \dots, \mu_r$ in all possible orders.
The factor of $m(\mu)!$ in Theorem~\ref{thm:path-Murnaghan--Nakayama} distinguishes between repeated parts of $\mu$.
\item  When calculating the $s$-expansion of $p_{\mu}$, one need only add the ribbons in a standard fashion.
When calculating the $s$-expansion of $\vec{p}_{\mu}$, Theorem~\ref{thm:path-Murnaghan--Nakayama} imposes the stronger
monotonic condition.
\end{enumerate}
We close this subsection with an example of Theorem~\ref{thm:path-Murnaghan--Nakayama}.

\begin{example}
Let $\mu = (3,2,1)$ so that $m(\mu)! = 1$.  We use Theorem~\ref{thm:path-Murnaghan--Nakayama} to calculate the Schur expansion of 
$\vec{p}_{\mu}$.
The relevant monotonic tilings, together with their signs, are as shown in Figure~\ref{fig:tilings}.
We have the Schur expansion
\begin{equation*}
\vec{p}_{321} = 6 s_6 - 4 s_{51} + 2 s_{411} + 2 s_{33} - s_{321}.
\end{equation*}
\end{example}




\subsection{Proof of the path Murhaghan-Nakayama rule}
This subsection gives a bijective proof of  Theorem~\ref{thm:path-Murnaghan--Nakayama}.
We fix a partition $\mu = (\mu_1, \dots, \mu_r)$ with $r$ positive parts
and an integer $N \gg 0$.

The factor of $m(\mu)!$ in Theorem~\ref{thm:path-Murnaghan--Nakayama} is a nuisance,
and we decorate our ribbon tilings with permutations in $\symm_r$ to get around it.  An {\em enhanced monotonic ribbon tiling} with ribbon sizes
$\mu$ is a pair $(T,w)$ where
\begin{itemize}
\item $T$ is a monotonic ribbon tiling whose type 
$(k_1, \dots, k_r)$ is a rearrangement of $(\mu_1, \dots, \mu_r)$, and
\item   $w \in \symm_r$ is a permutation such that $\mu_{w(i)} = k_i$ for $1 \leq i \leq r$.
\end{itemize}
For any $T$ as in the first bullet point, we have $m(\mu)!$ choices of $w \in \symm_r$ as in the second bullet point. For example, if $T$ is the tiling on the left of Figure~\ref{fig:monotonic-and-not}, we have 
$r = 6, (k_1, \dots, k_6) = (8,5,8,1,4,1),$ and $\mu = (8,8,5,4,1,1)$. The $m(\mu)! = 2! \cdot 2!$ permutations $w \in \symm_6$ such that $(T,w)$ is an enhanced monotonic ribbon tiling are
\[ [1,3,2,5,4,6], [2,3,1,5,4,6], [1,3,2,6,4,5], \text{ and } [2,3,1,6,4,5].
\]Theorem~\ref{thm:path-Murnaghan--Nakayama} is equivalent to the assertion 
\begin{equation}
\label{eqn:our-formulation}
\vec{p}_{\mu} = \sum_{(T,w) \, \in \, \TTT} \sign(T) \cdot s_{\shape}(T)
\end{equation}
where 
\begin{equation}
\TTT := \{ \text{all enhanced monotonic ribbon tilings $(T,w)$ with ribbon sizes $\mu$} \}.
\end{equation}
In light of Lemmas~\ref{lem:path-to-array} and \ref{lem:sign-reversing},
Equation~\eqref{eqn:our-formulation} in the finite variable set $\{x_1, \dots, x_N\}$ is equivalent to the equality
\begin{equation}
\label{goal-formulation}
\sum_{\omega \, \in \, \OOO} \varepsilon \cdot x^{\wt(\omega)} 
= \sum_{(T,w) \, \in \, \TTT} \sign(T) \cdot a_{\shape(T)}(x_1, \dots, x_N).
\end{equation}
of alternating polynomials in $\CC[x_1, \dots, x_N]$ where
\begin{equation}
\OOO := \{ \text{all standard stable $\mu$-word arrays $\omega = ( u_1 \mid \cdots \mid u_N )$} \}.
\end{equation}
Equation~\ref{goal-formulation} is the formulation of Theorem~\ref{thm:path-Murnaghan--Nakayama} we will prove. We show Equation~\eqref{goal-formulation} is true by exhibiting a bijection between $\TTT$ and $\OOO$ that preserves 
weights and signs.

\begin{proof} {\em (of Theorem~\ref{thm:path-Murnaghan--Nakayama})}
We define a function $\varphi: \TTT \rightarrow \OOO$ by the following algorithm. Start with an enhanced
monotonic ribbon tiling $(T, w) \in \TTT$. 
We have the type and depth sequences
$\type(T) = (k_1, \dots, k_r)$ and $\depth(T) = (b_1 \geq \cdots \geq b_r)$.  We initialize
\begin{equation*}
\varphi(T,w) := (\varnothing \mid \cdots \mid \varnothing) 
\end{equation*} 
to be $N$ copies of the empty word. The weight sequence of this empty word array is
\[\wt(\varnothing \mid \cdots \mid \varnothing) = (N-1, N-2, \dots, 1, 0).\]
For $i = 1, 2, \dots, r$ we do the following.
\begin{enumerate}
\item  Suppose the word array $\varphi(T,w) = (u_1 \mid \cdots \mid u_N)$ contains precisely
one copy of each of the letters $w(1), w(2), \dots, w(i-1)$, and no other letters.  
Let $\wt( u_1 \mid \cdots \mid u_N) = (a_1, \dots, a_N)$ be the weight sequence of this array; the entries in this weight
sequence are distinct.
\item  Let $1 \leq j \leq n$ be such that $a_j$ is the $b_i^{th}$ largest element among $a_1, \dots, a_N$.
Prepend the letter $w(i)$ to the beginning of the word $u_j$ and change the $j^{th}$ entry of the weight sequence
from $a_j$ to $a_j + k_i$.
\end{enumerate}

We verify that $\varphi$ is  a well-defined function from $\TTT$ to $\OOO$. 
Ill-definedness of $\varphi$ could happen in one of two ways.
\begin{enumerate}
\item  The replacement $a_j \leadsto a_j + k_i$ in (2) could cause two entries in the weight sequence
\begin{equation*}
(a_1, \dots, a_{j-1}, a_j + k_i, a_{j+1}, \dots, a_N)
\end{equation*}
 to coincide. Depending on the value of $b_{i+1}$ in the next iteration, the choice of which word
 among $u_1, \dots, u_{j-1}, w(i) \cdot u_j, u_{j+1}, \dots, u_N$ to prepend with $w(i+1)$ could be ill-defined.
\item  The replacement $u_j \leadsto w(i) \cdot u_j$ could destabilize the word array $\varphi(T,w)$.
\end{enumerate}
We need to show that neither of these things actually happen.
Since $T$ is monotonic, adding a ribbon of size $k_i$ whose tail occupies row $b_i$
results in the Young diagram of a partition, so
Observation~\ref{obs:ribbon-addition} applies to show that (1) cannot happen. As for (2), since we prepend $w(i)$ 
the only new suffix that appears in the transformation of word arrays
\begin{equation*}
(u_1 \mid \cdots \mid u_{j-1} \mid  u_j \mid u_{j+1} \mid \cdots \mid u_N) \quad \leadsto \quad
(u_1 \mid \cdots \mid u_{j-1} \mid w(i) \cdot u_j \mid u_{j+1} \mid \cdots \mid u_N)
\end{equation*}
is the entire word $w(i) \cdot u_j$ in position $j$. Applying the defining condition \eqref{unstable-equality},
if this transformation destabilized $\varphi(T,w)$ there would be some position $s \neq j$ and a prefix-suffix factorization
$u_s = u'_s u''_s$ of $u_s$ such that 
\begin{equation}
\label{eqn:trouble-equation}
\mu_{u''_s} - s = \mu_{w(i) \cdot u_j} - j = \mu_{w(i)} + \mu_{u_j} - j = k_i + \mu_{u_j} - j.
\end{equation}
The prefix $u'_s$ is either empty or nonempty. 
\begin{itemize}
\item
If $u'_s$ is empty, then $u''_s = u_s$ and
Equation~\eqref{eqn:trouble-equation} reads 
\begin{equation*}
\mu_{u_s} - s = k_i + \mu_{u_j} - j. 
\end{equation*}
By Observation~\ref{obs:ribbon-addition},
 adding a ribbon of size $k_i$ whose tail occupies row $b_i$ does not result in a Young diagram.
This contradicts the monotonicity of $T$.
\item  If $u'_s$ is not empty, the last letter of $u'_s$ is $w(p)$ for some $p < i$ and corresponds to a ribbon of $T$ located
strictly to the left of the $i^{th}$ left-to-right ribbon. Since $k_i > 0$, Equation~\eqref{eqn:trouble-equation}
implies $\mu_{u''_s} - s > \mu_{u_j} - j$ and the algorithm for computing $\varphi(T,w)$ implies that $b_p < b_i$.
Since the depth sequence $\depth(T) = (b_1 \geq \cdots \geq b_r)$ is weakly decreasing, this is a contradiction.
\end{itemize}
In summary, the function $\varphi: \TTT \rightarrow \OOO$ is well-defined.

The map $\varphi: \TTT \rightarrow \OOO$ is best understood with an example.
Let $r = 6, \mu = (8,8,5,4,1,1),$ and let $T$ be the monotonic ribbon tiling on the left of Figure~\ref{fig:monotonic-and-not}
with $\depth(T) = (b_1, \dots, b_6) = (5,5,3,2,2,1)$.
There are $m(\mu)! = 2! \cdot 2! = 4$ permutations $w \in \symm_6$ such that $(T,w) \in \TTT$. Among these, we take 
$w = [2,3,1,5,4,6]$.  The stable word array $\varphi(T,w)$ is computed using the following table.

\begin{center}
\begin{tabular}{c | c | c | c | c | c}
$i$ & $w(i)$ & $\mu_{w(i)}$ & $b_i$ &  $\varphi(T,w)$ & $\wt(\varphi(T,w))$ \\ \hline
 &  &  &  & $(\varnothing \mid \varnothing \mid \varnothing \mid \varnothing \mid \varnothing \mid \varnothing)$ & $(5,4,3,2,1,0)$ \\
1 & 2 & 8 &  5 & $(\varnothing \mid \varnothing \mid \varnothing \mid \varnothing \mid 2 \mid \varnothing)$ & $(5,4,3,2,9,0)$ \\
2 & 3 & 5 &  5 & $(\varnothing \mid \varnothing \mid \varnothing \mid 3 \mid 2 \mid \varnothing)$ & $(5,4,3,7,9,0)$ \\
3 & 1 & 8 & 3 & $(1 \mid \varnothing \mid \varnothing \mid 3 \mid 2 \mid \varnothing)$ & $(13,4,3,7,9,0)$ \\
4 & 5 & 1 & 2 & $(1 \mid \varnothing \mid \varnothing \mid 3 \mid 5 \, \, 2 \mid \varnothing)$ & $(13,4,3,7,10,0)$ \\
5 & 4 & 4 & 2 & $(1 \mid \varnothing \mid \varnothing \mid 3 \mid 4 \, \, 5 \, \, 2 \mid \varnothing)$ & $(13,4,3,7,14,0)$ \\
6 & 6 & 1 & 1 & $(1 \mid \varnothing \mid \varnothing \mid 3 \mid 6 \, \,  4 \, \, 5 \, \, 2 \mid \varnothing)$ & $(13,4,3,7,15,0)$ \\
\end{tabular}
\end{center}
To move from the previous row of the table to the current row, we add the number $\mu_{w(i)}$ to the $b_i^{th}$ largest entry in $\wt(\varphi(T,w))$ and prepend $w(i)$ to the word in the corresponding position of the array $\varphi(T,w)$. The last row of the table gives $\varphi(T,w) = (1 \mid \varnothing \mid \varnothing \mid 3 \mid 6 \, \,  4 \, \, 5 \, \, 2 \mid \varnothing) \in \OOO$.

We prove that $\varphi: \TTT \to \OOO$ is a bijection by constructing its inverse. To this end, we define a function $\psi: \OOO \rightarrow \TTT$ using the following algorithm. Let $\omega  \in \OOO$. We initialize \[\psi(\omega) := (T,w) \text{ where $T$ is the empty filling and $w$ is the empty word.}\]
We also initialize two `current word arrays' as \[(u_1 \mid \cdots \mid u_N) := \omega \text{ and }
(v_1 \mid \cdots \mid v_N) := (\varnothing \mid \cdots \mid \varnothing).\]
We build up $T$ and $w$ step-by-step using the following procedure which moves letters  from the $u$-array to the $v$-array.
For $i =1, \dots, r$ we do the following.
\begin{enumerate}
\item Suppose we have determined the first $i-1$ values $w(1), w(2), \dots, w(i-1)$ of the permutation $w$. Suppose further that
\begin{itemize}
    \item the current $v$-array $(v_1 \mid \cdots \mid v_N)$ is stable and has letters $w(1), w(2), \dots, w(i-1)$, 
    \item the current $u$-array $(u_1 \mid \cdots \mid u_N)$ has the remaining letters in $[r]$ which are not contained in the $v$-array, and
    \item the current monotonic ribbon tiling $T$ satisfies $\type(T) = (\mu_{w(1)}, \dots, \mu_{w(i-1)})$ and $\depth(T) = (b_1 \geq \cdots \geq b_{i-1})$.
\end{itemize}
Let $\wt(v_1 \mid \cdots \mid v_N) = (a_1, \dots, a_N)$
be the weight sequence of the current $v$-array.
\item At least one of the words $u_1, \dots, u_N$ is not empty. Among the positions $1 \leq j \leq N$ for which $u_j$ is not 
empty, choose $j$ so that $a_j$ is minimal.  Suppose $a_j$ is the $b_i^{th}$ largest number 
in the list  $(a_1, \dots, a_N)$.
\item
If $\ell$ is the last letter of the word $u_j$, set $w(i) := \ell$, erase $\ell$ from the end of $u_j$ and prepend $\ell$ to the start
of $v_j$. Add a ribbon of size $\mu_{\ell}$ to $T$ whose tail occupies row $b_i$.
\end{enumerate}

We verify that this algorithm gives a well-defined function $\psi: \OOO \rightarrow \TTT$.
This follows from two observations.
\begin{itemize}
\item
In Step 2, we need to verify that the value of $j$ is uniquely determined. In fact, we have the stronger property that the entries of $\wt(v_1 \mid \cdots \mid v_N) = (a_1, \dots, a_N)$ are distinct.  By construction, the words $v_1, \dots, v_N$ in the $v$-array $(v_1 \mid \cdots \mid v_N)$ 
are always suffices of those in the stable array $\omega$, so $(v_1 \mid \cdots \mid v_N)$ is stable by Lemma~\ref{lem:hereditary-stability}. By
Lemma~\ref{lem:collision-instability} 
$\wt(v_1 \mid \cdots \mid v_N) = (a_1, \dots, a_N)$  has distinct entries. 
\item  In Step 3, we need to show that we can add a ribbon $\xi$ to $T$ of size $|\xi| = \mu_{\ell}$ whose tail
occupies row $b_i$ such that $T \sqcup \xi$ remains a monotonic ribbon tiling.
Lemma~\ref{lem:hereditary-stability}, Lemma~\ref{lem:collision-instability}, and Observation~\ref{obs:ribbon-addition}
guarantee that a ribbon $\xi$ of size $\mu_{\ell}$ whose tail occupies row $b_i$ can be added to $T$ 
such that $T \sqcup \xi$ is the Young diagram of a partition. Step 2 forces $b_1 \geq \cdots \geq b_{i-1} \geq b_i$, so that the column occupied by the tail of $\xi$ is strictly to the right of the columns occupied by the tails of the ribbons in $T$.
\end{itemize}
By the above two bullet points, the function $\psi: \OOO \rightarrow \TTT$ is well-defined.

As with $\varphi$, an example should help clarify the definition of $\psi: \OOO \rightarrow \TTT$. 
Suppose we are given $\omega = (1 \mid \varnothing \mid \varnothing \mid 3 \mid 6 \, \,  4 \, \, 5 \, \, 2 \mid \varnothing) \in \OOO$
where $\mu = (8,8,5,4,1,1)$ so that $r = 6$.
\begin{center}
\begin{tabular}{ c | c | c | c | c | c | c}
$i$ & $w(i)$ & $b_i$ & $\mu_{w(i)}$ & $(u_1 \mid \cdots \mid u_N)$  & $(v_1 \mid \cdots \mid v_N)$ & 
$\wt(v_1 \mid \cdots \mid v_N)$ \\ \hline
  &  &  &  & $ (1 \mid \varnothing \mid \varnothing \mid 3 \mid 6 \, \,  4 \, \, 5 \, \, 2 \mid \varnothing)$ &
  $(\varnothing \mid \varnothing \mid \varnothing \mid \varnothing \mid \varnothing \mid \varnothing)$  & 
  $(5,4,3,2,1,0)$ \\
 1 & 2 & 5 & 8 & $ (1 \mid \varnothing \mid \varnothing \mid 3 \mid 6 \, \,  4 \, \, 5  \mid \varnothing)$ &
  $(\varnothing \mid \varnothing \mid \varnothing \mid \varnothing \mid 2 \mid \varnothing)$  & 
  $(5,4,3,2,9,0)$ \\
   2 & 3 & 5 & 5 & $ (1 \mid \varnothing \mid \varnothing \mid \varnothing \mid 6 \, \,  4 \, \, 5  \mid \varnothing)$ &
  $(\varnothing \mid \varnothing \mid \varnothing \mid 3 \mid 2 \mid \varnothing)$  & 
  $(5,4,3,7,9,0)$ \\
   3 & 1 & 3 & 8 & $ (\varnothing \mid \varnothing \mid \varnothing \mid \varnothing \mid 6 \, \,  4 \, \, 5  \mid \varnothing)$ &
  $(1 \mid \varnothing \mid \varnothing \mid 3 \mid 2 \mid \varnothing)$  & 
  $(13,4,3,7,9,0)$ \\
   4 & 5 & 2 & 1 & $ (\varnothing \mid \varnothing \mid \varnothing \mid \varnothing \mid 6 \, \,  4  \mid \varnothing)$ &
  $(1 \mid \varnothing \mid \varnothing \mid 3 \mid 5 \, \,  2 \mid \varnothing)$  & 
  $(13,4,3,7,10,0)$ \\
   5 & 4 & 2 & 4 & $ (\varnothing \mid \varnothing \mid \varnothing \mid \varnothing \mid 6   \mid \varnothing)$ &
  $(1 \mid \varnothing \mid \varnothing \mid 3 \mid 4 \, \,  5 \, \,  2 \mid \varnothing)$  & 
  $(13,4,3,7,14,0)$ \\
    6  & 6 & 1 & 1 & $ (\varnothing \mid \varnothing \mid \varnothing \mid \varnothing \mid \varnothing   \mid \varnothing)$ &
  $(1 \mid \varnothing \mid \varnothing \mid 3 \mid 6 \, \,  4 \, \,  5 \, \,  2 \mid \varnothing)$  & 
  $(13,4,3,7,15,0)$ \\
\end{tabular}
\end{center}
We conclude that $\psi(\omega) = (T,w)$ where $T$ is the unique monotonic $\mu$-ribbon tiling with ribbons of sizes 
$8,5,8,1,4,1$ being added from left to right with tails occupying rows $5,5,3,2,2,1$ (respectively) and 
$w = [2,3,1,5,4,6] \in \symm_6$.
This is again the ribbon tableau from Figure~\ref{fig:monotonic-and-not}.

Our next task is to show that $\varphi: \TTT \to \OOO$ and $\psi: \OOO \to \TTT$ are mutually inverse bijections. We start by establishing $\psi \circ \varphi = \mathrm{id}_\TTT$. To do this, we need some machinery.

Let $\omega \in \OOO$ be a stable word array. A permutation $s = [s(1), \dots, s(r)] \in \symm_r$ an {\em $\omega$-suffix permutation} if 
\begin{quote}
for all $0 \leq  i \leq r$ the word array
$\omega^{(s,i)} := (u_1^{(i)} \mid \cdots \mid u_N^{(i)} )$ obtained by restricting $\omega$ to the letters $s(1), s(2), \dots, s(i)$
has the property that each $u_j^{(i)}$ is a suffix of $u_j$. 
\end{quote}
If $\ell(u_i)$ is the length of the word $u_i$, the number of $\omega$-suffix permutations is the multinomial coefficient ${\ell(u_1) + \cdots + \ell(u_N) \choose \ell(u_1), \, \dots, \, \ell(u_N)}$. The $\omega$-suffix permutations $s \in \symm_r$ encode chains 
\[(\varnothing \mid \cdots \mid \varnothing) =  \omega^{(s,0)}, \, \omega^{(s,1)}, \dots, \omega^{(s,r)} = \omega\]
of word arrays in which $\omega^{(s,i)}$ is obtained from $\omega^{(s,i-1)}$ by prepending the letter $s(i)$ to the $s(i)^{th}$ word $u_{s(i)}^{(s,i-1)}$ in the array $\omega^{(s,i-1)}$. In our running example  $\omega = (1 \mid \varnothing \mid \varnothing \mid 3 \mid 6 \, 4 \, 5 \, 2 \mid \varnothing)$, 
 a permutation $s \in \symm_6$ is $\omega$-suffix if and only if $s^{-1}(2) < s^{-1}(5) < s^{-1}(4) < s^{-1}(6)$.
Two such permutations are $[2,3,5,4,6,1]$ and $[1,2,5,3,4,6]$.

Let $\omega \in \OOO$ and let $s \in \symm_r$ be an $\omega$-suffix permutation. We associate a `depth sequence' to $s$ as follows. By Lemma~\ref{lem:hereditary-stability} and the stability of $\omega$ the word array $\omega^{(s,i)}$ is stable for all $0 \leq i \leq r$.  By Lemma~\ref{lem:collision-instability}, the weight sequence $\wt(\omega^{(s,i)}) = (a_1^{(s,i)}, \dots, a_N^{(s,i)})$ has distinct entries for $0 \leq i \leq r$. For $1 \leq i \leq r$, the weight sequence
$\wt(\omega^{(s,i)}) = (a_1^{(s,i)}, \dots, a_N^{(s,i)})$ is obtained from the previous weight sequence
$\wt(\omega^{(s,i-1)}) = (a_1^{(s,i-1)}, \dots, a_N^{(s,i-1)})$ by adding $\mu_{s(i)}$ to an entry $a_{j_i}^{(s,i-1)}$ for some unique $1 \leq j_i \leq N$. If $a_{j_i}^{(s,i-1)}$ is the $b_i^{th}$ largest number among the $N$ distinct numbers $a_1^{(s,i-1)}, \dots, a_N^{(s,i-1)}$, we define the {\em depth sequence} of $s$ by $\depth(s) := (b_1, \dots, b_r)$. The $\omega$-suffix permutation $s$ is {\em monotonic} if its depth sequence $\depth(s) = (b_1 \geq \cdots \geq b_r)$ is weakly decreasing. As one might expect, monotonic suffix permutations and their depth sequences are related to monotonic tilings. 

\vspace{0.1in}

{\bf Claim:}  {\em Let $\omega \in \OOO$ be a stable word array. There exists a unique monotonic $\omega$-suffix permutation $w \in \symm_r$.  We have $\psi(\omega) = (T,w)$ and $\depth(T) = \depth(w)$.}

\vspace{0.1in}

Indeed, the unique monotonic $\omega$-suffix permutation $w \in \symm_r$ is the suffix permutation such that the value which differs between the weight sequences $\wt(\omega^{(w,i-1)}) = (a_1^{(w,i-1)}, \dots, a_N^{(w,i-1)})$ and $\wt(\omega^{(w,i)}) = (a_1^{(w,i)}, \dots, a_N^{(w,i)})$ is as small as possible for $i = 1, 2, \dots ,r$. This is precisely the algorithm defining the map $\psi: \OOO \to \TTT$. We use this Claim to prove that $\varphi$ and $\psi$ are mutually inverse bijections as follows.

Let $(T,w) \in \TTT$ be an enhanced monotonic ribbon tiling. Then $\varphi(T,w) = \omega \in \OOO$ is a stable word array. We have $\psi(\omega) = (T',w') \in \TTT$. Our goal is to prove $(T,w) = (T',w')$. Since $T$ is a mononic tiling, the algorithm defining $\varphi(T,w) = \omega$ forces $w$ to be the unique monotonic $\omega$-suffix permutation. The claim forces $w' = w$. It follows that $T'$ is a monotonic ribbon tiling with $\type(T') = \type(T)$ and $\depth(T') = \depth(T)$. Observation~\ref{obs:monotonic-characterization} forces $T = T'$.  We conclude that $(T,w) = (T',w')$ and therefore   $\psi \circ \varphi = \mathrm{id}_\TTT$.

We show that the map $\varphi: \TTT \to \OOO$ is surjective. Let $w \in \symm_r$ be the unique monotonic $\omega$-suffix permutation. We construct a monotonic ribbon tiling $T^{(w)}$ from $w$ as follows. For $0 \leq i \leq r$, let $\lambda^{(i)} = (\lambda^{(i)}_1 \geq \cdots \geq \lambda^{(i)}_N)$ be the partition such that
\begin{quote}
    $(\lambda^{(i)}_1 + N -1 > \lambda^{(i)}_2 + N - 2 >  \cdots > \lambda^{(i)}_N)$ is the decreasing rearrangement of the sequence $\wt(\omega^{(w,i)}) = (a_1^{(w,i)}, \dots, a_N^{(w,i)})$.
\end{quote}
Lemmas~\ref{lem:collision-instability} and \ref{lem:hereditary-stability} guarantee that the entries of $\wt(\omega^{(w,i)}) = (a_1^{(w,i)}, \dots, a_N^{(w,i)})$ sort to a partition with distinct parts, so that the entries of $\lambda^{(i)}$ are weakly decreasing and nonnegative for all $0 \leq i \leq r$. Since the sequence $\wt(\omega^{(w,i)})$ is obtained from $\wt(\omega^{(w,i-1)})$ by increasing a single value, we have $\lambda^{(i-1)} \subset \lambda^{(i)}$ and we see that $\xi^{(i)} := \lambda^{(i)}/\lambda^{(i-1)}$ is a ribbon for all $1 \leq i \leq r$. Let $T^{(w)}$ be the tiling
\[ T^{(w)} := \xi^{(1)} \sqcup \xi^{(2)} \sqcup \cdots \sqcup \xi^{(r)}\]
of the Young diagram $\lambda^{(r)}$ the ribbons $\xi^{(1)}, \dots, \xi^{(r)}$. Let $b_i$ be the row occupied by the tail of $\xi^{(i)}$ and let $c_i$ be the column occupied by the tail of $\xi^{(i)}$.  Since $w$ is monotonic we have $b_1 \geq \cdots \geq b_r$.  Since $\xi^{(1)} \sqcup \cdots \sqcup \xi^{(i)} = \lambda^{(i)}$ is a partition for all $1 \leq i \leq r$, the weak inequalities $b_1 \geq \cdots \geq b_r$ force the strict inequalities $c_1 < \cdots < c_r$. Therefore $T^{(w)}$ satisfies Definition~\ref{def:monotonic-definition} and is a monotonic ribbon tiling and $(T^{(w)},w) \in \TTT$ is an enhanced monotonic ribbon tiling. We have $\varphi(T^{(w)},w) = \omega$ by the construction of the map $\varphi: \TTT \to \OOO$. This proves that $\varphi: \TTT \to \OOO$ is surjective. Since $\TTT$ and $\OOO$ are finite sets and $\psi \circ \varphi = \mathrm{id}_\TTT$, we conclude that $\psi: \OOO \to \TTT$ and $\varphi: \TTT \to \OOO$ are mutually inverse bijections.

Recall that the desired Theorem~\ref{thm:path-Murnaghan--Nakayama} will be proven if we can establish
Equation~\eqref{goal-formulation}, which we restate here:
\begin{equation*}
\sum_{\omega \, \in \, \OOO} \varepsilon \cdot x^{\wt(\omega)} 
= \sum_{(T,w) \, \in \, \TTT} \sign(T) \cdot a_{\shape(T)}(x_1, \dots, x_N).
\end{equation*}
We have a bijection $\varphi: \TTT \to \OOO$ between the indexing sets of these sums.
 Observation~\ref{obs:ribbon-addition} shows that if $\varphi(T,w) = \omega$ we have
 \begin{equation}
 \label{term-equality}
 \varepsilon \cdot x^{\wt(\omega)} = \sign(T) \cdot a_{\shape(T)}(x_1, \dots, x_N).
 \end{equation}
 This proves Equation~\eqref{goal-formulation}.
 Dividing both sides of Equation~\eqref{goal-formulation} by $\varepsilon \cdot x^\delta$
 and taking the limit as $N \rightarrow \infty$ completes the proof of 
 Theorem~\ref{thm:path-Murnaghan--Nakayama}.
\end{proof}

\section{Characters of partial permutations}

Our motivation for this work is to expand the symmetric function $A_{n,I,J}$ associated to a partial permutation into the Schur basis.
In this section, we explain how to use the path Murnaghan--Nakayama formula to perform this computation using the, then use $A_{n,I,J}$ to compute the character evaluation $\chi_\lambda([I,J])$.
As a consequence, we show $A_{n,I,J}$ and exhibits a stability phenomenon that extends to $\chi_\lambda([I,J])$.

\subsection{Characters of partial permutations}
Let $\lambda, \mu \vdash n$.  We introduce the monotonic tiling enumerator
\begin{equation}
\vec{\chi}^{ \, \lambda}_{ \, \mu} := \sum_T \sign(T)
\end{equation}
where the sum is over monotonic ribbon tilings $T$ of shape $\lambda$ with ribbon sizes $\mu_1, \dots, \mu_r$.
With this notation, Theorem~\ref{thm:path-Murnaghan--Nakayama} reads
\begin{equation}
\label{eqn:pmn}
\vec{p}_{\mu} = m(\mu)! \cdot \sum_{\lambda \, \vdash \, n} \vec{\chi}^{ \, \lambda}_{ \, \mu}  \cdot s_{\lambda}
\end{equation}
in parallel with the classical expansion $p_{\mu} = \sum_{\lambda \vdash n} \chi^{\lambda}_{\mu} \cdot s_{\lambda}.$ Our combinatorial formula for the Schur expansion of $A_{n,I,J}$ reads as follows.

\begin{corollary}
\label{cor:atomic-schur-expansion}
Let $(I,J) \in \symm_{n,k}$ be a partial permutation with path type
$\mu \vdash a$ and cycle type
$\nu$.  Then
\begin{equation}
\label{eqn:atomic-schur-eqn}
A_{n,I,J} = m(\mu)! \cdot \sum_{\lambda \, \vdash \, n}  \left(
\sum_{\rho \, \vdash \, a}  \vec{\chi}_{ \, \mu}^{ \, \rho} \cdot \chi_{\nu}^{\lambda/\rho}
 \right) \cdot s_{\lambda}.
\end{equation}
\end{corollary}

\begin{proof}
We calculate
\begin{align}
A_{n,I,J} &= \vec{p}_{\mu} \cdot p_{\nu} \\
&= m(\mu)! \cdot \left( \sum_{\rho \, \vdash \, a} \vec{\chi}_{ \, \mu}^{ \, \rho} \cdot s_{\rho} \right) \cdot p_{\nu} \\
&= m(\mu)! \cdot \sum_{\lambda \, \vdash \, n}  \left(
\sum_{\rho \, \vdash \, a}  \vec{\chi}_{ \, \mu}^{ \, \rho} \cdot \chi_{\nu}^{\lambda/\rho}
 \right) \cdot s_{\lambda}
\end{align}
where the first equality uses Proposition~\ref{prop:path-cycle-factorization}, the second uses 
 Equation~\eqref{eqn:pmn}, 
and the third uses the classical Murnaghan--Nakayama rule (Theorem~\ref{thm:classical-mn}).
\end{proof}

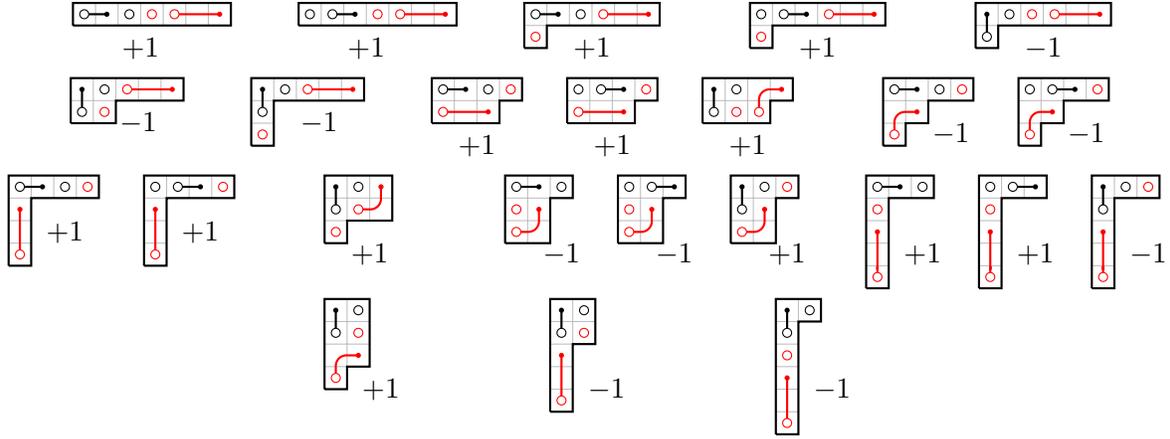
\begin{figure}
\begin{center}
\begin{tikzpicture}[scale = 0.3]

\begin{scope}
   \clip (0,0) -| (7,1) -| (0,0);
    \draw [color=black!25] (0,0) grid (7,1);
\end{scope}

\draw [thick] (0,0) -| (7,1) -| (0,0);

\draw [thick, rounded corners]  (0.5,0.5) -- (1.5,0.5);
\draw [color=black,fill=black,thick] (1.5,0.5) circle (.4ex);
\node [draw, circle, fill = white, inner sep = 1.2pt] at (0.5,0.5) { };    
\node [draw, circle, fill = white, inner sep = 1.2pt] at (2.5,0.5) { };    
\draw [color=red, thick, rounded corners]  (4.5,0.5) -- (6.5,0.5);   
\draw [color=red,fill=black,thick] (6.5,0.5) circle (.4ex);
\node [color=red,draw, circle, fill = white, inner sep = 1.2pt] at (4.5,0.5) { }; 
\node [color=red,draw, circle, fill = white, inner sep = 1.2pt] at (3.5,0.5) { };   
    
\node at (3,-1) {$+1$};

\begin{scope}
   \clip (10,0) -| (17,1) -| (10,0);
    \draw [color=black!25] (10,0) grid (17,1);
\end{scope}

\draw [thick] (10,0) -| (17,1) -| (10,0);

\draw [thick, rounded corners]  (11.5,0.5) -- (12.5,0.5);
\draw [color=black,fill=black,thick] (12.5,0.5) circle (.4ex);
\node [draw, circle, fill = white, inner sep = 1.2pt] at (11.5,0.5) { };    
\node [draw, circle, fill = white, inner sep = 1.2pt] at (10.5,0.5) { };    
\draw [color=red, thick, rounded corners]  (14.5,0.5) -- (16.5,0.5);   
\draw [color=red,fill=black,thick] (16.5,0.5) circle (.4ex);
\node [color=red,draw, circle, fill = white, inner sep = 1.2pt] at (14.5,0.5) { }; 
\node [color=red,draw, circle, fill = white, inner sep = 1.2pt] at (13.5,0.5) { };   
    
\node at (13,-1) {$+1$};

\begin{scope}
   \clip (20,-1) -| (21,0) -| (26,1) -| (20,-1);
    \draw [color=black!25] (20,-1) grid (26,1);
\end{scope}

\draw [thick] (20,-1) -| (21,0) -| (26,1) -| (20,-1);

\draw [thick, rounded corners]  (20.5,0.5) -- (21.5,0.5);
\draw [color=black,fill=black,thick] (21.5,0.5) circle (.4ex);
\node [draw, circle, fill = white, inner sep = 1.2pt] at (20.5,0.5) { };    
\node [draw, circle, fill = white, inner sep = 1.2pt] at (22.5,0.5) { };    
\draw [color=red, thick, rounded corners]  (23.5,0.5) -- (25.5,0.5);   
\draw [color=red,fill=black,thick] (25.5,0.5) circle (.4ex);
\node [color=red,draw, circle, fill = white, inner sep = 1.2pt] at (23.5,0.5) { }; 
\node [color=red,draw, circle, fill = white, inner sep = 1.2pt] at (20.5,-0.5) { };   
    
\node at (23,-1) {$+1$};

\begin{scope}
   \clip (30,-1) -| (31,0) -| (36,1) -| (30,-1);
    \draw [color=black!25] (30,-1) grid (36,1);
\end{scope}

\draw [thick] (30,-1) -| (31,0) -| (36,1) -| (30,-1);

\draw [thick, rounded corners]  (31.5,0.5) -- (32.5,0.5);
\draw [color=black,fill=black,thick] (32.5,0.5) circle (.4ex);
\node [draw, circle, fill = white, inner sep = 1.2pt] at (31.5,0.5) { };    
\node [draw, circle, fill = white, inner sep = 1.2pt] at (30.5,0.5) { };    
\draw [color=red, thick, rounded corners]  (33.5,0.5) -- (35.5,0.5);   
\draw [color=red,fill=black,thick] (35.5,0.5) circle (.4ex);
\node [color=red,draw, circle, fill = white, inner sep = 1.2pt] at (33.5,0.5) { }; 
\node [color=red,draw, circle, fill = white, inner sep = 1.2pt] at (30.5,-0.5) { };   
    
\node at (33,-1) {$+1$};

\begin{scope}
   \clip (40,-1) -| (41,0) -| (46,1) -| (40,-1);
    \draw [color=black!25] (40,-1) grid (46,1);
\end{scope}

\draw [thick] (40,-1) -| (41,0) -| (46,1) -| (40,-1);

\draw [thick, rounded corners]  (40.5,-0.5) -- (40.5,0.5);
\draw [color=black,fill=black,thick] (40.5,0.5) circle (.4ex);
\node [draw, circle, fill = white, inner sep = 1.2pt] at (40.5,-0.5) { };    
\node [draw, circle, fill = white, inner sep = 1.2pt] at (41.5,0.5) { };    
\draw [color=red, thick, rounded corners]  (43.5,0.5) -- (45.5,0.5);   
\draw [color=red,fill=black,thick] (45.5,0.5) circle (.4ex);
\node [color=red,draw, circle, fill = white, inner sep = 1.2pt] at (43.5,0.5) { }; 
\node [color=red,draw, circle, fill = white, inner sep = 1.2pt] at (42.5,0.5) { };   
    
\node at (43,-1) {$-1$};

\end{tikzpicture}
\end{center}

\begin{center}
\begin{tikzpicture}[scale = 0.3]

\begin{scope}
   \clip (0,0) -| (2,1) -| (5,2) -| (0,0);
    \draw [color=black!25] (0,0) grid (5,2);
\end{scope}

\draw [thick] (0,0) -| (2,1) -| (5,2) -| (0,0);

\draw [thick, rounded corners]  (0.5,0.5) -- (0.5,1.5);
\draw [color=black,fill=black,thick] (0.5,1.5) circle (.4ex);
\node [draw, circle, fill = white, inner sep = 1.2pt] at (0.5,0.5) { };    
\node [draw, circle, fill = white, inner sep = 1.2pt] at (1.5,1.5) { };    
\draw [color=red, thick, rounded corners]  (2.5,1.5) -- (4.5,1.5);   
\draw [color=red,fill=black,thick] (4.5,1.5) circle (.4ex);
\node [color=red,draw, circle, fill = white, inner sep = 1.2pt] at (2.5,1.5) { }; 
\node [color=red,draw, circle, fill = white, inner sep = 1.2pt] at (1.5,0.5) { };   
    
\node at (3,0) {$-1$};

\begin{scope}
   \clip (8,-1) -| (9,1) -| (13,2) -| (8,-1);
    \draw [color=black!25] (8,-1) grid (13,2);
\end{scope}

\draw [thick] (8,-1) -| (9,1) -| (13,2) -| (8,-1);

\draw [thick, rounded corners]  (8.5,0.5) -- (8.5,1.5);
\draw [color=black,fill=black,thick] (8.5,1.5) circle (.4ex);
\node [draw, circle, fill = white, inner sep = 1.2pt] at (8.5,0.5) { };    
\node [draw, circle, fill = white, inner sep = 1.2pt] at (9.5,1.5) { };    
\draw [color=red, thick, rounded corners]  (10.5,1.5) -- (12.5,1.5);   
\draw [color=red,fill=black,thick] (12.5,1.5) circle (.4ex);
\node [color=red,draw, circle, fill = white, inner sep = 1.2pt] at (10.5,1.5) { }; 
\node [color=red,draw, circle, fill = white, inner sep = 1.2pt] at (8.5,-0.5) { };   
    
\node at (11,0) {$-1$};

\begin{scope}
   \clip (16,0) -| (19,1) -| (20,2) -| (16,0);
    \draw [color=black!25] (16,0) grid (20,2);
\end{scope}

\draw [thick] (16,0) -| (19,1) -| (20,2) -| (16,0);

\draw [thick, rounded corners]  (16.5,1.5) -- (17.5,1.5);
\draw [color=black,fill=black,thick] (17.5,1.5) circle (.4ex);
\node [draw, circle, fill = white, inner sep = 1.2pt] at (16.5,1.5) { };    
\node [draw, circle, fill = white, inner sep = 1.2pt] at (18.5,1.5) { };    
\draw [color=red, thick, rounded corners]  (16.5,0.5) -- (18.5,0.5);   
\draw [color=red,fill=black,thick] (18.5,0.5) circle (.4ex);
\node [color=red,draw, circle, fill = white, inner sep = 1.2pt] at (16.5,0.5) { }; 
\node [color=red,draw, circle, fill = white, inner sep = 1.2pt] at (19.5,1.5) { };   
    
\node at (18,-1) {$+1$};

\begin{scope}
   \clip (22,0) -| (25,1) -| (26,2) -| (22,0);
    \draw [color=black!25] (22,0) grid (26,2);
\end{scope}

\draw [thick]  (22,0) -| (25,1) -| (26,2) -| (22,0);

\draw [thick, rounded corners]  (23.5,1.5) -- (24.5,1.5);
\draw [color=black,fill=black,thick] (24.5,1.5) circle (.4ex);
\node [draw, circle, fill = white, inner sep = 1.2pt] at (23.5,1.5) { };    
\node [draw, circle, fill = white, inner sep = 1.2pt] at (22.5,1.5) { };    
\draw [color=red, thick, rounded corners]  (22.5,0.5) -- (24.5,0.5);   
\draw [color=red,fill=black,thick] (24.5,0.5) circle (.4ex);
\node [color=red,draw, circle, fill = white, inner sep = 1.2pt] at (22.5,0.5) { }; 
\node [color=red,draw, circle, fill = white, inner sep = 1.2pt] at (25.5,1.5) { };   
    
\node at (24,-1) {$+1$};

\begin{scope}
   \clip (28,0) -| (31,1) -| (32,2) -| (28,0);
    \draw [color=black!25] (28,0) grid (32,2);
\end{scope}

\draw [thick]  (28,0) -| (31,1) -| (32,2) -| (28,0);

\draw [thick, rounded corners]  (28.5,0.5) -- (28.5,1.5);
\draw [color=black,fill=black,thick] (28.5,1.5) circle (.4ex);
\node [draw, circle, fill = white, inner sep = 1.2pt] at (28.5,0.5) { };    
\node [draw, circle, fill = white, inner sep = 1.2pt] at (29.5,1.5) { };    
\draw [color=red, thick, rounded corners]  (30.5,0.5) |- (31.5,1.5);   
\draw [color=red,fill=black,thick] (31.5,1.5) circle (.4ex);
\node [color=red,draw, circle, fill = white, inner sep = 1.2pt] at (30.5,0.5) { }; 
\node [color=red,draw, circle, fill = white, inner sep = 1.2pt] at (29.5,0.5) { };   
    
\node at (30,-1) {$+1$};

\begin{scope}
   \clip (36,-1) -| (37,0) -| (38,1) -| (40,2) -| (36,-1);
    \draw [color=black!25] (36,-1) grid (40,2);
\end{scope}

\draw [thick]  (36,-1) -| (37,0) -| (38,1) -| (40,2) -| (36,-1);

\node at (39,-0.5) {$-1$};

\draw [thick, rounded corners]  (36.5,1.5) -- (37.5,1.5);
\draw [color=black,fill=black,thick] (37.5,1.5) circle (.4ex);
\node [draw, circle, fill = white, inner sep = 1.2pt] at (36.5,1.5) { };    
\node [draw, circle, fill = white, inner sep = 1.2pt] at (38.5,1.5) { };  
\draw [color=red, thick, rounded corners]  (36.5,-0.5) |- (37.5,0.5);   
\draw [color=red,fill=black,thick] (37.5,0.5) circle (.4ex);
\node [color=red,draw, circle, fill = white, inner sep = 1.2pt] at (36.5,-0.5) { }; 
\node [color=red,draw, circle, fill = white, inner sep = 1.2pt] at (39.5,1.5) { };

\begin{scope}
   \clip (42,-1) -| (43,0) -| (44,1) -| (46,2) -| (42,-1);
    \draw [color=black!25] (42,-1) grid (46,2);
\end{scope}

\draw [thick] (42,-1) -| (43,0) -| (44,1) -| (46,2) -| (42,-1);

\draw [thick, rounded corners]  (43.5,1.5) -- (44.5,1.5);
\draw [color=black,fill=black,thick] (44.5,1.5) circle (.4ex);
\node [draw, circle, fill = white, inner sep = 1.2pt] at (43.5,1.5) { };    
\node [draw, circle, fill = white, inner sep = 1.2pt] at (42.5,1.5) { };  
\draw [color=red, thick, rounded corners]  (42.5,-0.5) |- (43.5,0.5);   
\draw [color=red,fill=black,thick] (43.5,0.5) circle (.4ex);
\node [color=red,draw, circle, fill = white, inner sep = 1.2pt] at (42.5,-0.5) { }; 
\node [color=red,draw, circle, fill = white, inner sep = 1.2pt] at (45.5,1.5) { };

\node at (45,-0.5) {$-1$};

\end{tikzpicture}
\end{center}

\begin{center}
\begin{tikzpicture}[scale = 0.3]

\begin{scope}
   \clip (0,0) -| (1,3) -| (4,4) -| (0,0);
    \draw [color=black!25] (0,0) grid (4,4);
\end{scope}

\draw [thick] (0,0) -| (1,3) -| (4,4) -| (0,0);

\draw [thick, rounded corners]  (0.5,3.5) -- (1.5,3.5);
\draw [color=black,fill=black,thick] (1.5,3.5) circle (.4ex);
\node [draw, circle, fill = white, inner sep = 1.2pt] at (0.5,3.5) { };    
\node [draw, circle, fill = white, inner sep = 1.2pt] at (2.5,3.5) { };    
\draw [color=red, thick, rounded corners]  (0.5,0.5) -- (0.5,2.5);   
\draw [color=red,fill=black,thick] (0.5,2.5) circle (.4ex);
\node [color=red,draw, circle, fill = white, inner sep = 1.2pt] at (0.5,0.5) { }; 
\node [color=red,draw, circle, fill = white, inner sep = 1.2pt] at (3.5,3.5) { };   
    
\node at (2.5,1.5) {$+1$};

\begin{scope}
   \clip (6,0) -| (7,3) -| (10,4) -| (6,0);
    \draw [color=black!25] (6,0) grid (10,4);
\end{scope}

\draw [thick] (6,0) -| (7,3) -| (10,4) -| (6,0);

\draw [thick, rounded corners]  (7.5,3.5) -- (8.5,3.5);
\draw [color=black,fill=black,thick] (8.5,3.5) circle (.4ex);
\node [draw, circle, fill = white, inner sep = 1.2pt] at (7.5,3.5) { };    
\node [draw, circle, fill = white, inner sep = 1.2pt] at (6.5,3.5) { };    
\draw [color=red, thick, rounded corners]  (6.5,0.5) -- (6.5,2.5);   
\draw [color=red,fill=black,thick] (6.5,2.5) circle (.4ex);
\node [color=red,draw, circle, fill = white, inner sep = 1.2pt] at (6.5,0.5) { }; 
\node [color=red,draw, circle, fill = white, inner sep = 1.2pt] at (9.5,3.5) { };   
    
\node at (8.5,1.5) {$+1$};

\begin{scope}
   \clip (14,1) -| (15,2) -| (17,4) -| (14,1);
    \draw [color=black!25] (14,1) grid (17,4);
\end{scope}

\draw [thick] (14,1) -| (15,2) -| (17,4) -| (14,1);

\draw [thick, rounded corners]  (14.5,2.5) -- (14.5,3.5);
\draw [color=black,fill=black,thick] (14.5,3.5) circle (.4ex);
\node [draw, circle, fill = white, inner sep = 1.2pt] at (14.5,2.5) { };  
\node [draw, circle, fill = white, inner sep = 1.2pt] at (15.5,3.5) { };  
\draw [color=red, thick, rounded corners]  (15.5,2.5) -| (16.5,3.5);   
\draw [color=red,fill=black,thick] (16.5,3.5) circle (.4ex);
\node [color=red,draw, circle, fill = white, inner sep = 1.2pt] at (15.5,2.5) { }; 
\node [color=red,draw, circle, fill = white, inner sep = 1.2pt] at (14.5,1.5) { }; 
\node at (16,0.5) {$+1$};

\begin{scope}
   \clip (22,1) -| (24,3) -| (25,4) -| (22,1);
    \draw [color=black!25] (22,1) grid (25,4);
\end{scope}

\draw [thick] (22,1) -| (24,3) -| (25,4) -| (22,1);

\draw [thick, rounded corners]  (22.5,3.5) -- (23.5,3.5);
\draw [color=black,fill=black,thick] (23.5,3.5) circle (.4ex);
\node [draw, circle, fill = white, inner sep = 1.2pt] at (22.5,3.5) { };  
\node [draw, circle, fill = white, inner sep = 1.2pt] at (24.5,3.5) { };  
\node [color=red,draw, circle, fill = white, inner sep = 1.2pt] at (22.5,2.5) { }; 
\draw [color=red, thick, rounded corners]  (22.5,1.5) -| (23.5,2.5);  
\draw [color=red,fill=black,thick] (23.5,2.5) circle (.4ex);
\node [color=red,draw, circle, fill = white, inner sep = 1.2pt] at (22.5,1.5) { }; 
\node at (24.5,0.5) {$-1$};

\begin{scope}
   \clip (27,1) -| (29,3) -| (30,4) -| (27,1);
    \draw [color=black!25] (27,1) grid (30,4);
\end{scope}

\draw [thick] (27,1) -| (29,3) -| (30,4) -| (27,1);

\draw [thick, rounded corners]  (28.5,3.5) -- (29.5,3.5);
\draw [color=black,fill=black,thick] (29.5,3.5) circle (.4ex);
\node [draw, circle, fill = white, inner sep = 1.2pt] at (28.5,3.5) { };  
\node [draw, circle, fill = white, inner sep = 1.2pt] at (27.5,3.5) { };  
\node [color=red,draw, circle, fill = white, inner sep = 1.2pt] at (27.5,2.5) { }; 
\draw [color=red, thick, rounded corners]  (27.5,1.5) -| (28.5,2.5);  
\draw [color=red,fill=black,thick] (28.5,2.5) circle (.4ex);
\node [color=red,draw, circle, fill = white, inner sep = 1.2pt] at (27.5,1.5) { }; 
\node at (29.5,0.5) {$-1$};

\begin{scope}
   \clip (32,1) -| (34,3) -| (35,4) -| (32,1);
    \draw [color=black!25] (32,1) grid (35,4);
\end{scope}

\draw [thick] (32,1) -| (34,3) -| (35,4) -| (32,1);

\draw [thick, rounded corners]  (32.5,2.5) -- (32.5,3.5);
\draw [color=black,fill=black,thick] (32.5,3.5) circle (.4ex);
\node [draw, circle, fill = white, inner sep = 1.2pt] at (32.5,2.5) { };  
\node [draw, circle, fill = white, inner sep = 1.2pt] at (33.5,3.5) { };  
\node [color=red,draw, circle, fill = white, inner sep = 1.2pt] at (34.5,3.5) { }; 
\draw [color=red, thick, rounded corners]  (32.5,1.5) -| (33.5,2.5);  
\draw [color=red,fill=black,thick] (33.5,2.5) circle (.4ex);
\node [color=red,draw, circle, fill = white, inner sep = 1.2pt] at (32.5,1.5) { }; 
\node at (34.5,0.5) {$+1$};

\begin{scope}
   \clip (38,-1) -| (39,3) -| (41,4) -| (38,-1);
    \draw [color=black!25] (38,-1) grid (41,4);
\end{scope} 

\draw [thick] (38,-1) -| (39,3) -| (41,4) -| (38,-1);

\draw [thick, rounded corners]  (38.5,3.5) -- (39.5,3.5);
\draw [color=black,fill=black,thick] (39.5,3.5) circle (.4ex);
\node [draw, circle, fill = white, inner sep = 1.2pt] at (38.5,3.5) { };  
\node [draw, circle, fill = white, inner sep = 1.2pt] at (40.5,3.5) { };  
\node [color=red,draw, circle, fill = white, inner sep = 1.2pt] at (38.5,2.5) { }; 
\draw [color=red, thick, rounded corners]  (38.5,-0.5) -| (38.5,1.5);  
\draw [color=red,fill=black,thick] (38.5,1.5) circle (.4ex);
\node [color=red,draw, circle, fill = white, inner sep = 1.2pt] at (38.5,-0.5) { }; 
\node at (40.5,0.5) {$+1$};

\begin{scope}
   \clip (43,-1) -| (44,3) -| (46,4) -| (43,-1);
    \draw [color=black!25] (43,-1) grid (46,4);
\end{scope} 

\draw [thick] (43,-1) -| (44,3) -| (46,4) -| (43,-1);

\draw [thick, rounded corners]  (44.5,3.5) -- (45.5,3.5);
\draw [color=black,fill=black,thick] (45.5,3.5) circle (.4ex);
\node [draw, circle, fill = white, inner sep = 1.2pt] at (44.5,3.5) { };  
\node [draw, circle, fill = white, inner sep = 1.2pt] at (43.5,3.5) { };  
\node [color=red,draw, circle, fill = white, inner sep = 1.2pt] at (43.5,2.5) { }; 
\draw [color=red, thick, rounded corners]  (43.5,-0.5) -| (43.5,1.5);  
\draw [color=red,fill=black,thick] (43.5,1.5) circle (.4ex);
\node [color=red,draw, circle, fill = white, inner sep = 1.2pt] at (43.5,-0.5) { }; 
\node at (45.5,0.5) {$+1$};

\begin{scope}
   \clip (48,-1) -| (49,3) -| (51,4) -| (48,-1);
    \draw [color=black!25] (48,-1) grid (51,4);
\end{scope} 

\draw [thick] (48,-1) -| (49,3) -| (51,4) -| (48,-1);

\draw [thick, rounded corners]  (48.5,2.5) -- (48.5,3.5);
\draw [color=black,fill=black,thick] (48.5,3.5) circle (.4ex);
\node [draw, circle, fill = white, inner sep = 1.2pt] at (48.5,2.5) { };  
\node [draw, circle, fill = white, inner sep = 1.2pt] at (49.5,3.5) { };  
\node [color=red,draw, circle, fill = white, inner sep = 1.2pt] at (50.5,3.5) { }; 
\draw [color=red, thick, rounded corners]  (48.5,-0.5) -| (48.5,1.5);  
\draw [color=red,fill=black,thick] (48.5,1.5) circle (.4ex);
\node [color=red,draw, circle, fill = white, inner sep = 1.2pt] at (48.5,-0.5) { }; 
\node at (50.5,0.5) {$-1$};

\end{tikzpicture}
\end{center}

\begin{center}
\begin{tikzpicture}[scale = 0.3]

\begin{scope}
   \clip (0,0) -| (1,1) -| (2,4) -| (0,0);
    \draw [color=black!25] (0,0) grid (2,4);
\end{scope}

\draw [thick] (0,0) -| (1,1) -| (2,4) -| (0,0);

\draw [thick, rounded corners]  (0.5,2.5) -- (0.5,3.5);
\draw [color=black,fill=black,thick] (0.5,3.5) circle (.4ex);
\node [draw, circle, fill = white, inner sep = 1.2pt] at (0.5,2.5) { };    
\node [draw, circle, fill = white, inner sep = 1.2pt] at (1.5,3.5) { };    
\node [color=red,draw, circle, fill = white, inner sep = 1.2pt] at (1.5,2.5) { };  
\draw [color=red, thick, rounded corners]  (0.5,0.5) |- (1.5,1.5);   
\draw [color=red,fill=black,thick] (1.5,1.5) circle (.4ex);
\node [color=red,draw, circle, fill = white, inner sep = 1.2pt] at (0.5,0.5) { }; 
\node at (2.5,0) {$+1$};

\begin{scope}
   \clip (10,-1) -| (11,2) -| (12,4) -| (10,-1);
    \draw [color=black!25] (10,-1) grid (12,4);
\end{scope}

\draw [thick]  (10,-1) -| (11,2) -| (12,4) -| (10,-1);
\draw [thick, rounded corners]  (10.5,2.5) -- (10.5,3.5);
\draw [color=black,fill=black,thick] (10.5,3.5) circle (.4ex);
\node [draw, circle, fill = white, inner sep = 1.2pt] at (10.5,2.5) { };    
\node [draw, circle, fill = white, inner sep = 1.2pt] at (11.5,3.5) { }; 
\node [color=red,draw, circle, fill = white, inner sep = 1.2pt] at (11.5,2.5) { };   
\draw [color=red, thick, rounded corners]  (10.5,-0.5) -- (10.5,1.5);   
\draw [color=red,fill=black,thick] (10.5,1.5) circle (.4ex);
\node [color=red,draw, circle, fill = white, inner sep = 1.2pt] at (10.5,-0.5) { };  
\node at (12.5,0) {$-1$};

\begin{scope}
   \clip (20,-2) -| (21,3) -| (22,4) -| (20,-2);
    \draw [color=black!25] (20,-2) grid (22,4);
\end{scope}

\draw [thick]  (20,-2) -| (21,3) -| (22,4) -| (20,-2);
\draw [thick, rounded corners]  (20.5,2.5) -- (20.5,3.5);
\draw [color=black,fill=black,thick] (20.5,3.5) circle (.4ex);
\node [draw, circle, fill = white, inner sep = 1.2pt] at (20.5,2.5) { };    
\node [draw, circle, fill = white, inner sep = 1.2pt] at (21.5,3.5) { }; 
\node [color=red,draw, circle, fill = white, inner sep = 1.2pt] at (20.5,1.5) { };  
\draw [color=red, thick, rounded corners]  (20.5,-1.5) -- (20.5,0.5);  
\draw [color=red,fill=black,thick] (20.5,0.5) circle (.4ex);
\node [color=red,draw, circle, fill = white, inner sep = 1.2pt] at (20.5,-1.5) { };    
\node at (22.5,0) {$-1$};

\end{tikzpicture}
\end{center}

\caption{The Schur expansion in Example~\ref{atomic-expansion-example}.}
\label{fig:atomic-expansion}
\end{figure}

\begin{example}
\label{atomic-expansion-example}
Let $(I,J) \in \symm_{7,5}$ with $I = [1,4,5,6,7]$ and $J = [2,5,6,4,7]$.
The path type of $(I,J)$ is $\mu = (2,1)$ while the cycle type is $\nu = (3,1)$.  
Figure~\ref{fig:atomic-expansion} illustrates how Corollary~\ref{cor:atomic-schur-expansion} calculates the Schur expansion
\begin{equation*}
A_{7,I,J} = 2 s_7 + s_{61} - s_{52} - s_{511} + 3 s_{43} - 2 s_{421} + 2 s_{41^3} + s_{331} - s_{322} + s_{31^4} + s_{2^3 1} - s_{221^3} - s_{21^5}
\end{equation*} 
of $A_{7,I,J}$.  One first adds a 1-ribbon and a 2-ribbon in both orders in a monotonic fashion; this is shown in black.
Once this is done, one adds a 1-ribbon and a 3-ribbon (in that order) in a fashion which may not be monotonic; this is shown in red.
The signs of both the black and red ribbons contribute to the expansion.
\end{example}

We give our promised combinatorial formula for the character evaluation $\chi^\lambda([I,J]).$

\begin{corollary}
\label{cor:trace-interpretation}
Let $(I,J) \in \symm_{n,k}$ be a partial permutation of path type $\mu \vdash a$ and cycle type $\nu$. We have
\begin{equation}
\chi^{\lambda}([I,J]) = m(\mu)! \cdot \sum_{\rho \, \vdash \, a}
 \vec{\chi}_{ \, \mu}^{ \, \rho} \cdot \chi_{\nu}^{\lambda/\rho}
\end{equation}
where $\chi^{\lambda}: \CC[\symm_n] \rightarrow \CC$ is the irreducible character of $\symm_n$ and
$[I,J] = \sum_{w(I) = J} w$.
\end{corollary}

\begin{proof}
    Apply Proposition~\ref{prop:atomic-character-interpretation} and Corollary~\ref{cor:atomic-schur-expansion}.
\end{proof}

For example, suppose $(I,J)$ is the partial permutation of $[n]$ with $I = J = \varnothing$ so that $[I,J] = \sum_{w \in \symm_n} w$. As explained in the introduction, one can show algebraically that
\[
\chi^\lambda([\varnothing,\varnothing]) = \begin{cases}
    n! & \lambda = (n), \\
    0 & \lambda \neq (n).
\end{cases}
\]
On the other hand, Corollary~\ref{cor:atomic-schur-expansion} gives $A_{n,I,J} = n! \cdot s_{n}$ since the only monotonic ribbon tiling consisting
of $n$ ribbons of size 1 is 
\begin{center}
\begin{tikzpicture}[scale = 0.3]
\begin{scope}
   \clip (0,0) -| (2,1) -|  (0,0);
    \draw [color=black!25] (0,0) grid (10,1);
\end{scope}

\begin{scope}
   \clip (8,0) -| (10,1) -|  (8,0);
    \draw [color=black!25] (0,0) grid (10,1);
\end{scope}

\draw [thick]  (0,0) -| (10,1) -| (0,0);
\node [draw, circle, fill = white, inner sep = 1.2pt] at (0.5,0.5) { };   
\node [draw, circle, fill = white, inner sep = 1.2pt] at (1.5,0.5) { };   
\node [draw, circle, fill = white, inner sep = 1.2pt] at (9.5,0.5) { };   
\node [draw, circle, fill = white, inner sep = 1.2pt] at (8.5,0.5) { };   

\node at (5,0.5) {$\cdots$};
\end{tikzpicture}
\end{center}
which has shape $(n)$.
At the other extreme, if $(I,J) \in \symm_{n,n} = \symm_n$ is a genuine (rather than merely partial) permutation, 
Corollary~\ref{cor:trace-interpretation} reduces to the classical Murnaghan--Nakayama rule. We record a lemma on monotonic tilings for future use.

\begin{lemma}
    \label{lem:monotonic-restriction}
    Let $\lambda, \mu \vdash n$ and suppose that $\ell(\mu) > \lambda_1$. There are no monotonic ribbon tilings of shape $\lambda$ with $\sort(\type(T)) = \mu$.
\end{lemma}

\begin{proof}
    The tails of the $\ell(\mu)$ ribbons such a monotonic tiling would occur in distinct columns.
\end{proof}


\subsection{Stability and the $n \to \infty$ limit} Let $I = (i_1, \dots, i_k)$ and $J = (j_1, \dots, j_k)$ be two length $k$ sequences of distinct positive integers. Then $(I,J)$ is a size $k$ partial permutation of $[n]$ whenever $n \geq \max(i_1, \dots,i_k,j_1,\dots,j_k)$, and one has  $[I,J] = \sum_{w(I) = J} w \in \CC[\symm_n]$.  Corollary~\ref{cor:trace-interpretation} is a signed combinatorial rule for evaluating $\chi^\lambda([I,J])$ for any $\lambda \vdash n$.
Using this rule, we show the character evaluations $\{ \chi^\lambda([I,J]) \,:\, \lambda \vdash n \}$ exhibit a stability phenomenon -- for {\em all} $n$ sufficiently large they are calculated using a fixed and finite set of tilings.
By comparison, a na\"ive computation of these characters would apply the classical Murnaghan--Nakayam rule $(n-k)!$ times for each $\lambda \vdash n$.

The efficiency gain from Corollary~\ref{cor:trace-interpretation} as $n \to \infty$ relies on a decomposition of monotonic ribbon tilings. Let $T$ be a monotonic ribbon tiling and let $\xi$ be a ribbon in $T$. We say $\xi$ is {\em frozen} if its tail is not in the first row and {\em tropical} otherwise. Let $T_0$ be the union of the frozen ribbons of $T$ and $T_1$ be the union of the tropical ribbons of $T$. Then
\begin{equation}
\label{eq:frozen-tropical-decomposition}
    T = T_0\, \sqcup \, T_1.
\end{equation}

\begin{observation}
    \label{obs:frozen-tropical}
    The frozen region $T_0$ of $T$ is a (left justified) Young diagram while the tropical region $T_1$ lies entirely within the first row of $T$ and to the right of $T_0$.
\end{observation}

Observation~\ref{obs:frozen-tropical} is shown schematically below, with the frozen region in blue and the tropical region in red. The long first row and the large number of size 1 ribbons are relevant in the $n \to \infty$ situation of interest.

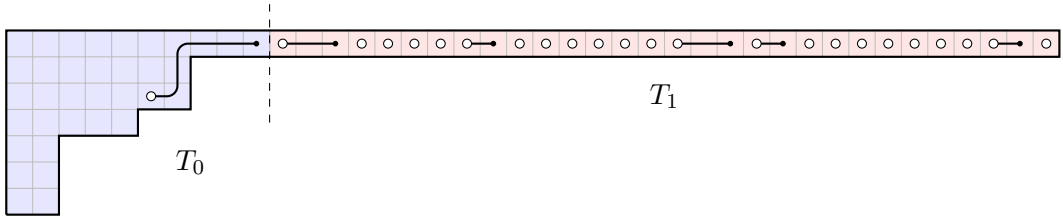
\begin{center}
\begin{tikzpicture}[scale = 0.35]

\draw [fill = blue!10] (0,0) -| (2,3) -| (5,4) -| (7,6) -| (10,7) -| (0,0) ;
\draw [fill = red!10] (10,6) -| (40,7) -| (10,6);

\begin{scope}

\clip (0,0) -| (2,3) -| (5,4) -| (7,6) -| (40,7) -| (0,0) ;
  \draw [color=black!25] (0,0) grid (40,7);

\end{scope}

\draw [thick]  (0,0) -| (2,3) -| (5,4) -| (7,6) -| (40,7) -| (0,0) ;

\draw [dashed] (10,3.5) -- (10,8);

  \draw [thick, rounded corners]  (9.5,6.5) -| (6.5,4.5) -- (5.5,4.5) ;
  \draw [color=black,fill=black,thick] (9.5,6.5) circle (.4ex);
  \node [draw, circle, fill = white, inner sep = 1.2pt] at (5.5,4.5) { }; 
  
  \node at (7,2) {$T_0$}; 
  
   \draw [thick, rounded corners]  (10.5,6.5) -- (12.5,6.5) ;
  \draw [color=black,fill=black,thick] (12.5,6.5) circle (.4ex);
  \node [draw, circle, fill = white, inner sep = 1.2pt] at (10.5,6.5) { }; 
  \node [draw, circle, fill = white, inner sep = 1.2pt] at (13.5,6.5) { }; 
  \node [draw, circle, fill = white, inner sep = 1.2pt] at (14.5,6.5) { }; 
  \node [draw, circle, fill = white, inner sep = 1.2pt] at (15.5,6.5) { }; 
  \node [draw, circle, fill = white, inner sep = 1.2pt] at (16.5,6.5) { }; 
     \draw [thick, rounded corners]  (17.5,6.5) -- (18.5,6.5) ;
  \draw [color=black,fill=black,thick] (18.5,6.5) circle (.4ex);
  \node [draw, circle, fill = white, inner sep = 1.2pt] at (17.5,6.5) { }; 
  \node [draw, circle, fill = white, inner sep = 1.2pt] at (19.5,6.5) { }; 
  \node [draw, circle, fill = white, inner sep = 1.2pt] at (20.5,6.5) { }; 
   \node [draw, circle, fill = white, inner sep = 1.2pt] at (21.5,6.5) { }; 
  \node [draw, circle, fill = white, inner sep = 1.2pt] at (22.5,6.5) { }; 
   \node [draw, circle, fill = white, inner sep = 1.2pt] at (23.5,6.5) { }; 
     \node [draw, circle, fill = white, inner sep = 1.2pt] at (24.5,6.5) { }; 
  \draw [thick, rounded corners]  (25.5,6.5) -- (27.5,6.5) ;
  \draw [color=black,fill=black,thick] (27.5,6.5) circle (.4ex);
  \node [draw, circle, fill = white, inner sep = 1.2pt] at (25.5,6.5) { }; 
  \draw [thick, rounded corners]  (28.5,6.5) -- (29.5,6.5) ;
  \draw [color=black,fill=black,thick] (29.5,6.5) circle (.4ex);
  \node [draw, circle, fill = white, inner sep = 1.2pt] at (28.5,6.5) { }; 
    \node [draw, circle, fill = white, inner sep = 1.2pt] at (30.5,6.5) { }; 
    \node [draw, circle, fill = white, inner sep = 1.2pt] at (31.5,6.5) { }; 
 \node [draw, circle, fill = white, inner sep = 1.2pt] at (32.5,6.5) { }; 
 \node [draw, circle, fill = white, inner sep = 1.2pt] at (33.5,6.5) { }; 
  \node [draw, circle, fill = white, inner sep = 1.2pt] at (34.5,6.5) { }; 
  \node [draw, circle, fill = white, inner sep = 1.2pt] at (35.5,6.5) { }; 
   \node [draw, circle, fill = white, inner sep = 1.2pt] at (36.5,6.5) { }; 
     \draw [thick, rounded corners]  (37.5,6.5) -- (38.5,6.5) ;
  \draw [color=black,fill=black,thick] (38.5,6.5) circle (.4ex);
  \node [draw, circle, fill = white, inner sep = 1.2pt] at (37.5,6.5) { }; 
  \node [draw, circle, fill = white, inner sep = 1.2pt] at (39.5,6.5) { };

\node at (25,4.5) {$T_1$}; 

\end{tikzpicture}
\end{center}

By Observation~\ref{obs:frozen-tropical}, the tropical portion of a monotonic tiling is entirely contained in the first row, so
\begin{equation}
    \label{eqn:sign-relation}
    \sign(T) = \sign(T_0)
\end{equation}
whenever $T$ is a monotonic ribbon tiling with frozen portion $T_0$. Given $\mu \vdash n$, by Theorem~\ref{thm:path-Murnaghan--Nakayama}  
\begin{equation}
\label{eq:reduced-path-mn}
    \vec{p}_\mu = m(\mu)! \cdot \sum_T \sign(T) \cdot s_{\shape(T)} = m(\mu)! \cdot \sum_T \sign(T_0) \cdot s_{\shape(T)}
\end{equation}
where the sum is over monotonic ribbon tilings $T$ whose with $\sort(\type(T)) = \mu$.

We use Equation~\eqref{eq:reduced-path-mn} to study the path Murnaghan--Nakayama rule in the $n \to \infty$ limit.  Recall that if $\lambda$ is a partition and $n \geq |\lambda| + \lambda_1$, the padded partition $\lambda[n] := (n-|\lambda|,\lambda_1,\lambda_2, \cdots ) \vdash n$ is obtained by prepending a row of length $n - |\lambda|$. Similarly, if $\mu = (\mu_1, \dots, \mu_r)$ is a partition and $n \geq |\mu|$, we define 
\[\mu(n) := (\mu_1, \mu_2, \dots, \mu_r, 1^{n-|\mu|}) \vdash n\]
by adding $n - |\mu|$ boxes to the first column of $\mu$. 

For a fixed partition $\mu$, our goal is to state a finite formula for the Schur expansion of $\vec{p}_{\mu(n)}$ for all $n$ sufficiently large. In order to do this, we will need some terminology. Let $T = T_0 \sqcup T_1$ be a monotonic ribbon tiling with frozen region $T_0$. 
If $T$ has $\leq n$ boxes, define 
\begin{equation}
    \sigma(T,n) := \shape(T_0)[n],
\end{equation}
which is the partition obtained by adding $n - |\shape(T)|$ boxes to the first row of $\shape(T)$.
The tiling $T$ is {\em frozen} if $T = T_0$.
As before, we write $\rho_i(T_0)$ for the number of frozen ribbons of size $i$.

\begin{proposition}
\label{prop:path-power-formula}
Let $\mu$ be a partition whose parts are all of size $> 1$. 
\begin{enumerate}
\item
For any $n \geq |\mu|$ we have 
\begin{multline}
\label{eq:path-power-sum-expansion}
    \vec{p}_{\mu(n)} =  (n - |\mu|)! \cdot m(\mu)! \times \\ \sum_{T_0} \sign(T_0) \cdot {n - |\mu| + \ell(\mu) - \rho(T_0) \choose n - |\mu| - \rho_1(T_0), m_2(\mu) - \rho_2(T_0), m_3(\mu) - \rho_3(T_0), \dots } \cdot s_{\sigma(T_0,n)}
\end{multline}
where the sum is over all frozen monotonic ribbon tableaux $T_0$ which have $\leq n - |\mu|$ ribbons of size 1 and $\leq m_i(\mu)$ ribbons of size $i$ for $i > 1$.
\item The set of frozen tilings $T_0$ indexing the sum in Equation~\eqref{eq:path-power-sum-expansion} is the same (and finite) for all $n \geq 2 \cdot (|\mu| - \ell(\mu))$. 
\item If $s_{\sigma(T_0,n)}$ appears with nonzero coefficient in Equation~\eqref{eq:path-power-sum-expansion},  the difference in lengths between the first and second rows of $\sigma(T_0,n)$ is bounded below according to \[ \sigma(T_0,n)_1 - \sigma(T_0,n)_2 \geq n - 2 \cdot (|\mu| - \ell(\mu)).\]
\end{enumerate}
\end{proposition}

Equation~\eqref{eq:path-power-sum-expansion} may appear forbidding, but we will see the multinomial coefficient has a simple interpretation. Proposition~\ref{prop:path-power-formula} (2) says the sum has finitely many terms, stabilizing when $n = 2 \cdot (|\mu| - \ell(\mu))$. Proposition~\ref{prop:path-power-formula} (3) is a technical condition required to establish an analogous finiteness property when we augment our monotonic tilings with skew classical ribbon tableaux to find the character evaluations $\chi^\lambda([I,J])$.

\begin{proof}
    We start by proving (1). Let $n \geq |\mu|$. Equation~\eqref{eq:reduced-path-mn} reads
    \begin{equation}
    \label{eq:degree-bound-one}
        \vec{p}_{\mu(n)} = m(\mu(n))! \cdot \sum_T \sign(T_0) \cdot s_{\shape(T)} = (n - |\mu|)! \cdot m(\mu)! \cdot \sum_T \sign(T_0) \cdot s_{\shape(T)}
    \end{equation}
    where the sum is over monotonic ribbon tilings $T$ with $\sort(\type(T)) = \mu(n)$ and $T_0$ is the frozen region of $T$. In particular, if $T_0$ contributes to the sum in \eqref{eq:degree-bound-one} then $\rho_i(T_0) \leq m_i(\mu)$ for all $i > 1$ and $\rho_1(T_0) \leq n - |\mu|$. 
    
     For a fixed frozen region $T_0$, how many ways are there to extend $T_0$ to a monotonic ribbon tiling $T = T_0 \sqcup T_1$ with $\sort(\type(T)) = \mu(n)$? The overall tiling $T$ contains ribbons of sizes $m_i(\mu)$ ribbons of size $i$ for each $i > 1$ together with $n - |\mu|$ singleton ribbons.
Each of these ribbons is in $T_0$ or in $T_1$.  
We have
\begin{multline}
\label{eq:multinomial-coefficient}
\# \{ \text{ribbon tilings $T = T_0 \sqcup T_1$ of $\lambda[n]$ extending $T_0$} \}  \\ = 
\binom{n - |\mu| + \ell(\mu) - \rho(T_0)}{ n - |\mu| - \rho_1(T_0), m_2(\mu) - \rho_2(T_0), m_3(\mu) - \rho_3(T_0), \dots }
\end{multline}
where the multinomial coefficient counts ways order ribbons within the tropical region $T_1$ from left to right. Furthermore, we have 
\begin{equation}
    \shape(T) = \sigma(T_0,n)
\end{equation}
since $T$ is formed from $T_0$ by adding boxes to the first row. Plugging Equation~\eqref{eq:multinomial-coefficient} into Equation~\eqref{eq:degree-bound-one} proves (1).

Let $n \geq |\mu|$ and let $T = T_0 \sqcup T_1$ be a monotonic ribbon tiling with $\sort(\type(T)) = \mu(n)$. Our proof of (2) rests on the following Claim.
    \begin{quote}
        {\bf Claim:} {\em The frozen region $T_0$ of $T$ contains at most $|\mu| - 2 \cdot \ell(\mu)$ ribbons of size $1$ for all $n \geq |\mu|$.}
    \end{quote}
This Claim is proven as follows. Let $\xi$ be a size 1 ribbon in $T_0$. We know that
\begin{itemize}
    \item the size 1 ribbon $\xi$ occurs outside the first row (for otherwise $\xi$ would be in $T_1$), and
    \item the box $b(\xi)$ immediately above $\xi$ in $T_0$ is a non-tail box in a frozen ribbon $\xi'$ of size $> 1$ (for otherwise there would be two ribbon tails in the same column, in violation of Definition~\ref{def:monotonic-definition}).
\end{itemize}
On the other hand, suppose that $\xi'$ is a ribbon in $T_0$ of size $> 1$. We claim that at least two boxes of $\xi'$ are not immediately above a size 1 ribbon in $T_0$. For otherwise the above two bullet points would force the structure of $T_0$ near $\xi'$ to have the form
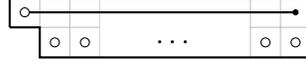
\begin{center}
\begin{tikzpicture}[scale = 0.4]
\begin{scope}
\clip (0,0) |- (3,1) |-  (1,-1) |- (0,0);
\draw [color=black!25] (0,-1) grid (3,2);
\end{scope}

 \begin{scope}
   \clip (8,-1) -| (10,1) -|  (8,-1);
    \draw [color=black!25] (8,-1) grid (10,1);
\end{scope}

\draw[thick] (0,0) -| (1,-1) -| (10,1) -| (0,0);

\draw [thick, rounded corners]  (0.5,0.5) -- (9.5,0.5);
  \draw [color=black,fill=black,thick] (9.5,0.5) circle (.4ex);
  \node [draw, circle, fill = white, inner sep = 1.2pt] at (0.5,0.5) {}; 

  \node [draw, circle, fill = white, inner sep = 1.2pt] at (1.5,-0.5) {};
  \node [draw, circle, fill = white, inner sep = 1.2pt] at (2.5,-0.5) {};
  \node [draw, circle, fill = white, inner sep = 1.2pt] at (8.5,-0.5) {};
  \node [draw, circle, fill = white, inner sep = 1.2pt] at (9.5,-0.5) {};

  \node at (5.5,-0.5) {$\cdots$};
\end{tikzpicture}
\end{center}
where $\xi'$ is the ribbon on top. This configuration violates Definition~\ref{def:monotonic-definition} because the position immediately below the tail of $\xi'$ must be occupied by a ribbon preceding $\xi'$. Let $\AAA$ be the set of size $1$ ribbons in $T_0$ and let $\BBB$ be the set of boxes which are not in size 1 ribbons. We have a function $b: \AAA \to \BBB$ where $b(\xi)$ is the box immediately above $\xi$. The function $b$ is injective, and each ribbon $\xi'$ of size $> 1$ contains at least two boxes which are not in the image of $b$. It follows that
\begin{equation} \# \AAA \leq \# \BBB - 2 \cdot \sum_{i \, > \, 1} \rho_i(T_0) = \sum_{i \, > \, 1} i \cdot \rho_i(T_0) - 2 \cdot \sum_{i \, > \, 1} \rho_i(T_0) \leq |\mu| - 2 \cdot \ell(\mu) \end{equation}
where the last inequality uses the fact that every part of $\mu$ has size $> 1$. This completes the proof of the Claim.
    
To see why (2) follows from the Claim, fix $n \geq 2 \cdot (|\mu| - \ell(\mu))$ and let $T = T_0 \sqcup T_1$ be the frozen-tropical decomposition of a monotonic ribbon tiling with $\sort(\type(T)) = \mu(n+1)$. We assert that the tropical region $T_1$ contains at least one ribbon of size 1. If not, then all $n + 1 - |\mu|$ of the size 1 ribbons in $T$ would be frozen, so the number of size 1 ribbons in $T_0$ would be
\[n+1 - |\mu| > 2 \cdot (|\mu| - \ell(\mu)) - |\mu| = |\mu| - 2 \cdot \ell(\mu),\]
which contradicts the Claim.

Finally, we prove (3). Let $T$ be a monotonic ribbon tiling whose multiset of ribbon sizes is $\mu(n)$. Lemma~\ref{lem:monotonic-restriction} implies that 
\begin{equation*}
    \begin{array}{c}
        \text{number of boxes outside} \\ \text{the first row of $T$}
    \end{array} \leq n - \ell(\mu(n)) = |\mu| - \ell(\mu).
\end{equation*} 
Since $\shape(T)$ is a partition of $n$, this forces
\begin{equation*}
    \shape(T)_1 - \shape(T)_2 \geq n - 2 \cdot (|\mu|-\ell(\mu)).
\end{equation*}
If $T = T_0 \sqcup T_1$ is the decomposition of $T$ into frozen and tropical regions, we have
\[ \shape(T) = \sigma(T_0,n)\]
and the proof is complete.
\end{proof}

We use the following method for computing $\{\chi^\lambda([I,J]) \,:\, \lambda \vdash n\}$ for a partial permutation $(I,J)$.
\begin{enumerate}
    \item Factor the atomic symmetric function $A_{n,I,J} = \vec{p}_\mu \cdot p_\nu$ where $\mu$ is the path type of $(I,J)$ and $\nu$ is the cycle type of $(I,J)$. (Proposition~\ref{prop:path-cycle-factorization})
    \item Compute the Schur expansion of $\vec{p}_\mu$ using Proposition~\ref{prop:path-power-formula} (1).
    \item Use the classical Murnaghan--Nakayama rule and the result from the previous step to compute the Schur expansion of $A_{n,I,J} = \sum_{\lambda \vdash n} \chi^\lambda([I,J]) \cdot s_\lambda$.
\end{enumerate}
We close this section by showing that this method requires a constant (and finite) amount of computation whenever $n \geq 2k$, where $k$ is the size of the partial permutation $(I,J)$.

Let $I = (i_1,\dots,i_k)$ and $J = (j_1, \dots, j_k)$ be two sequences of distinct positive integers of the same length $k$. Without loss of generality, assume that the pair $(I,J)$ is {\em packed}, i.e.
$I \cup J = \{1, \dots, r\} \text{ for some } r \geq 0.$
Regarding $(I,J)$ as a partial permutation in $\symm_{r,k}$, let $\mu$ be the path type of $(I,J)$ and let $\nu$ be the cycle type of $(I,J)$. We make the observations
\begin{equation}
\label{eq:efficiency-conditions}
\text{every part of $\mu$ is $> 1$ and $|\mu| + |\nu| - \ell(\mu) = k$}
\end{equation}
where the first statement is true because $(I,J)$ is packed and the second statement applies Observation~\ref{obs:path-cycle-type}.

For any $n \geq r$, we may view $(I,J)$ as a partial permutation in $\symm_{n,k}$. It is not hard to see that the path type of $(I,J) \in \symm_{n,k}$ is $\mu(n - |\nu|)$ and the cycle type of $(I,J) \in \symm_{n,k}$ is $\nu$. By Proposition~\ref{prop:atomic-character-interpretation}, the irreducible character evaluations $\{ \chi^\lambda([I,J]) \,:\, \lambda \vdash n \}$ are the coefficients in the Schur expansion
\begin{equation}
\label{eq:atomic-expansion-limit}
    A_{n,I,J} = \vec{p}_{\mu(n-|\nu|)} \cdot p_\nu = \sum_{\lambda \, \vdash \, n} \chi^\lambda([I,J]) \cdot s_\lambda.
\end{equation}
By Proposition~\ref{prop:path-power-formula} (1) and the classical Murnaghan--Nakayama rule, the Schur expansion in Equation~\eqref{eq:atomic-expansion-limit} is 
\begin{multline}
\label{eq:double-power-expansion}
    A_{n,I,J} =  \vec{p}_{\mu(n-|\nu|)} \cdot p_\nu = m(\mu)! \cdot (n - |\mu| - |\nu|)! \times \\
    \sum_{\lambda \, \vdash \, n} \left(
    \sum_{T_0} \sign(T_0) \cdot {n - |\mu| - |\nu| + \ell(\mu) - \rho(T_0) \choose n - |\mu| - |\nu| - \rho_1(T_0), m_2(\mu) - \rho_2(T_0), m_3(\mu) - \rho_3(T_0), \dots } \cdot \chi^{\lambda/\sigma(T_0,n-|\nu|)}_\nu \right) \cdot s_\lambda
\end{multline}
where the inner sum is over frozen monotonic ribbon tilings $T_0$ for which $\rho_1(T_0) \leq n - |\mu| - |\nu|$ and $\rho_i(T_0) \leq m_i(\mu)$ for all $i > 1$. Equation~\eqref{eq:atomic-expansion-limit} is less than aesthetic at first glance, but it has  computational and theoretical utility.
For example, one consequence is:
\begin{theorem}
     \label{rmk:degree-bounds} For $(I,J)$ a partial permutation of size $k$ and $n \geq \max(I),\max(J), 2k$, there are polynomials $\{ c_\lambda(n) \,:\, |\lambda| \leq k \}$ with $\deg c_\lambda(n) \leq k - |\lambda|$ such that 
    \[  A_{n,I,J} = (n - \#(I \cup J))! \times \sum_{|\lambda| \, \leq \, k} c_\lambda(n) \cdot s_{\lambda[n]}. \]
\end{theorem}
\begin{proof}
    Corollary~\ref{cor:atomic-support-bound} shows only partitions satisfying $|\lambda| \leq k$ are needed.
To show $c_\lambda(n)$ is a polynomial, observe that each term in~\eqref{eq:double-power-expansion} contributing to $c_\lambda(n)$ is a polynomial in $n$, with the factor $(n-|\mu|-|\nu|)!$ subsumed by the $(n - \#(I \cup J))!$ term since $\#(I \cup J) = |\mu| + |\nu| + \ell(\mu)$.
To bound the degree of $c_\lambda(n)$, note the degree of each multinomial in~\eqref{eq:double-power-expansion} is equal to $\ell(\mu) - \ell(\shape(T_0)) + \rho_1(T_0)$.
This means the degree of $c_\lambda(n)$ decreases by one for each ribbon in a contributing $T_0$ of length greater than one, while $|\lambda|$ increases by at least one for each such ribbon.
The partition $\lambda/\sigma(T_0,n-|\nu|)$ has size $n - |\nu| \leq k$, so $\chi^{\lambda/\sigma(T_0,n-|\nu|)}_\nu$ is constant as $n$ varies.
Since each term satisfies the desired degree contribution, the result follows.
\end{proof}

\section{Statistic class functions}
\label{sec:Statistic}

Our original motivation for introducing the symmetric functions $A_{n,I,J}$ was to better understand class functions given by symmetrized permutation statistics.
In this section we discuss how to apply our results for this purpose.
A key tool for this approach is the theory of character polynomials, introduced in~\cite{Specht}.
We apply this theory to several examples including exceedances, inversion count.
The interested reader can find several other examples and additional details in our original paper~\cite{HR}.

\subsection{Applications to permutation statistics}

We begin by stating:
\begin{theorem}[Character polynomials]
    \label{t:character-polynomial}
    For $\lambda \vdash n$ and $w \in \symm_n$ of cycle type $\mu$, there is a multivariate polynomial $f_\lambda$ with rational coefficients so that
\[
\chi^{\lambda}(w) = f_\lambda(m_1(\mu),m_2(\mu),\dots,m_n(\mu))
\]    
where $\deg(f_\lambda) = n-\lambda_1$ with the grading where $\deg(m_k) =k$.
\end{theorem}

Viewing Theorem~\ref{rmk:degree-bounds} as the Frobenius character of $R_n\, 1_{IJ}$ (with constant factor) and applying Theorem~\ref{t:character-polynomial}, we obtain a result equivalent to Theorem 1.1 from our companion paper~\cite{HRMoment}:

\begin{corollary}
    \label{c:character-version}
    For $(I,J)$ a partial permutation of size $k$, there is a polynomial $f_{IJ}(n,m_1,\dots,m_k)$ of degree $k$ where $\deg n = 1$ and $\deg m_i = i$ so that
    \[
    \frac{1}{n - \#(I \cup J)} R_n\,1_{IJ} = f_{IJ}(n,m_1,\dots,m_k).
    \]
\end{corollary}

The proof of Corollary~\ref{c:character-version} in~\cite{HRMoment} is entirely combinatorial, avoiding representation theory and symmetric function theory entirely.
Using the fact that characters form a basis for class functions, results from~\cite{HRMoment} can be applied to recover Theorem~\ref{rmk:degree-bounds}.
However, our proof in~\cite{HRMoment} is indirect, relying on M\"obius inversion on the set partition lattice.
Therefore, the methods in this paper are easier to render effective.

A \emph{permutation statistic} is a function $\Psi:\sqcup_{n\geq 1} \symm_n \to \mathbb{C}$.
Many classical permutation statistics can be expressed as linear combinations of $1_{IJ}$'s.
Recall that the Reynolds operator is the projection $R: \Fun(\symm_n,\CC) \to \Class(\symm_n,\CC)$ defined by $(R \, f)(w) = \frac{1}{n!} \sum_{v \in \symm_n} f(v^{-1}wv)$.
Applying the Reynolds operator $R_n$ to $\Psi$, we obtain a class function whose associated character is then a linear combination of $A_{n,I,J}$'s.
As a consequence of Theorem~\ref{rmk:degree-bounds}, we see the expansion of $R_n\,\Psi$ into irreducible characters would have coefficients with growth governed by the polynomials $c_\lambda(n)$.
Further applying Corollary~\ref{c:character-version}, we obtain an explicit expansion of $R_n\,\Psi$ as a function of $n$ and $m_1,\dots,m_k$.
Our companion paper~\cite{HRMoment} introduces a family of \emph{regular statistics} whose expansion into $1_{IJ}$'s is particularly well-suited to such expansions.
Here, we limit our focus to several special cases, all of which are classical.

 For any $f: \symm_n \to \CC$ we have a symmetric function $\ch_n(R \, f) \in \Lambda_n$. Various authors \cite{GR, GP, Hultman} have studied the Schur expansion of $\ch_n(R \, f)$.
Our original motivation for introducing the path Murnaghan--Nakayama rule was to calculate these expansions systematically, which we explain how to do in this section before presenting examples of several special cases.

\subsection{Examples} We present an in-depth analysis of the exceedance statistic $\exc$  and major index statistic $\maj$ to demonstrate our methods.

\begin{example}[Exceedance]
The {\em exceedance number} of $w \in \symm_n$ is  $\exc(w) := \# \{ 1 \leq i \leq n \,:\, w(i) > i \}$. The statistic $\exc: \symm_n \to \CC$ is $1$-local with
\begin{equation}
\label{eq:exc-definition}
    \exc = \sum_{1 \leq i < j \leq n} \one_{i,j}.
\end{equation}
The Schur expansion of $\ch_n(R \, \exc)$ is
\begin{equation}
\label{eq:exceedance}
    \ch_n(R \, \exc) = \frac{n-1}{2} s_n - \frac{1}{2} s_{n-1,1},
\end{equation}
which first appears in~\cite{Hultman}.
We now rederive this result using the the path Murhaghan-Nakayama rule.

Applying the composition $(\ch_n \circ R)$ to both sides of \eqref{eq:exc-definition} gives
\begin{equation}
\label{eq:exc-one}
    \ch_n( R\, \exc) = \ch_n \left( \sum_{1 \leq i < j \leq n} R \, \one_{i,j} \right) = \frac{1}{n!} \sum_{1 \leq i < j \leq n} A_{n,i,j} = \frac{{n \choose 2}}{n!} A_{n,1,2}
\end{equation}
where the last equality follows from Proposition~\ref{prop:atomic-conjugacy-independence}.
The graph of  $((1),(2)) \in \symm_{n,2}$ consists of one size 2 path and $n-2$ size 1 paths
By Proposition~\ref{prop:path-cycle-factorization} and the path Murnaghan--Nakayama rule (Theorem~\ref{thm:path-Murnaghan--Nakayama}) one has 
\begin{equation}
\label{eq:exc-two}
   A_{n,1,2} =  \vec{p}_{21^{n-2}} = (n-2)! \times ((n-1) \cdot s_n - s_{n-1,1}).
\end{equation}
Then~\eqref{eq:exceedance} follows by combining \eqref{eq:exc-one} and \eqref{eq:exc-two}.

With a bit more effort, we can apply the above method of proof to the statistic $\exc^2: \symm_n \to \CC$.  Lemma~\ref{lem:multiply-by-singleton} gives the indicator expansion of $\exc^2$ as
\begin{multline}
\label{eq:exc-2-one}
\exc^2 = 
    \left( \sum_{1 \leq i < j \leq n} \one_{i,j} \right)^2  \\ = 
    \sum_{1 \leq i < j \leq n} \one_{i,j} + 2 \cdot  \sum_{1 \leq a < b < c \leq n} \one_{ab,bc} + 2 \cdot \sum_{1 \leq a < b < c < d \leq n} ( \one_{ab,cd}  + \one_{ab,dc} + \one_{ac,bd} ).
\end{multline}
Applying $(\ch_n \circ R)$ to both sides of Equation~\eqref{eq:exc-2-one} and taking Proposition~\ref{prop:atomic-conjugacy-independence} into account gives 
\begin{equation}
    \label{eq:exc-2-two}
    \ch_n(R \, \exc^2) = \frac{{n \choose 2}}{n!}  \cdot A_{n,1,2} + \frac{2{n \choose 3}}{n!} \cdot  A_{n,12,23} + \frac{6 {n \choose 4}}{n!} \cdot A_{n,12,34}.
\end{equation}
By Proposition~\ref{prop:path-cycle-factorization} and Theorem~\ref{thm:path-Murnaghan--Nakayama} one has
\begin{equation}
    \label{eq:exc-2-three}
    \begin{cases}
        A_{n,12,23} = \vec{p}_{3,1^{n-3}} = (n-3)! \times ((n-2) s_n - s_{n-1,1} - s_{n-2,2} + s_{n-2,1,1}) \\
        A_{n,12,34} = \vec{p}_{2^2,1^{n-4}} = 2! \cdot (n-4)! \times ({n-2 \choose 2} s_n - (n-3) s_{n-1,1} + s_{n-2,2}).
    \end{cases}
\end{equation}
Combining Equations~\eqref{eq:exc-two}, \eqref{eq:exc-2-two}, and \eqref{eq:exc-2-three}, one obtains
\begin{equation}
    \label{eq:exc-2-four}
    \ch_n(R \, \exc^2) = \frac{3n^2 + n + 2}{12} s_n - \frac{3n-4}{6} s_{n-1,1} + \frac{1}{6} s_{n-2,2} + \frac{1}{3} s_{n-2,1,1}.
\end{equation}
Using Theorem~\ref{t:character-polynomial} and~\cite[Fig. 3]{HR} we have
\begin{equation}
    R \, \exc = \frac{n-m_1}{2}, \quad R \, \exc^2 = \frac{1}{12} \cdot (3m_1^2 - 6 m_1 n + 3n^2 -2m_2-m_1+1)
\end{equation}
so the variance of $\exc$ for a uniformly random element of a given cycle type is
\begin{equation}
    R \, \exc^2 - (R \, \exc)^2 = \frac{n - m_1 - 2m_2}{12}.
\end{equation}
Note Theorem~\ref{rmk:degree-bounds} only applies for $n > 4$.
However, the reader can check directly this formula holds for smaller $n$ as well, a special case of Proposition~5.28 in~\cite{HR}.
\end{example}

\begin{example}[Major index]
The major index statistic
\begin{equation}
    \maj(w) = \sum_{i:\  w(i) > w(i)+1} i
\end{equation} decomposes into indicator functions as
\begin{equation}
\label{eq:inv-definition}
    \maj = \sum_{\substack{1 \leq i  \leq n-1 \\ 1 \leq j < k \leq n}} i \cdot  \one_{(i, i+1), (k, j)}.
\end{equation}
The Schur expansion of $\ch_n(R \, \maj)$, as first computed in~\cite{Hultman}, is
\begin{equation}
\label{eq:inversion}
    \ch_n(R \, \maj) = \frac{n(n-1)}{4} s_n - \frac{1}{2} s_{n-1,1} - \frac{1}{2} s_{n-2,1,1}.
\end{equation}
This result can be derived using the path Murnaghan--Nakayama rule as follows.
We decompose Equation~\eqref{eq:inv-definition} according to the relative order of the values $i, i+1,j,k$ to obtain
    \begin{align}
\nonumber
\maj =& \sum_{i} i \cdot \one_{(i,i+1)(i+1,i)} + \sum_{i < i+1 < j} i \cdot \one_{(i,i+1)(j,i)} + \sum_{i < i+1 < j} i \cdot \one_{(i,i+1)(j,i+1)} \\
&+ \sum_{i <j} j \cdot \one_{(j,j+1),(j,i)} + \sum_{i <j} j \cdot \one_{(j,j+1)(j+1,i)} +
\sum_{1 \leq i < i+1 < j < k} i \cdot \one_{(i,i+1)(k,j)} \label{eq:maj-example-two} \\ 
\nonumber
&+ \sum_{i < j < j+1 <k} j \cdot \one_{(j,j+1)(k,i)} + \sum_{i < j < k < k+1} k \cdot \one_{(k,k+1)(j,i)}.
\end{align}
 Apply $(\ch_n \circ R)$ to both sides of Equation~\eqref{eq:maj-example-two}; grouping like terms via Proposition~\ref{prop:atomic-conjugacy-independence} one obtains 
    \begin{equation}
        \label{eq:inv-two}
        \ch_n(R \, \inv) = \frac{1}{n!} \left(\frac{(n)_4}{4} A_{n,12,43} + \frac{(n)_3}{2} A_{n,12,31} + \frac{(n)_3}{2} A_{n,12,32} + \frac{(n)_2}{2}  A_{n,12,21} \right).
    \end{equation}
    By Proposition~\ref{prop:path-cycle-factorization} one has
    \begin{equation}
    \label{eq:inv-three}
         A_{n,12,43} = \vec{p}_{2^2 1^{n-4}}, \quad 
         A_{n,12,31} = \vec{p}_{31^{n-3}}, \quad
         A_{n,12,32} = \vec{p}_{21^{n-3}} \cdot p_1, \quad
         A_{n,12,21} = \vec{p}_{1^{n-2}} \cdot p_2.
    \end{equation}
    Applying the path and classical Murnaghan--Nakayama rules gives the Schur expansions
    \begin{equation}
        \label{eq:inv-four}
        \begin{cases}
        \vec{p}_{2^2 1^{n-4}} = 2! \cdot (n-4)! \times  \left( {n -2 \choose 2} s_n - (n-3) \cdot s_{n-1,1} + s_{n-2,2} \right), \\
        \vec{p}_{31^{n-3}} = (n-3)! \times \left( (n-2)  s_n - s_{n-1,1} - s_{n-2,2} + s_{n-2,1,1} \right), \\
        \vec{p}_{21^{n-3}} \cdot p_1 = (n-3)! \times ((n-2) s_n + (n-3) s_{n-1,1} - s_{n-2,2} - s_{n-2,1,1}) \\
        \vec{p}_{1^{n-2}} \cdot p_2 = (n-2)! \times (s_n + s_{n-2,2} - s_{n-2,1,1}).
        \end{cases}
    \end{equation}
    Now~\eqref{eq:inversion} follows by combining \eqref{eq:maj-example-two}, \eqref{eq:inv-two}, \eqref{eq:inv-three}, and \eqref{eq:inv-four}.

A similar, but more involved, analysis may be applied to  the second moment $\maj^2: \symm_n \to \CC$ of the major index statistic. The resulting symmetric function is as follows.
\begin{align}
\nonumber
\ch_n \, R \, \maj^2 = & 
\frac{9 n^4 - 14 n^3 + 15 n^2 - 10 n}{144} \cdot s_n 
+ \frac{-3 n^2 + 3n + 8}{12} s_{n-1,1} \\
\label{eq:maj-example-seven}
&+ \frac{7}{6} s_{n-2,2} + 
\frac{-3 n^2 + 3n + 8}{12} s_{n-2,1,1}  
+ \frac{7}{6} s_{n-2,2,1} + \frac{1}{2} s_{n-3,1^3} \\
\nonumber
&+ \frac{1}{2} s_{n-4,2,2} + \frac{1}{2} s_{n-4,1^4}.
\end{align}
Using Theorem~\ref{t:character-polynomial} with data from Figure 3 in~\cite{HR}, we then have
\begin{equation}
    R\, \maj = \frac{n(n-1)}{4} - \frac{1}{4}\cdot m_1^2 + \frac{1}{2} \cdot m_2 + \frac{1}{4} \cdot m_1
\end{equation}
 and
\begin{align}
\nonumber
    R \, \maj^2 = &
    \frac{1}{16} \cdot m_1^4 - 
    \frac{1}{8} \cdot m_1^2n^2 
    + \frac{1}{16} \cdot n^4 - \frac{11}{72} \cdot m_1^3 - 
    \frac{1}{4} \cdot m_1^2 m_2 + \frac{1}{8} \cdot m_1^2 n + \frac{1}{8} \cdot m_1 n^2 
    \\&+ 
    \frac{1}{4} \cdot m_2 n^2  - 
    \frac{7}{72} \cdot n^3 + 
    \frac{1}{48} \cdot m_1^2 + 
    \frac{1}{4} \cdot m_1 m_2 + 
    \frac{3}{4} \cdot m_2^2 - \frac{1}{8} \cdot m_1 n - 
    \frac{1}{4} \cdot m_2 n 
    \label{eq:maj-example-eight}
    \\&+ 
    \frac{5}{48} \cdot n^2 + 
    \frac{5}{72} \cdot m_1 - \frac{3}{4} \cdot m_2 - 
    \frac{2}{3} \cdot m_3 - \frac{1}{2} \cdot m_4 - 
    \frac{5}{72} \cdot n
    \nonumber.
\end{align}
Together, these equations yield the variance of $\maj$ on a given cycle type:
\begin{align}
    R \, \maj^2 - \left( R \, \maj \right)^2 =
    &-\frac{1}{36} \cdot m_1^3 +
    \frac{1}{36} \cdot n^3 - \frac{1}{24} \cdot m_1^2 + 
    \frac{1}{2} \cdot m_2^2 + \frac{1}{24} \cdot n^2 
    \nonumber
    \\ &+ 
    \frac{5}{72} \cdot m_1 - \frac{3}{4} \cdot m_2 - 
    \frac{2}{3} \cdot m_3 -
    \frac{1}{2} \cdot m_4 - \frac{5}{72} \cdot n.
    \label{eq:maj-example-nine}
\end{align}
The leading term by degree of the variance 
is therefore $(n^3 - m_1^3)/36$.
This variance computation is, to our knowledge, previously unknown.
\end{example}

\section{Concluding Remarks}
\label{sec:Conclusion}

We present some further directions of inquiry.

\subsection{Parabolic subgroups}

If $\mu = (\mu_1, \dots, \mu_k) \vdash n$, the {\em parabolic subgroup} $\symm_\mu \subseteq \symm_n$ is $\symm_\mu := \symm_{\mu_1} \times \cdots \times \symm_{\mu_k} \subseteq \symm_n$. For any set $X \subseteq \symm_n$ of permutations, we write $[X]_+ \in \CC[\symm_n]$ for the group algebra element $[X]_+ := \sum_{w \in X} w$.

Let $(I,J) \in \symm_{n,k}$ be a partial permutation. As in the proof of Theorem~\ref{thm:aw-image}, there exist permutations $u_1,u_2 \in \symm_n$ such that $[I,J] = u_1 \cdot [\symm_{n-k}]_+ \cdot u_2$ in the group algebra $\CC[\symm_n]$. For $\lambda \vdash n$, by cyclic invariance of trace we have
\[ \chi^\lambda([I,J]) = \chi^\lambda \left( u_1 \cdot [\symm_{n-k}]_+ \cdot  u_2 \right) = \chi^\lambda \left( w \cdot  [\symm_{n-k}]_+ \right)\]
where $w := u_2 u_1$. Corollary~\ref{cor:trace-interpretation} is a combinatorial rule for evaluating this trace. We can ask for analogous rules for parabolic subgroups more general than $\symm_{n-k}$.

\begin{problem}
    \label{prob:other-parabolic subgroups}
    Let $\lambda,\mu \vdash n$ and let $w \in \symm_n$. Determine a combinatorial rule for evaluating 
    \[ \chi^\lambda \left( w \cdot [\symm_\mu]_+ \right) = \sum_{v \, \in \, \symm_\mu} \chi^\lambda(wv).\]
\end{problem}

\subsection{Unitary groups}

We can also ask for a continuous analog of Corollary~\ref{cor:trace-interpretation} in the context of unitary groups. For $n \geq 0$, let $U(n)$ be the group of $n \times n$ complex unitary matrices. The irreducible polynomial characters $\chi^\lambda: U(n) \to \CC$ are indexed by partitions $\lambda$ with $\ell(\lambda) \leq n$. If $A \in U(n)$ has eigenvalues $x_1, \dots, x_n \in \CC$, then $\chi^\lambda(A) = s_\lambda(x_1,\dots,x_n)$.  This realizes $\chi^\lambda$ as an explicit continuous function $U(n) \to \CC$.

For $r \leq n$, we have a natural embedding $U(r) \subseteq U(n)$ of unitary groups which sends a matrix $B \in U(r)$ to the matrix direct sum $B \oplus I_{n-r}$ where $I_{n-r}$ is the size $n-r$ identity matrix. Since $U(r)$ is a compact Lie group, it admits a unique Haar measure, and any continuous function on $U(r)$ is integrable against this measure.

\begin{problem}
    \label{prob:unitary-analogue}
    Let $r \leq n$, let $\lambda$ be a partition with $\ell(\lambda) \leq n$, and consider the embedding of unitary groups $U(r) \subseteq U(n)$.  For a fixed $A \in U(n)$, evaluate the integral
    \[\int_{B \, \in \, U(r)} \chi^\lambda(AB)\]
    with respect to Haar measure on $U(r)$.
\end{problem}

\subsection{Coefficients of $s_{\lambda[n]}$}
 For any real-valued $f: \symm_n \rightarrow \RR$ has Schur expansion
 \begin{equation*}
 \ch_n(Rf) = \sum_{\lambda \vdash n} c_{\lambda}(n) \cdot s_{\lambda}.
 \end{equation*}
Proposition~7.5 in~\cite{HR} shows $c_{(n)}$ is the expected value of $f$ on $\symm_n$.
 
 \begin{problem}
 \label{other-values-problem}
Interpret the other coefficients $c_{\lambda}$ appearing in the Schur expansion of $\ch_n(Rf)$.
 \end{problem}
 
For $v \in \symm_k$, let $N_v$ be the number of occurences of $v$ as a subpattern.
Gaetz and Pierson give an explicit formula for the coefficient of $s_{(n-1,1)}$ in $\ch_n(R \, N_v)$~\cite[Prop. 4.1]{GP}.
Additionally, for the special case where $v = 12\dots k$, they show the coefficients of $s_{(n-2,1,1)}$ and $s_{(n-2,2)}$ are non-negative.
This is partial progress towards their conjecture:

\begin{conjecture} \em{(\cite[Conj. 1.4]{GP})}
\label{GaetzPierson}
	For $\lambda$ a fixed partition, the coefficient of $s_{\lambda[n]}$ in $\ch_n(R \, N_{12\dots k})$ is non-negative.
\end{conjecture}

Recently, Iskander made significant progress on Conjecture~\ref{GaetzPierson}, demonstrating that the coefficient of $s_{\lambda[n]}$ for $\ch_n(R \, N_{12\dots k})$ can be viewed as a bivariate polynomial in $n$ and $k$~\cite{Iskander}.
He has verified the conjecture up for $|\lambda| \leq 8$ and disproved a stronger conjecture that such coefficients are real rooted.


\section{Acknowledgements}
\label{sec:Acknowledgements}

The authors are grateful to Michael Coopman, Mohammed Slim
Kammoun, Gene Kim, Toby Johnson, Kevin Liu, Jasper Liu, Arnaud Marsiglietti, Jon Novak, James Pascoe, Bruce Sagan, John Stembridge, and Yan Zhuang for helpful conversations. The authors thank Valentin F{\'e}ray, Christian Gaetz, and Dan Rockmore for help with references and Tony Mendes for sharing his ribbon
drawing code. We are especially grateful to Eric Ramos, who helped us conceptualize the project at its inception. Z. Hamaker was partially supported by NSF Grant
DMS-2054423. B. Rhoades was partially supported by NSF Grant DMS-2246846.

\end{document}